\documentclass[aos]{imsart}

\RequirePackage{amsthm,amsmath,amsfonts,amssymb}
\usepackage[utf8]{inputenc}
\usepackage[english]{babel}
\usepackage{verbatim}
\usepackage[numbers]{natbib}
\usepackage{enumitem}
\usepackage{multibib}

\newcites{sec}{References}

\usepackage[colorlinks,citecolor=blue,urlcolor=blue,linkcolor =blue,hypertexnames=false]{hyperref}

\usepackage{graphicx}
\usepackage{subcaption}

\theoremstyle{plain}
\newtheorem{theorem}{Theorem}
\newtheorem{lemma}[theorem]{Lemma}
\newtheorem{corollary}[theorem]{Corollary}
\newtheorem{proposition}[theorem]{Proposition}

\theoremstyle{remark}

\newtheorem{example}[theorem]{Example}

\numberwithin{equation}{section}
\numberwithin{theorem}{section}

\newtheorem*{assumption*}{\assumptionnumber}
\providecommand{\assumptionnumber}{}

\makeatletter
\newenvironment{myassumption}[1]
 {%
  \renewcommand{\assumptionnumber}{Assumption #1}%
  \begin{assumption*}%
  \protected@edef\@currentlabel{#1}%
 }
 {%
  \end{assumption*}
 }
 \makeatother

\newcommand*\Norml{\mathop{}\!\mathbin\lVert}
\newcommand*\Normr{\mathop{}\!\mathbin\rVert}

\newcommand*\at{\mathop{}\!\mathbin a_{\vartheta}}

\renewcommand{\sc}[2]{\langle#1,#2\rangle}
\newcommand{\norm}[1]{\lVert#1\rVert}
\let\P\relax
\DeclareMathOperator{\P}{{\mathbb P}}
\DeclareMathOperator{\E}{{\mathbb E}}
\DeclareMathOperator{\R}{{\mathbb R}}
\DeclareMathOperator{\N}{{\mathbb N}}
\renewcommand{\theta}{\vartheta}    
\renewcommand{\epsilon}{\varepsilon}

\numberwithin{equation}{section}
\numberwithin{theorem}{section}

\allowdisplaybreaks[4]

\begin{document}

\begin{frontmatter}
\title{Optimal parameter estimation for linear SPDEs from multiple measurements}
\runtitle{Optimal parameter estimation for linear SPDEs}

\begin{aug}
%
%
\author[A]{\fnms{Randolf}~\snm{Altmeyer}\ead[label=e1]{ra591@cam.ac.uk}},
\author[B]{\fnms{Anton}~\snm{Tiepner}\ead[label=e2]{tiepner@math.au.dk}}
 \and
\author[C]{\fnms{Martin}~\snm{Wahl}\ead[label=e3]{martin.wahl@math.uni-bielefeld.de}}
%
%
\address[A]{Department of Pure Mathematics and Mathematical Statistics, University of Cambridge\printead[presep={,\ }]{e1}}
\address[B]{Department of Mathematics, Aarhus University\printead[presep={,\ }]{e2}}
\address[C]{Fakult\"{a}t f\"{u}r Mathematik, Universit\"{a}t Bielefeld\printead[presep={,\ }]{e3}}
\end{aug}

\begin{abstract}
The coefficients in a second order parabolic linear stochastic partial differential equation (SPDE) are estimated from multiple spatially localised measurements. Assuming that the spatial resolution tends to zero and the number of measurements is non-decreasing, the rate of convergence for each coefficient depends on its differential order and is faster for higher order coefficients. Based on an explicit analysis of the reproducing kernel Hilbert space of a general stochastic evolution equation, a Gaussian lower bound scheme is introduced. As a result, minimax optimality of the rates as well as sufficient and necessary conditions for consistent estimation are established.
\end{abstract}



\end{frontmatter}


	\section{Introduction}
    Stochastic partial differential equations (SPDEs) form a flexible class of models for space-time data. They combine phenomena such as diffusion and transport that occur naturally in many processes, but also include random forcing terms, which may arise from microscopic scaling limits or account for model uncertainty. Quantifying the size of these different effects is an important step in model validation. 
    
    Suppose that $X=(X(t))_{0\leq t\leq T}$ solves the linear parabolic SPDE 
    \begin{equation}
		\label{eq:SPDE}
			dX(t) = A_\theta X(t)dt+ dW(t), \quad 0\leq t\leq T,
	\end{equation}
	on an open, bounded and smooth domain $\Lambda\subset\mathbb{R}^d$ with some initial value $X_0$, a space-time white noise $dW$ and a second order elliptic operator
	\begin{align}
	    A_{\theta} = \sum_{i=1}^p \theta_i A_{i} + A_0\label{eq:generalA}
	\end{align}
    satisfying zero Dirichlet boundary conditions. The $A_i$ are known differential operators of differential order $n_i\in \{0,1,2\}$ and we aim at recovering the unknown parameter $\theta\in\R^p$. A prototypical example is 
	\begin{align}\label{eq:canonical:model}
        A_{\theta}=\theta_1\Delta+\theta_2(\nabla\cdot b)+\theta_3,\qquad \theta\in(0,\infty)\times\R\times(-\infty,0],
    \end{align}
    with diffusivity, transport and reaction coefficients $\theta_1$, $\theta_2$, $\theta_3$ in front of the Laplace operator $\Delta$ and the divergence operator $\nabla \cdot$ such that $n_1=2$, $n_2=1$, $n_3=0$, with a known unit velocity vector $b\in\R^d$. The general form in \eqref{eq:generalA} allows for wide range of models affected by a mixture of different, possibly anisotropic, mechanisms.  Equations such as \eqref{eq:SPDE} are also called stochastic advection–diffusion equations and often serve as building blocks for more complex models, with applications in different areas such as neuroscience \cite{ Tuckwell2013, sauer_analysis_2016, walsh_stochastic_1981}, biology \cite{altmeyer_parameter_2020, alonso2018modeling}, spatial statistics \cite{sigrist_stochastic_2015, liu_statistical_2021} and data assimilation \cite{llopis2018particle}.  For concrete examples of \eqref{eq:generalA} with a mixture of known and unknown model coefficients from fluid dynamics and engineering see \cite{karalashvili2011, luce2013, catania2006}.

     While the estimation of a scalar parameter in front of the highest order operator $A_i$ is well studied in the literature \cite{huebner_asymptotic_1995, kriz_central_2018, cialenco_drift_2019, cialenco2021statistical,  gugushvili2020bayesian}, there is little known about estimating the lower order coefficients or the full multivariate parameter $\theta$. Relying on discrete space-time observations $X(t_k,x_j)$ in case of $\eqref{eq:canonical:model}$ and in dimension $d=1$, \cite{bibinger_volatility_2020, hildebrandt_parameter_2019, tonaki2022parameter} have analysed power variations and contrast estimators. For two parameters in front of operators $A_1$ and $A_2$, \cite{lototsky_parameter_2003} computed the maximum likelihood estimator from $M$ spectral measurements $(\sc{X(t)}{e_j})_{0\leq t\leq T}$, $j=1,\dots,M$, where the $e_j$ are the eigenfunctions of $A_{\theta}$ and $\sc{\cdot}{\cdot}$ is the inner product on $L^2(\Lambda)$. This leads as $M\rightarrow\infty$ to rates of convergence depending on the differential order of the operators $A_1$, $A_2$, but is restricted to domains and diagonalisable operators with known $e_j$, independent of $\theta$. In particular, in the spectral approach there is no known estimator for the transport coefficient $\theta_2$ in \eqref{eq:canonical:model}. Estimators for nonlinearities or noise specifications are studied e.g. by \cite{chong_high-frequency_2020, hildebrandt2021nonparametric, gaudlitz2022estimation, benth2022weak}.
     
     In contrast, we construct an estimator $\hat{\theta}_{\delta}$ of $\theta$ on general domains and with arbitrary possibly anisotropic $A_{\theta}$ from local measurement processes $X_{\delta,k}=(\sc{X(t)}{K_{\delta,x_k}})_{0\leq t\leq T}$, $X^{A_i}_{\delta,k}=(\sc{X(t)}{A^*_i K_{\delta,x_k}})_{0\leq t\leq T}$ for $i=0,\dots,p$ and locations $x_1,\dots,x_M\in\Lambda$. The $K_{\delta,x_k}$, also known as \emph{point spread functions} in optical systems \cite{aspelmeier_modern_2015,backer_extending_2014}, are compactly supported functions on subsets of $\Lambda$ with radius $\delta>0$ and centred at the $x_k$. They are part of the observation scheme and describe the physical limitation that point sources $X(t_k,x_j)$ can only be measured up to a convolution with the point spread function. Local measurements were introduced in a recent contribution by \cite{altmeyer_nonparametric_2020} to demonstrate that a nonparametric diffusivity can already be identified at $x_k$ from the spatially localised information $X_{\delta,k}$ as $\delta\rightarrow 0$ with $T>0$ fixed. See \cite{altmeyer_parameterSemi_2020} for robustness to semilinear perturbations and different noise configurations besides space-time white noise. For more details on practical aspects of local measurements, as well as a concrete example from cell biology \cite{altmeyer_parameter_2020}, see Section \ref{sec:Practical_aspects} below.

   Let us briefly describe our main contributions. Our first result extends the augmented MLE $\hat{\theta}_{\delta}$ and the CLT of \cite{altmeyer_nonparametric_2020} to $M= M(\delta)$ measurements and joint asymptotic normality of
    \begin{equation*}
		  (M^{1/2}\delta^{1-n_i}(\hat{\theta}_{\delta,i}-\theta_{i}))_{i=1}^p,\quad \delta\rightarrow 0.
	\end{equation*}
    This yields the convergence rates $M^{1/2}\delta^{1-n_i}$ for $\vartheta_i$, with the fastest rate for diffusivity terms with $n_i=2$ and the slowest rate for reaction terms with $n_i=0$. We then turn to the  problem of establishing optimality of these rates in case of \eqref{eq:canonical:model}. We compute the reproducing kernel Hilbert space (RKHS) of the Gaussian measures induced by the laws of $X$ and of the local measurements. From this we derive minimax lower bounds, implying that the rates in the CLT are indeed optimal, and provide conditions under which consistent estimation is impossible. 
    Combined with our CLT we deduce for general point spread functions $K_{\delta,x_k}$ with non-intersecting supports that consistent estimation is possible if and only if $M^{1/2}\delta^{1-n_i}\rightarrow\infty$. Since $M$ is at most of order $\delta^{-d}$, reaction terms cannot be estimated when $d=1$.
    
    Conceptually, spectral measurements can be obtained approximately from local measurements on a dense grid over the entire domain by a discrete Fourier transform and we recover the rates of convergence of \cite{huebner_asymptotic_1995} by taking $M$ of maximal order $\delta^{-d}$. 
	
    The information geometry underlying local measurements is complex due to the non-linear dependence of the solution $X$ on $\theta$ (cf. \eqref{eq:weakSolution}) and the non-Markovian dynamics of the processes $X_{\delta,k}$, $X^{A_i}_{\delta,k}$. This leads to a non-explicit likelihood function, making standard MLE-based estimation and optimality results for continuously observed diffusion processes \cite{kutoyants_statistical_2013} non-applicable in this context. Instead, we introduce a novel lower bound scheme for Gaussian measures, which exploits that the Hellinger distance of their laws can be related to properties of their RKHS. This is different from the lower bound approach of \cite{altmeyer_nonparametric_2020} for $M=1$ and paves the way to rigorous lower bounds for each coefficient and an arbitrary number of measurements. One of our key results states that the RKHS of the Gaussian measure induced by $X$ with $A_\vartheta=\Delta$ consists of all absolutely continuous $h\in L^2([0,T];L^2(\Lambda))$ with $\Delta h,h'\in L^2([0,T];L^2(\Lambda))$ and its squared RKHS norm equals
    \begin{align*}
	     \quad & \|\Delta h\|_{L^2([0,T];L^2(\Lambda))}^2 +\|h'\|_{L^2([0,T];L^2(\Lambda))}^2 +\|(-\Delta)^{1/2} h(0)\|^2_{L^2(\Lambda)}+\|(-\Delta)^{1/2} h(T)\|^2_{L^2(\Lambda)}.
	\end{align*}
	This surprisingly simple formula generalises the finite-dimensional Ornstein-Uhlenbeck case \cite{MR3024389}, and provides a route to obtain the RKHS of local measurements as linear transformations of $X$. To the best of our knowledge the RKHS of $X$ has not been stated before in the literature, and may be of independent interest, e.g. in constructing Bayesian procedures with Gaussian process priors, cf. \cite{van_der_vaart_rates_2008}.
	
    The paper is organised as follows. Section \ref{sec:main} deals with the local measurement scheme, the construction of our estimator and the CLT. Section \ref{Sec:RKHS:SPDE} addresses the RKHS of $X$ and of the local measurements, while Section \ref{sec:optimality} presents the lower bounds for the rates established in the CLT. Section \ref{sec:examples} covers model examples, the boundary case for estimating zero order terms in $d=2$ and some practical aspects. All proofs are deferred to Section \ref{sec:proofs} and to the Appendix \ref{app:additional:proofs}. 
    
    \subsection*{Basic notation}
	
	Throughout the paper, we work on a filtered probability space $(\Omega, \mathcal{F},(\mathcal{F}_t)_{0\leq t\leq T},\P)$. We write $a\lesssim b$ if $a\leq Cb$ for a universal constant $C$ not depending on $\delta$, but possibly depending on other quantities such as $T$ and $\Lambda$. Unless stated otherwise, all limits are understood as  $\delta\rightarrow 0$ with non-decreasing $M=M(\delta)$ possibly depending on~$\delta$. 
	
	The Euclidean inner product and distance of two vectors $a,b\in\R^p$ is denoted by $a\cdot b$ and $|b-a|$, $I_{p\times p}$ is the identity matrix in $\R^{p\times p}$. We write $\|\cdot\|_{\operatorname{op}}$ for the operator norm of a matrix. For an open set $U\subset\R^d$ and $p\geq 1$, $L^p(U)$ is the usual $L^p$-space with norm $\norm{\cdot}_{L^p(U)}$ and the inner product on $L^2(U)$ is denoted $\sc{\cdot}{\cdot}_{L^2(U)}$. We write $\sc{\cdot}{\cdot}=\sc{\cdot}{\cdot}_{L^2(\Lambda)}$, $\norm{\cdot}=\norm{\cdot}_{L^2(\Lambda)}$. Let $H^k(U)$ denote the usual Sobolev spaces and let $H_0^1(U)$ be the completion of the space of smooth compactly supported functions $C_c^{\infty}(U)$  relative to the $H^1(U)$-norm. 
 
    We write $D_i$, $D_{ij}$ for partial derivatives. The gradient and Laplace operators are $\nabla$, $\Delta=\sum_{i=1}^d D_{ii}$. The divergence of a $d$-dimensional vector field $v$ is $\nabla\cdot v=\sum_{i=1}^d D_i v_i$. The Laplace operator $\Delta$ will be considered with domain $H_0^1(\Lambda)\cap H^2(\Lambda)$, while with domain $H^2(\R^d)$ it will be denoted by $\Delta_0$.
	
	For a Hilbert space $\mathcal{H}$, the space $L^2([0,T];\mathcal{H})$ consists of all measurable functions $h:[0,T]\rightarrow \mathcal{H}$ with $\int_0^T\|h(t)\|_{\mathcal{H}}^2dt<\infty$. We write $\norm{T}_{\operatorname{HS}(\mathcal{H}_1,\mathcal{H}_2)}$ for the Hilbert-Schmidt norm of a linear operator $T:\mathcal{H}_1\rightarrow \mathcal{H}_2$ between two Hilbert spaces $\mathcal{H}_1,\mathcal{H}_2$.
    
    \section{Joint parameter estimation}\label{sec:main}
	
    \subsection{Setup}\label{sec:setup}

    Let $\theta\in\Theta\subset\R^p$ be an unknown parameter. For $i=0,\dots,p$, suppose that the operators in \eqref{eq:generalA} are of the form $A_i = \nabla \cdot a^{(i)} \nabla +  \nabla\cdot b^{(i)} + c^{(i)}$ for symmetric $a^{(i)}\in\R^{d\times d}$, $b^{(i)}\in\R^d$ and $c^{(i)}\in\R$, where for each $i=1,\dots,p$ only one of the coefficients $a^{(i)}$, $b^{(i)}$, $c^{(i)}$ is non-vanishing. For each $A_i$, the formal adjoint is $A_i^* = \nabla \cdot a^{(i)} \nabla - \nabla\cdot b^{(i)} + c^{(i)}$, and its differential order $n_i= \operatorname{ord}(A_i)\in\{0,1,2\}$ is the number of non-vanishing derivatives. With $a_{\theta}=\sum_{i=1}^p \theta_i a^{(i)} + a^{(0)}$, $b_{\theta}=\sum_{i=1}^p \theta_i b^{(i)} + b^{(0)}$ and $c_{\theta}=\sum_{i=1}^p \theta_i c^{(i)} + c^{(0)}$, \eqref{eq:generalA} gives $$A_{\theta} = \nabla \cdot a_{\theta} \nabla + \nabla\cdot b_{\theta} + c_{\theta}.$$We suppose that $a_{\theta}$ is positive definite for all $\theta\in\Theta$. Then $A_{\theta}$ is a strongly elliptic operator and generates with domain $H^1_0(\Lambda)\cap H^2(\Lambda)$ an analytic semigroup $(S_{\theta}(t))_{t\geq 0}$ on $L^2(\Lambda)$ \cite{lunardi_analytic_1995}. Considered with the same domain, the adjoint $A_{\theta}^*=\sum_{i=1}^p \theta_i A_i^*+A_0^*$ generates the adjoint semigroup $(S_{\theta}^*(t))_{t\geq 0}$ \cite[Section 2.5.3]{yagi_abstract_2009}.
    
    With an $\mathcal{F}_0$-measurable initial value $X_0$ and a cylindrical Wiener process $W$ on $L^2(\Lambda)$ define a process $X=(X(t))_{0\leq t\leq T}$ by
	\begin{align}
	    X(t) = S_{\theta}(t)X_0 + \int_0^t S_{\theta}(t-t')dW(t'),\quad 0\leq t\leq T.\label{eq:weakSolution}
	\end{align}
	Due to the low spatial regularity of $W$ this process is understood as a random element with values in $L^2(\Lambda)\subset\mathcal{H}_1$ almost surely for a larger Hilbert space $\mathcal{H}_1$ with an embedding $\iota:L^2(\Lambda)\rightarrow\mathcal{H}_1$ such that $\int_0^t\norm{\iota S_{\theta}(t')}_{\operatorname{HS}(L^2(\Lambda),\mathcal{H}_1)}^2dt'<\infty$ \cite[Remark 6.6]{hairer_introduction_2009}. Such an embedding always exists. For example, $\mathcal{H}_1$ can be realised as a negative Sobolev space (see Section \ref{sec:RKHS:computations} below). Let $\mathcal{H}_1'$ denote the dual space of $\mathcal{H}_1$ with the associated dual pairing $\sc{\cdot}{\cdot}_{\mathcal{H}_1\times\mathcal{H}_1'}$. Let $(e_k)_{k\geq 1}$ be an orthonormal basis of $L^2(\Lambda)$ and let $\beta_k$ be independent scalar Brownian motions. Then, realising the Wiener process as $W=\sum_{k\geq 1}e_k\beta_k$, we find for all $z\in\mathcal{H}_1'\subset L^2(\Lambda)$, $0\leq t\leq T$, that (see, e.g.~\cite[Lemma 2.4.1 and Proposition 2.4.5]{liu_stochastic_2015}) 
	\begin{align*}
        \sc{X(t)-S_{\theta}(t)X_0}{z}_{\mathcal{H}_1\times\mathcal{H}_1'} 
        &= \sum_{k\geq 1} \int_{0}^t\sc{S_{\theta}(t-t')e_k}{z}_{\mathcal{H}_1\times\mathcal{H}_1'}d\beta_k(t')\\
        &= \int_{0}^t\sc{S^*_{\theta}(t-t')z}{dW(t')}.
    \end{align*}
    According to \cite[Proposition 2.1]{altmeyer_nonparametric_2020} and \cite[Lemma 2.4.2]{liu_stochastic_2015} this allows us to extend the dual pairings $\sc{X(t)}{z}_{\mathcal{H}_1\times\mathcal{H}_1'}$ to a real-valued Gaussian process $(\sc{X(t)}{z})_{0\leq t\leq T, z\in L^2(\Lambda)}$ by  
    \begin{align}
        &\sc{X(t)}{z} = \sc{S_{\theta}(t)X_0}{z} + \int_{0}^t\sc{S^*_{\theta}(t-t')z}{dW(t')}\label{eq:X_GP}
    \end{align}
    (the notation $\sc{X(t)}{z}$ is used for convenience and indicates that the process does not depend on the embedding space $\mathcal{H}_1$). This process solves the SPDE \eqref{eq:SPDE} in the sense that for  all $z\in H^1_0(\Lambda)\cap H^2(\Lambda)$ and $0\leq t\leq T$
    \begin{align}
	    \sc{X(t)}{z} = \sc{X_0}{z} +\int_0^t\sc{X(t')}{A_{\theta}^*z}dt'+\sc{W(t)}{z}, \label{eq:Ito}
	\end{align}
	where $\sc{W(t)}{z}/\norm{z}_{L^2(\Lambda)}$ is a scalar Brownian motion.
	
	\subsection{Local measurements, construction of the estimator}
	
	Introduce for $z\in L^2(\R^d)$ the scale and shift operation \begin{align}
		z_{\delta,x}(y) & =\delta^{-d/2}z(\delta^{-1}(y-x)),\quad x,y\in\Lambda,\quad \delta>0.\label{eq:scaling}
	\end{align}
	Suppose that $K\in H^2(\R^d)$ is an (unscaled) point spread function with compact support (see Section \ref{sec:examples} for concrete examples). Consider locations $x_1,\dots,x_M\in\Lambda$, $M\in\N$, and a resolution level $\delta>0$, which is small enough to ensure that the point spread functions $K_{\delta,x_k}$ are supported on $\Lambda$. Local measurements of $X$ at the locations $x_1,\dots,x_M$ at resolution $\delta$ correspond to the continuously observed processes $X_{\delta}, X^{A_0}_\delta\in L^2([0,T];\R^M)$, $X^A_{\delta}\in L^2([0,T];\R^{p\times M})$, where  for $i=1,\dots,p$, $k=1,\dots,M$
	\begin{align*}
	    (X_{\delta})_k & = X_{\delta,k} = (\sc{X(t)}{K_{\delta,x_k}})_{0\leq t\leq T},\\
	    (X^{A_0}_{\delta})_{k}&=X_{\delta,k}^{A_0}=(\sc{X(t)}{A^*_0 K_{\delta,x_k}})_{0\leq t\leq T},\\
	    (X^A_{\delta})_{ik} & = X^{A_i}_{\delta,k} =  (\sc{X(t)}{A^*_i K_{\delta,x_k}})_{0\leq t\leq T}.
	\end{align*}
    According to \eqref{eq:Ito}, every local measurement is an Itô process
	\begin{equation}
		\label{eq:SPDEk}
		dX_{\delta,k}(t)=\bigg(\sum_{i=1}^p \theta_i X^{A_i}_{\delta,k}(t) + X^{A_0}_{\delta,k}(t)\bigg) dt+\norm{ K}_{L^2(\R^d)} dW_k(t)
	\end{equation}
    with initial values $X_{\delta,k}(0)=\sc{X_0}{K_{\delta,x_k}}$ and scalar Brownian motions $W_k(t)=\sc{W(t)}{K_{\delta,x_k}}/\norm{ K}_{L^2(\R^d)}$. 
	 
	It should be noted that neither \eqref{eq:SPDEk} nor the system of equations augmented with $X^A_{\delta}$, $X^{A_0}_{\delta}$ are Markov processes, because the time evolution at $x_k$ depends on the spatial structure of the whole process $X$, and not only of $X_{\delta}$. This is due to the infinite speed of propagation in space by $S_{\theta}(t)$. This also means the processes $X_{\delta,k}$ are \emph{not} independent, even if the driving noise processes $W_k$ are, e.g., due to non-overlapping supports of the $K_{\delta,x_k}$ as will be assumed below. Therefore, standard results for estimating the parameters $\theta_i$ from continuously observed diffusion processes by the maximum likelihood estimator (e.g., \cite{kutoyants_statistical_2013}) do not apply here. Instead, a general Girsanov theorem for multivariate It{\^o} processes, cf. \cite[Section 7.6]{liptser_statistics_2001}, yields after ignoring conditional expectations, the initial value and possible correlations between measurements the modified log-likelihood function
	\begin{equation}
		\begin{split}
		\ell_{\delta}(\theta)=  \norm{K}^{-2}_{L^2(\R^d)}\sum_{k=1}^{M}\bigg(\int_{0}^{T} \bigg(\sum_{i=1}^p \theta_i X^{A_i}_{\delta,k}(t)+X^{A_0}_{\delta,k}(t)\bigg)dX_{\delta,k}(t)\nonumber\\
			-\frac{1}{2}\int_{0}^{T} \bigg(\sum_{i=1}^p \theta_i X^{A_i}_{\delta,k}(t)+X^{A_0}_{\delta,k}(t)\bigg)^2\bigg)dt.\label{eq:logLik}
		\end{split}
	\end{equation}
	Maximising $\ell_\delta(\theta)$ with respect to $\theta$ leads to the estimator
	\begin{align}
			\hat{\theta}_{\delta} & = \mathcal{I}_{\delta}^{-1} \sum_{k=1}^{M} \left(\int_{0}^{T}X^A_{\delta,k}(t)dX_{\delta,k}(t)-\int_{0}^{T}X^A_{\delta,k}(t) X^{A_0}_{\delta,k}(t)dt\right),\label{eq:augMLE}
	\end{align}
	which we call \emph{augmented MLE} generalising
    \cite[Section 4.1]{altmeyer_parameter_2020}, with \emph{observed Fisher information}
	\begin{equation}\label{eq:Fisher} 
		\mathcal{I}_{\delta}=\sum_{k=1}^{M}\int_{0}^{T}X^A_{\delta,k}(t) X^A_{\delta,k}(t)^\top dt.
	\end{equation}
	
	\subsection{A central limit theorem}
	
	We show now that the augmented MLE $\hat{\theta}_{\delta}$ satisfies a CLT as $\delta\rightarrow0$. Replacing $dX_{\delta,k}(t)$ in the definition of the augmented MLE by the right hand side in \eqref{eq:SPDEk} yields the basic decomposition
	\begin{align}
	    \hat{\theta}_{\delta} = \theta + \norm{K}_{L^2(\R^d)}\mathcal{I}_{\delta}^{-1} \mathcal{M}_{\delta}\label{eq:basicError}
	\end{align}
	with the martingale term 
	\begin{align}
	    \mathcal{M}_{\delta} = \sum_{k=1}^{M} \left(\int_{0}^{T}X^A_{\delta,k}(t)dW_{k}(t)\right)\label{eq:martingale}.
	\end{align}
    If the Brownian motions $W_k$ are independent, then the matrix $\mathcal{I}_{\delta}$ corresponds to the quadratic co-variation process of $\mathcal{M}_{\delta}$ and we therefore expect $\mathcal{I}_{\delta}^{-1/2}\mathcal{M}_{\delta}$ to follow approximately a multivariate normal distribution. The rate at which the estimation error in \eqref{eq:basicError} vanishes corresponds to the speed at which the components of the observed Fisher information diverge. Exploiting scaling properties of the underlying semigroup (cf. Lemma \ref{rescaledsemigroup}), we will see that this depends on its action on the point spread functions $K_{\delta,x_k}$. We define a diagonal matrix of scaling coefficients $\rho_{\delta}\in \R^{p\times p}$,
	\begin{align}
	    (\rho_{\delta})_{ii}  = M^{-1/2}\delta^{n_i-1},
	\end{align}
	and make the following additional structural assumptions.
	
	\begin{myassumption}{H}\label{assump:mainAssump}\,
	 \begin{enumerate}[label=(\roman*)]
	    \item The functions $A_i K$ are linearly independent for all $i=1,\dots,p$.
	    \item $n_i > 1-d/2$ for all $i=1,\dots,p$.
	    \item The locations $x_k$, $k=1,\dots,M$, belong to a fixed compact set $\mathcal{J}\subset\Lambda$, which is independent of $\delta$ and $M$. There exists $\delta'>0$ such that $\mathrm{supp}(K_{\delta,x_k})\cap\mathrm{supp}(K_{\delta,x_l})=\emptyset$ for $k\neq l$ and all $\delta\leq\delta'$.
	    \item $\sup_{x\in\mathcal{J}} \int_{0}^{T}\E[\sc{X_0}{S^*_{\theta}(t)A_i^* K_{\delta,x}}^2] dt=o(\delta^{2-2n_i})$ for all $i=1,\dots,p$.
	\end{enumerate}
	\end{myassumption}
	Assumption \ref{assump:mainAssump}(i) guarantees invertibility of the observed Fisher information, for a proof see Section \ref{sec:remainingProofs}. 
	
	\begin{lemma}\label{lem:FisherInvertible}
	    Under Assumption \ref{assump:mainAssump}(i), $\mathcal{I}_{\delta}$ is $\P$-almost surely invertible.
	\end{lemma}
	
	The support condition in Assumption \ref{assump:mainAssump}(iii) is natural in applications, e.g., in microscopy. It guarantees $\sc{K_{\delta,x_k}}{K_{\delta,x_l}}=0$ and thus independence of the Brownian motions $W_k$ in \eqref{eq:SPDEk} as $\delta\rightarrow 0$. It holds for $x_k$, which are separated by a Euclidean distance of at least $C\delta$ for a fixed constant $C$, hence there are at most $M = O(\delta^{-d})$ such locations. 
    The next lemma shows that Assumption \ref{assump:mainAssump}(iv) on the initial value is satisfied in most relevant situations. For a proof see again Section \ref{sec:remainingProofs}.
	
	\begin{lemma}\label{lem:X0}
	    Assumption \ref{assump:mainAssump}(ii) implies Assumption \ref{assump:mainAssump}(iv) for any $X_0\in L^q(\Lambda)$, $q>2$, and if $c_{\theta}\leq 0$ also for the stationary initial condition $X_0=\int_{-\infty}^0 S_{\theta}(-t')dW(t')$.
	\end{lemma}
    We establish now the asymptotic behaviour of the observed Fisher information and a CLT for the augmented MLE as the resolution $\delta$ tends to zero. To this extent, consider the positive operator $-\nabla\cdot a_{\theta}\nabla$ with domain $H^2(\R^d)$. Its spectral calculus induces for each $s\in\R$ the fractional operator $(-\nabla \cdot a_{\theta}\nabla)^s$, which acts in the Fourier domain as the multiplication operator with multiplier $\xi\mapsto (-\xi^\top a_{\theta}\xi)^s$, cf.~\cite{kwasnicki2017ten} or \cite[Chapter VI.5]{EngNag00}. By positive definiteness of $a_{\theta}$, this means $(-\nabla \cdot a_{\theta}\nabla)^s z \in L^2(\R^d)$ as soon as $\xi\mapsto |\xi|^{2s}\mathcal{F}z(\xi)\in L^2(\R^d)$  with the Fourier transform $\mathcal{F}z$. By usual Fourier calculus \cite[Lemma VI.5.4]{EngNag00}, $\mathcal{F}D_jz=i\xi_j\mathcal{F}z$. Together with Assumption \ref{assump:mainAssump}(ii), this means $(-\nabla\cdot a_{\theta}\nabla)^{-1/2}A_i^* K\in L^2(\R^d)$ for all $i=1,\dots,p$. 

	\begin{theorem}
		\label{thm:clt}
		Under Assumption \ref{assump:mainAssump} the matrix $\Sigma_{\theta}\in\R^{p\times p}$ with entries
    	\begin{align*}
    	    (\Sigma_{\theta})_{ij}=(T/2)\sc{(-\nabla\cdot a_{\theta}\nabla)^{-1/2}A_i^* K}{(-\nabla\cdot a_{\theta}\nabla)^{-1/2}A_j^* K}_{L^2(\R^d)}
    	\end{align*}
    	is invertible and $\rho_{\delta} \mathcal{I}_{\delta}\rho_{\delta} \stackrel{\P}\rightarrow \Sigma_{\theta}$ as $\delta\rightarrow 0$. Moreover, the augmented MLE satisfies the CLT
		\begin{equation*}
			(\rho_\delta\mathcal{I}_{\delta}\rho_\delta)^{1/2}\rho_\delta^{-1}(\hat{\theta}_{\delta}-\vartheta)\stackrel{d}{\rightarrow}\mathcal{N}(0, \norm{K}^{2}_{L^2(\R^d)} I_{p\times p}),\quad\delta\rightarrow 0,
		\end{equation*}
		or, equivalently,
		\begin{equation*}
		  (M^{1/2}\delta^{1-n_i}(\hat{\theta}_{\delta,i}-\theta_{i}))_{i=1}^p\stackrel{d}{\rightarrow}\mathcal{N}(0, \norm{K}^{2}_{L^2(\R^d)} \Sigma^{-1}_{\theta}).
		\end{equation*}
    \end{theorem}

    Theorem \ref{thm:clt} shows that parameters $\theta_i$ of an operator $A_i$ with differential order $n_i$ can be estimated at the rate of convergence $M^{1/2}\delta^{1-n_i}$. Consistency requires $M^{1/2}\delta^{1-n_i}\rightarrow \infty$. This excludes reaction terms $\theta_i$ in $d=1$ with $n_i=0$ and $M=O(\delta^{-1})$, but in $d=2$ a logarithmic rate holds for a restricted class of functions $K$, see Proposition \ref{prop: reactiond2}. The asymptotic variances for two parameters $\theta_i$, $\theta_j$ are independent if $A^*_iK$ and $A^*_jK$ are orthogonal in the geometry induced by $\norm{(-\nabla\cdot a_{\theta}\nabla)^{-1/2}\cdot}_{L^2(\R^d)}$. The theorem generalises \cite[Theorem 5.3]{altmeyer_nonparametric_2020} in the parametric case to the anisotropic setting with $M$ measurement locations.

    \section{The RKHS}\label{Sec:RKHS:SPDE}

    In Section \ref{sec:optimality}, we show optimality of the rates of convergence appearing in Theorem \ref{thm:clt}. A crucial ingredient for these lower bound considerations is a good understanding of the reproducing kernel Hilbert space (RKHS) of the Gaussian measure induced by the law of the observations when $A_{\theta}=\Delta$.

	We first derive the RKHS of the stochastic convolution \eqref{eq:weakSolution} in a more general setting. Suppose that $A$ is an (unbounded) negative self-adjoint closed operator on a Hilbert space $(\mathcal{H},\norm{\cdot}_{\mathcal{H}})$ with domain $\mathcal{D}(A)\subset\mathcal{H}$ such that  $Ae_j=-\lambda_j e_j$ for a non-decreasing sequence $(\lambda_j)_{j\geq 1}$ of positive real numbers with $0<\lambda_1\leq \lambda_2\leq \cdots$ and an orthonormal basis $(e_j)_{j\geq 1}$ of $\mathcal{H}$, and such that $A$ generates a strongly continuous semigroup $(S(t))_{t\geq 0}$ on $\mathcal{H}$ \cite{EngNag00}. With a cylindrical Wiener process $W$, consider the stationary stochastic convolution
	\begin{align}
	    X(t)=\int_{-\infty}^t S(t-t')dW(t'),\quad t\geq 0.\label{eq:stochConv}
	\end{align}
	As discussed after \eqref{eq:weakSolution} the process $X=(X(t))_{0\leq t\leq T}$ is understood as a random element with values in $\mathcal{H}\subset\mathcal{H}_1$ almost surely for some larger Hilbert space $\mathcal{H}_1$. 

    In what follows, we use the convention that a RKHS is denoted by the letter $H$. Moreover, we add a subscript to indicate the process which is under consideration. For instance, $H_{X}$ denotes the RKHS of $X$ considered as a Gaussian random variable taking values in the Hilbert space $L^2([0,T];\mathcal{H}_1)$. Since the RKHS of $X$ depends only on its distribution, the RKHS, as well as its norm, in the next theorem are independent of the embedding space $\mathcal{H}_1$ (see, e.g.,~Exercise 2.6.5 in \cite{gine_mathematical_2016}). For the proof and some background on the RKHS of a Gaussian measure see Section \ref{sec:RKHS:computations}.

	\begin{theorem}\label{thm:RKHS:SPDE}
	    Let $(H_{X},\|\cdot\|_{X})$ be the RKHS of the process $X$ in \eqref{eq:stochConv}. Let $T\geq 1$. Then
	    \begin{align*}
	        H_X 
	            & = \{h\in L^2([0,T];\mathcal{H}):h\text{ absolutely continuous, } A h,h'\in L^2([0,T];\mathcal{H})\}
	     \end{align*}
	     and for $h\in H_X$
	     \begin{align*}
	     \quad & \|h\|_{X}^2 
                =\|A h\|_{L^2([0,T];\mathcal{H})}^2 +\|h'\|_{L^2([0,T];\mathcal{H})}^2+\|(-A)^{1/2} h(0)\|^2_{\mathcal{H}}+\|(-A)^{1/2} h(T)\|^2_{\mathcal{H}},
	    \end{align*}
	    as well as 
	    \begin{align*}
	        \|h\|_{X}^2\leq 3\|A h\|_{L^2([0,T];\mathcal{H})}^2 +\|h\|_{L^2([0,T];\mathcal{H})}^2+2\|h'\|_{L^2([0,T];\mathcal{H})}^2.
	    \end{align*}
	\end{theorem}
    Note that $h,h',A h\in L^2([0,T];\mathcal{H})$ implies that the map $t\mapsto \langle Ah(t),h(t)\rangle$ is absolutely continuous (cf.~the proof of \cite[Theorem 5.9.3]{evans_partial_2010} and the proof of Theorem \ref{thm:RKHS:SPDE}), so that the norm $\|\cdot\|_{X}$ is indeed well-defined.
    Theorem \ref{thm:RKHS:SPDE} generalises the result for scalar Ornstein-Uhlenbeck processes to the infinite dimensional process $X$, cf. Lemma \ref{lemma:RKHS_norm_OU_process} below. 
    
    Next, as in \eqref{eq:Ito}, consider the Gaussian process $(\sc{X(t)}{z}_{\mathcal{H}})_{t\geq 0,z\in\mathcal{H}}$, where the `inner product' here corresponds to
    \begin{align*}
        \sc{X(t)}{z}_{\mathcal{H}} = \int_{-\infty}^t\sc{S(t-t')z}{dW(t')}_{\mathcal{H}},
    \end{align*}
    satisfying \eqref{eq:Ito} for $z\in\mathcal{D}(A^*)=\mathcal{D}(A)$ by analogous arguments. We study the RKHS of $(\sc{X(t)}{z}_{\mathcal{H}})_{0\leq t\leq T}$ for finitely many $z$. A first proof considers $z$ from the dual space of $\mathcal{H}_1$. In that case, we can realise $\sc{X}{z}_{\mathcal{H}}$ as a linear transformation of $X$ by a bounded linear map $L$ from $L^2([0,T];\mathcal{H}_1)$ to $L^2([0,T])^M$, and this allows for relating the RKHS of $X$ and $\sc{X}{z}_{\mathcal{H}}$ using Theorem \ref{thm:RKHS:SPDE}. Another proof is presented in Section \ref{app:RKHS:measurements:approximation:argument}, which circumvents this by an approximation argument.

	\begin{theorem}\label{thm:upper:bound:RKHS:norm:M>1}
    For $K_1,\dots,K_M\in \mathcal{D}(A)$ and with $X$ in \eqref{eq:stochConv} consider the process $X_K$ with $X_{K}(t)=(\sc{X(t)}{K_k}_{\mathcal{H}})_{k=1}^M$. Suppose that the Gram matrix $G=(\sc{K_k}{K_l}_{\mathcal{H}})_{1\leq k,l\leq M}$ is non-singular, and let $G_{A}=(\sc{A K_k}{AK_l}_{\mathcal{H}})_{1\leq k,l\leq M}$. Let $T\geq 1$. Then the RKHS $(H_{X_K},\|\cdot\|_{X_{K}})$ of $X_K$ satisfies $H_{X_K}=H^M$, where
    \begin{align*}
       H=\{h\in L^2([0,T]):h\text{ absolutely continuous}, h'\in L^2([0,T])\}
    \end{align*} 
    and for $h=(h_k)_{k=1}^M\in H_{X_K}$
    \begin{align*}
        \norm{h}_{X_K}^2&\leq (3\|G^{-1}\|_{\operatorname{op}}^2\|G_{A} \|_{\operatorname{op}}+\|G^{-1}\|_{\operatorname{op}})\sum_{k=1}^M\norm{h_k}^2_{L^2([0,T])}+2\|G^{-1}\|_{\operatorname{op}}\sum_{k=1}^M\norm{h_k'}^2_{L^2([0,T])}.
    \end{align*}
    \end{theorem}

    Theorem \ref{thm:upper:bound:RKHS:norm:M>1} (and its slight generalisation in \eqref{eq:upper:bound:RKHS:norm:M>1}) can be used to compute the RKHS of quite general observation schemes. In the specific case $A=\Delta$ and local measurements with $K_k=K_{\delta,x_k}$ we obtain the following. 

	\begin{corollary}\label{cor:upper:bound:RKHS:norm:M>1_Laplace}
    Let $(H_{X_\delta},\|\cdot\|_{X_\delta})$ be the RKHS of $X_\delta$ with respect to $A=\Delta$, $K\in H^2(\R^d)$ with $\norm{K}_{L^2(\R^d)}=1$, and points $x_1,\dots,x_M$ such that $\mathrm{supp}(K_{\delta,x_k})\subset\Lambda$ for all $k=1,\dots,M$ and $\mathrm{supp}(K_{\delta,x_k})\cap\mathrm{supp}(K_{\delta,x_l})=\emptyset$ for all $1\leq k\neq l\leq M$. Suppose that $\delta^2\leq \norm{\Delta K}_{L^2(\R^d)}$ and $T\geq 1$. Then $ H_{X_\delta}=H^M$
    and for $h=(h_k)_{k=1}^M\in H_{X_{\delta}}$
    \begin{align*}
       \|h\|_{X_\delta}^2&\leq 4\frac{\norm{\Delta K}^2_{L^2(\R^d)}}{\delta^{4}}\sum_{k=1}^M\|h_k\|_{L^2([0,T])}^2+2\sum_{k=1}^M\|h_k'\|_{L^2([0,T])}^2.
    \end{align*}
    \end{corollary}
    
	Similar results hold for the RKHS of $(X_\delta, X^A_{\delta})$, see Corollary \ref{cor:lower:bound:RKHS:norm:M>1:mult:measurement}.

    \section{Optimality}\label{sec:optimality}

	In this section, we show that the rates of convergence $M^{1/2}\delta^{1-n_i}$ achieved by the augmented MLE for parameters $\theta_i$ with respect to operators $A_i$ of order $n_i=\operatorname{ord}(A_i)$ are indeed optimal and cannot be improved in our general setup. The proof strategy (presented in Section \ref{sec:proof:optimality:result}) relies on a novel lower bound scheme for Gaussian measures by relating the Hellinger distance of their laws to properties of their RKHS. The Gaussian lower bound is then applied to one-dimensional submodels $(\mathbb{P}_{\theta})_{\vartheta\in\Theta_i}$ with $A_{\theta}$ from \eqref{eq:canonical:model} assuming a sufficiently regular kernel function $K$ and a stationary initial condition. 
	
	\begin{myassumption}{L}\label{assu:lowerBound} Suppose that $\P_{\theta}$ corresponds to the law of the stationary solution $X$ to the SPDE \eqref{eq:SPDE} and assume that the following conditions hold:
	     \begin{enumerate}[label=(\roman*)]
	        \item The kernel function satisfies $K=\Delta^2\tilde{K}$ with $\tilde{K}\in C_c^{\infty}(\R^d)$.
	        \item The models are  $A_{\theta}=\theta_1\Delta+\theta_2(\nabla\cdot b)+\theta_3$ for $\theta\in\R^3$, a fixed unit vector $b\in\R^d$, and where $\theta$ lies in one of the parameter classes
	         	\begin{align*}
	                \Theta_1&=\{\vartheta=(\vartheta_1,0,0):\vartheta_1\geq 1\},\nonumber\\
	                 \Theta_2&=\{\vartheta=(1,\vartheta_2,0):\vartheta_2\in[0,1]\},\nonumber\\
	                 \Theta_3&=\{\vartheta=(1,0,\vartheta_3):\vartheta_3\leq 0\}.
            	\end{align*}
	         \item Let $x_1,\dots,x_M$ be $\delta$-separated points in $\Lambda$, that is, $|x_k-x_l|> \delta$ for all $1\leq k\neq l\leq M$. Moreover, suppose that $\mathrm{supp}(K_{\delta,x_k})\subset\Lambda$ for all $k=1,\dots,M$ and  $\mathrm{supp}(K_{\delta,x_k})\cap\mathrm{supp}(K_{\delta,x_l})=\emptyset$ for all $1\leq k\neq l\leq M$.
	     \end{enumerate}
	\end{myassumption}
	
    The parameter classes $\Theta_i$ correspond to the cases of estimating the diffusivity $\theta_1$, transport coefficient $\theta_2$ and reaction coefficient $\theta_3$ in front of operators $A_i$ with differential orders $n_1=2$, $n_2=1$, $n_3=0$. We start with a non-asymptotic lower bound when only $X_{\delta}$ is observed. 
	
	\begin{theorem}\label{thm:lower:bound:M>1}
	   Grant Assumption \ref{assu:lowerBound} with $M\geq 1$, $T\geq 1$ and let $i\in\{1,2,3\}$. Then there exist constants $c_1,c_2>0$ depending only on $K$ and an absolute constant $c_3>0$ such that the following assertions hold:
	   \begin{itemize}
	       \item[(i)] If $\delta^{n_i-1}/\sqrt{TM}< 1$ and $\delta\leq c_1$, then 
	       \begin{align*}
            \inf_{\hat\vartheta_i}\sup_{\substack{\vartheta\in\Theta_i\\ |\vartheta-(1,0,0)^\top|\leq  c_2\frac {\delta^{n_i-1}}{\sqrt{TM}}}}
            \P_\vartheta\Big(|\hat\vartheta_i-\vartheta_i|\geq \frac{c_2}{2}\frac{\delta^{n_i-1}}{\sqrt{TM}}\Big)>c_3.
            \end{align*}
            \item[(ii)] If $\delta^{n_i-1}/\sqrt{TM}\geq 1$ and $\delta\leq c_1$, then 
            \begin{align*}
            \inf_{\hat\vartheta_i}\sup_{\substack{\vartheta\in\Theta_i\\|\vartheta-(1,0,0)^\top|\leq c_2}}\P_\vartheta(|\hat\vartheta_i-\vartheta_i|\geq c_2/2)>c_3.
            \end{align*}
	   \end{itemize}
        In (i) and (ii), the infimum is taken over all real-valued estimators $\hat\vartheta_i=\hat\vartheta_i(X_{\delta})$.
	\end{theorem}

	Several comments are in order for the above result. First, by Markov's inequality \mbox{Theorem \ref{thm:lower:bound:M>1}} also implies lower bounds for the squared risk. Second, part (ii) detects settings under which consistent estimation is impossible. For instance, if $i=2$, then consistent estimation is impossible for $T=1$ (resp.~$T$ bounded) and $M=1$, that is, if only a single spatial measurement is observed in a bounded time interval. A similar conclusion holds in the case $i=3$, in which case consistent estimation is even impossible in a full observation scheme with $M=\lceil c\delta^{-d}\rceil$ locations for $d\leq 2$ and $T$ bounded. Third, part (i) of \mbox{Theorem \ref{thm:lower:bound:M>1}} shows that the different rates in our CLT are minimax optimal. In particular, it easily implies an asymptotic minimax lower bound when $\delta\rightarrow 0$. A first important case is $M=1$ and $i=1$ in which case Theorem \ref{thm:lower:bound:M>1} also follows from Proposition 5.12 in \cite{altmeyer_nonparametric_2020} and gives the rate of convergence $\delta$. For $M=\lceil c\delta^{-d}\rceil$ we get the following.
	
	\begin{corollary}\label{cor:lower:bound:M>1}
	Grant Assumption \ref{assu:lowerBound} with $M=\lceil c\delta^{-d}\rceil$, $\delta\rightarrow 0$ and $T\geq 1$, and let $i\in\{1,2,3\}$.
    If $n_i-1+d/2>0$, then  
    \[
    \liminf_{\delta\rightarrow 0}\inf_{\hat{\theta}_{i}}\sup_{|\theta-(1,0,0)^\top|\leq c_1}\P_{\theta}(\delta^{-n_i+1-d/2}|\hat{\theta}_{i}-\theta_{i}|\geq c_{2})>0,
    \]
    where the infimum is taken over all real-valued estimators $\hat\vartheta_i=\hat\vartheta_i(X_{\delta})$. 
    \end{corollary}
    
    Similar optimality results have been derived in \cite{huebner_asymptotic_1995} for the case of $M$ spectral measurements. Provided there exists an orthonormal basis of eigenfunctions $(e_j)_{j=1}^\infty$ of $A_{\theta}$ independent of $\vartheta$ (e.g., in the case $i=1$ or $i=3$), it is possible to estimate $\vartheta_i$ from $M$ spectral measurements $(\sc{X(t)}{e_j})_{0\leq t\leq T,1\leq j\leq M}$ with rates $M^{-\tau}$ or  $\log M$ if $\tau=n_i/d-1/d+1/2>0$ or $\tau=0$, respectively. Consistent estimation fails to hold for $\tau<0$.
    While \cite{huebner_asymptotic_1995} obtained asymptotic efficiency by combining Girsanov's theorem with LAN techniques, these rates can also be derived from Lemma \ref{lem:RKHS:projected:SPDE} combined with a version of Lemma \ref{lem:seriesBound:M>1}. For $\delta=cM^{-1/d}$ the rate in Corollary \ref{cor:lower:bound:M>1} and Theorem \ref{thm:clt} coincides with $M^{-\tau}$ if $\tau>0$, and $\tau=0$ is again a boundary case. Regarding the latter case, we briefly discuss in Section \ref{sec:examples} that a non-negative point spread function achieves the $\log M$-rate when $i=3$ and $d=2$.

	Recall that the augmented MLE  $\hat{\theta}_{\delta}$ depends also on the measurements $X^A_{\delta}$. We show next that including them into the lower bounds does not change the optimal rates of convergence.

	\begin{theorem}\label{thm:lower:bound:M>1:add:measurements}
	Theorem \ref{thm:lower:bound:M>1} remains valid when the infimum is taken over all real-valued estimators $\hat\vartheta_i=\hat\vartheta_i(X_{\delta},X_{\delta}^\Delta,X_{\delta}^{\nabla\cdot b})$, provided that $K$, $\Delta K$ and $(\nabla\cdot b) K$ are linearly independent and Assumption \ref{assu:lowerBound} holds for  $K$, $\Delta K$ and $(\nabla\cdot b) K$.
	\end{theorem}

    \section{Applications and extensions}\label{sec:examples}
    
    \subsection{Examples}
    
   	Let us illustrate the main results in two examples.
  
   	\begin{example}
   	\label{ex:1}
		Suppose $A_\theta=\theta_1\Delta+\theta_2\nabla\cdot b+c$ for $\theta_1>0$. This corresponds to \eqref{eq:generalA} with $A_0=c$, $A_1=\Delta$, $A_2=\nabla\cdot b$ for $c\in\R$ and a unit vector $b\in\R^d$, and with differential orders $n_1=2$, $n_2=1$. A typical realisation of the solution $X$ in $d=1$ can be seen in Figure \ref{fig: heatmap_RMSE}(left). For known $c$, the augmented MLE $\hat{\theta}_{\delta}$ is a consistent estimator of $\theta\in\R^2$ by \mbox{Theorem \ref{thm:clt}}, attaining the optimal rates of convergence $M^{1/2}\delta^{-1}$, $M^{1/2}$ for the diffusivity and the transport terms, respectively according to the lower bounds in Theorem \ref{thm:lower:bound:M>1}. If we suppose for simplicity $\norm{K}_{L^2(\R^d)}=1$, then the CLT holds with a diagonal matrix $$\Sigma_{\theta} = \frac{T}{2\theta_1}\operatorname{diag}\left(\norm{\nabla K}^2_{L^2(\R^d)}, \norm{(-\Delta_0)^{-1/2}(\nabla\cdot b) K}^2_{L^2(\R^d)}\right),$$implying that $\hat{\theta}_{\delta,1}$ and $\hat{\theta}_{\delta,2}$ are asymptotically independent. 
		
		Figure \ref{fig: heatmap_RMSE}(right) presents root mean squared  errors in $d=1$ for local measurements obtained from the data displayed in the left part of the figure with  $K(x)=\exp(-5/(1-x^2))\mathbf{1}(-1<x<1)$ and the maximal choice of $M\asymp\delta^{-1}$. We see that the optimal rates of convergence, and even the exact asymptotic variances (blue dashed lines) are approached quickly as $\delta\rightarrow 0$. For comparison, we have included in Figure \ref{fig: heatmap_RMSE}(right) estimation errors for an estimator $\bar{\theta}_{\delta}$ without the correction factor depending on the lower order `nuisance operator' $A_0$ in \eqref{eq:augMLE}. We can see that this introduces only a small bias, which is negligible as $\delta\rightarrow 0$.
   	\end{example}
   	
   	\begin{example}\label{ex:2}
   		Consider now  $A_{\theta}=\theta_1\Delta + \theta_2\nabla\cdot b + \theta_3$ such that $A_1$, $A_2$ are as in the last example, but now also $A_0 =0$, $A_3 = 1$ with $n_3=0$. If $d\geq 3$ and $M^{1/2}\delta\rightarrow\infty$, then the CLT in Theorem \ref{thm:clt} applies with optimal rates of convergence as in the last example for $\theta_{\delta,1}$, $\theta_{\delta,2}$ and with rate $M^{1/2}\delta$ for the reaction term $\theta_3$. Using integration by parts we find
   	        \begin{equation*}
    			\Sigma_{\theta} = \frac{T}{2\theta_1}\begin{pmatrix}
                    \norm{\nabla K}^2_{L^2(\R^d)} & 0 & -1\\
                    0 & \norm{(-\Delta_0)^{-1/2}(\nabla\cdot b) K}^2_{L^2(\R^d)} & 0\\
                    -1 & 0 & \norm{(-\Delta_0)^{-1/2}K}^2_{L^2(\R^d)}
                    \end{pmatrix},
		    \end{equation*} 
        so we have pairwise asymptotic independence of diffusion and transport estimators, as well as of transport and reaction estimators. Similar numerical results as in the first example were obtained, but details are omitted.
    \end{example}

		\begin{figure}
        \centering
        \begin{subfigure}{.49\textwidth}
    	    \centering\includegraphics[width=1.\linewidth]{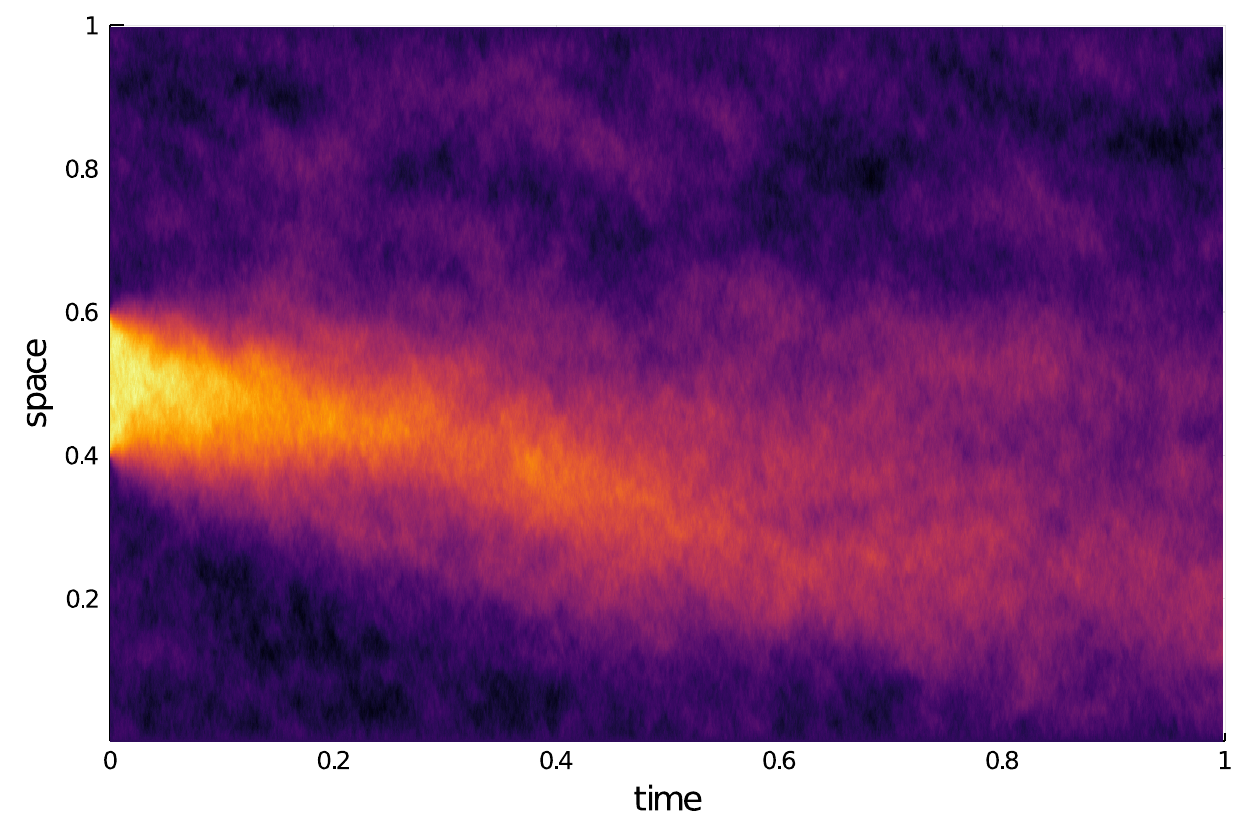}
    	\end{subfigure}
        \begin{subfigure}{.49\textwidth}
    	      \centering\includegraphics[width=1.\linewidth]{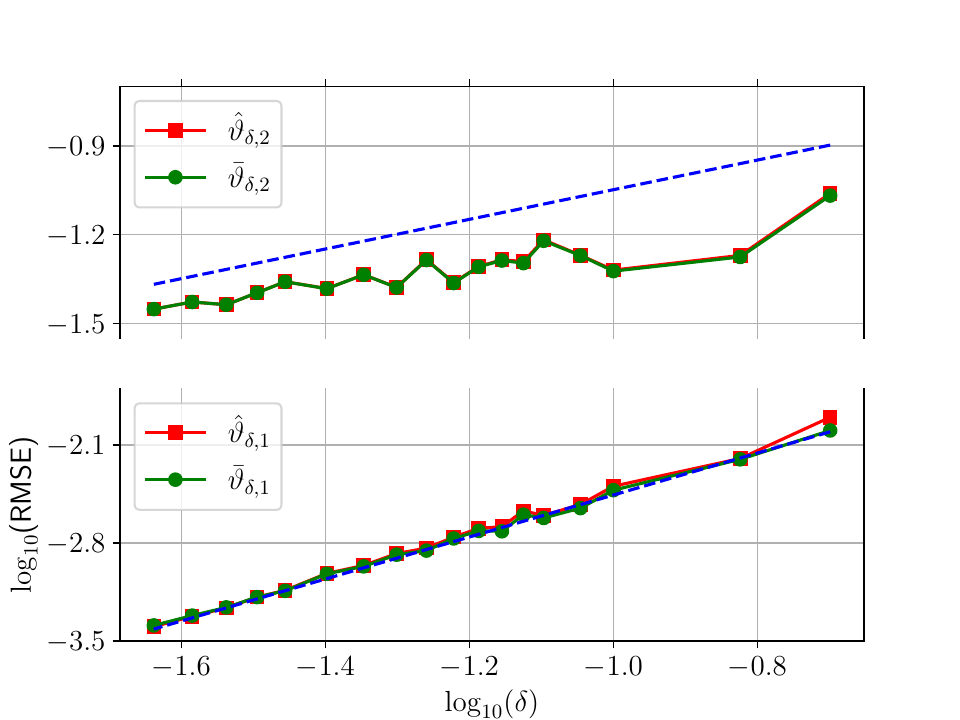}
        \end{subfigure}
        \caption{(left) heat map for a typical realisation of $X(t,x)$ corresponding to \eqref{eq:SPDE} in $d=1$ with domain $\Lambda=(0,1)$ in Example \ref{ex:1}; (right) $\log_{10}\log_{10}$ plot of the root mean squared errors for estimating $\theta_1$ and $\theta_2$ in Example \ref{ex:1}.}
        \label{fig: heatmap_RMSE}
    \end{figure}

	\subsection{A boundary case: estimation in $d=2$}

    Theorem \ref{thm:clt} is not valid for $d\leq 2$ and reaction terms $\theta_i$ with differential order $n_i=0$. The singularities of the heat kernel on $\R^d$ in $d\leq 2$ (cf. the discussion before Theorem \ref{thm:clt}) can be avoided for sufficiently regular $K$, e.g., by assuming $K=\Delta \tilde{K}$ for some $\tilde{K}\in H^4(\R^d)$. In that case, the CLT still holds with the same proof, but consistency towards $\theta_i$ is lost, because $M^{1/2}\delta$ does not diverge. Nevertheless, we show now that in $d=2$ for non-negative $K$, a logarithmic rate holds. This is consistent with results for the MLE from spectral observations in $d=2$, cf. \cite{huebner_asymptotic_1995}. For a proof see Section\ref{sec: reaction}. For simplicity, only a simplified model is considered.
	
	\begin{proposition}\label{prop: reactiond2}
	Suppose that $d=2$, $A_{\theta}=\Delta+\theta$ for $\theta\in\R$, $X(0)=0$ and $M\delta^2\rightarrow 1$ as $\delta\rightarrow 0$. If $K\geq 0$ and $K\neq 0$, then
	        $\hat{\theta}_\delta=\theta+O_{\P}(\log(\delta^{-1})^{-1/2})$.
	\end{proposition}

\subsection{Practical aspects}\label{sec:Practical_aspects} 

    In this section, we outline a precise situation where local measurements arise, and how the augmented MLE can be applied, even if the additional measurements $X^{A_i}_{\delta,k}$ are not available.
    
    Optical measurements of physical or chemical concentrations $X(t)$ at a focal point $x_k\in\Lambda$ are obtained as (normalised) counts of certain markers, e.g., photons \cite{egner2020}. According to classical microscopy \cite{kulaitis_what_2020}, diffraction leads to a blurred image of $X(t)$, and the blur pattern can be described by convolution with a specific point spread function, which can be written as 
    \begin{align}
        X_{\delta,k}(t)=\sc{X(t)}{K_{\delta,x_k}}=(X(t)*\bar{K}_{\delta})(x_k),\qquad \bar{K}_{\delta}(y)=\delta^{d/2}K(-\delta^{-1}y).\label{eq:locMeas}
    \end{align}
    It is reasonable to assume that the additional measurement noise due to photon counting is negligible and that measuring happens on faster time scales than the dynamics of $X$. 

    The resolution $\delta$ is specific to the measurement device and determines how far focal points can be apart to distinguish them \cite[Definition 2.5]{kulaitis_what_2020}. The point spread function depends inversely on $\delta$, and is often approximated by a normal density with standard deviation  $\delta$ \cite{aspelmeier_modern_2015}. This phenomenon is the source for the large number of statistical works on Gaussian deconvolution. In applications, both the point spread function and the resolution $\delta$ are usually known, and can even be engineered to meet desired specifications \cite{backer_extending_2014}. Note that multiplicative constants such as the scaling of $K_{\delta,x_k}$ cancel out in the augmented MLE and therefore play no role for parameter estimation. 

    If we have (time discrete) local measurements \eqref{eq:locMeas} at our disposal, then  exchanging differentiation and convolution gives$$X^{A_i}_{\delta,k}=(X(t)*A_i^*\bar{K}_{\delta})(x_k)=A_i^* (X(t)*\bar{K}_{\delta})(x_k).$$This can be approximated by finite differences. For example, if $A_i=\Delta$ and $x_{k-1}=x_k-\delta e_i$, $x_{k+1}=x_k+\delta e_i$ are `neighbours' of $x_k$ in the $i$-th coordinate with the unit vector $e_i$, separated by a distance $\delta$, then $X^{A_i}_{\delta,k}(t)$ can be approximated by $\delta^{-2}(X_{\delta,k+1}(t)-2X_{\delta,k}(t)+X_{\delta,k-1}(t))$. Using suitable Riemann sum approximations for Lebesgue and stochastic integrals, we thus obtain a discretised version of the augmented MLE $\hat{\theta}_{\delta}$. 

    While a full analysis of such discretisation schemes is beyond the scope of this paper, we shortly report on a recent case study for cell motility, using the augmented MLE for real and simulated data \cite[Sections 5 and 6]{altmeyer_parameter_2020}. There, the first component of a coupled stochastic activator-inhibitor system $(X_1,X_2)$ follows a semi-linear SPDE with diffusity $\theta$, reaction function $f$ and noise level $\sigma>0$,
    \begin{align*}
        dX_1(t) = (\theta \Delta X_1(t) + f(X_1(t),X_2(t)))dt+\sigma dW(t).
    \end{align*}
    The equation models the change in actine concentration along the cell cortex during cell repolarisation. In \cite{lockley_image-based_2017}, on a time grid of up to 256 seconds $M=100$ measurements for 18 different cells, expected to have about the same diffusivities, were used to fit parameters in the deterministic PDE with $\sigma=0$. In \cite{altmeyer_parameter_2020}, the same data were taken as local measurements from $X_1(t)$, and $\theta$ was estimated by the discretised augmented MLE as discussed above, providing a biologically reasonable magnitude for $\theta$, which can be used to distinguish the mechanisms contributing to diffusion. The resolution $\delta$ was found as an upper bound on the spatial mesh size. The estimates are stable across the cell populations as opposed to \cite{lockley_image-based_2017}, which averaged the different estimates across cells to reduce `noise', and obtained in this way a much inflated average diffusivity.

	\section{Proofs}\label{sec:proofs}

	\subsection{Proof of the central limit theorem}

    \subsubsection{Preliminaries}
    
	We write $\Lambda_{\delta,x}  = \{\delta^{-1}(y-x):y\in\Lambda\}$, $\Lambda_{0,x}=\mathbb{R}^d$ and introduce with domains $H_0^1(\Lambda_{\delta,x})\cap H^2(\Lambda_{\delta,x})$  the operators 
    \begin{align}
	     A_{\theta,\delta,x} &= \nabla \cdot a_{\theta} \nabla + \delta \nabla\cdot b_{\theta} + \delta^2 c_{\theta}, \quad \tilde{A}_{\theta,\delta,x} = \nabla \cdot a_{\theta} \nabla\label{eq:A_local}.
	\end{align}
	They generate the analytic semigroups $(S_{\theta,\delta,x}(t))_{t\geq 0}$ and $(\tilde{S}_{\theta,\delta,x}(t))_{t\geq 0}$ on $L^2(\Lambda_{\delta,x})$. Similarly, the adjoint operators $A_{\theta,\delta,x}^*$ and $\tilde{A}_{\theta,\delta,x}^*$ generate with the same domains the adjoint semigroups $(S_{\theta,\delta,x}^*(t))_{t\geq 0}$ and $(\tilde{S}^*_{\theta,\delta,x}(t))_{t\geq 0}$. When $a_{\theta}$ is the identity matrix, then we also write $\Delta_{\delta,x}=\tilde{A}_{\theta,\delta,x}$ and $e^{t\Delta_{\delta,x}}=\tilde{S}_{\theta,\delta,x}(t)$. Moreover, $e^{t\Delta_0}$ and $e^{t\nabla\cdot a_{\theta}\nabla}$ are semigroups on $L^2(\R^d)$ generated by $\Delta_0$ and $\nabla\cdot a_{\theta}\nabla$, respectively, with domain $H^2(\R^d)$.  We often use implicitly that $z\in L^2(\Lambda_{\delta,x})$ extends to an element of $L^2(\R^d)$ by setting $z(y)=0$ outside of $\Lambda_{\delta,x}$. The $A_i$ and their formal adjoints $A_i^*$ are considered as differential operators on sufficiently weakly differentiable functions without boundary conditions.

	\subsubsection{The rescaled semigroup}\label{sec:rescaledSemigroup}

 In this section we collect some results on the semigroup operators $S_{\theta}(t)$ and their actions on localised functions $z_{\delta,x}(\cdot)=\delta^{-d/2}z(\delta^{-1}(\cdot-x))$.

    By a standard-PDE result (see, e.g., \cite[Example III.6.11]{kato2013perturbation} or \cite[equation (5.1)]{reddy1994pseudospectra}), the operator $A_{\theta,\delta,x}$ and the generated semigroup are diagonalizable \cite[Example 2.1 in Section II.2]{EngNag00}. This yields the useful representations 
	\begin{align}
	    A_{\theta,\delta,x}=U_{\theta,\delta,x}(\tilde{A}_{\theta,\delta,x} +\delta^2\tilde{c}_{\theta})U_{\theta,\delta,x}^{-1},\quad S_{\theta,\delta,x}(t)=e^{t\delta^2\tilde{c}_{\theta}}U_{\theta,\delta,x}\tilde{S}_{\theta,\delta,x}(t)U_{\theta,\delta,x}^{-1}\label{eq:diagonalizable}
	\end{align}
	with the multiplication operators $U_{\theta,\delta,x}z(y)=\exp(-(a_\theta^{-1}b_\theta)\cdot (\delta y+x)/2)z(y)$ and with $\tilde{c}_\theta=c_\theta -\frac{1}{4}b_\theta\cdot (a_\theta^{-1}b_\theta)$. Observe the following scaling properties.

    \begin{lemma}\label{rescaledsemigroup}
        Let $\delta'\geq \delta\geq 0$, $x\in\Lambda$, $i=1,\dots,p$. 
        \begin{enumerate}[label=(\roman*)]
            \item If $z\in H_0^1(\Lambda_{\delta,x})\cap H^2(\Lambda_{\delta,x})$, then $A_i^*z_{\delta,x}=\delta^{-n_i}(A_i^*z)_{\delta,x}$, $A^*_{\theta}z_{\delta,x}=\delta^{-2}(A^*_{\theta,\delta,x}z)_{\delta,x}$.
            \item If $z\in L^2(\Lambda_{\delta,x})$, $t\geq 0$, then  $S_{\theta}^*(t)z_{\delta,x}=(S_{\theta,\delta,x}^*(t\delta^{-2})z)_{\delta,x}$.
        \end{enumerate}
    \end{lemma}
    \begin{proof}
        Part (i) is clear, part (ii) follows analogously to \cite[Lemma 3.1]{altmeyer_nonparametric_2020}. 
    \end{proof}
     The semigroup on the bounded domain $\Lambda_{\delta,x}$ is after zooming in as $\delta\rightarrow 0$ intuitively close to the semigroup on $\R^d$. The next result makes this precise, uniformly in $x\in\mathcal{J}$.
	
	\begin{lemma}
		\label{FeynmanKac}
		Under Assumption \ref{assump:mainAssump}(iii) the following holds:
		\begin{itemize}
			\item [(i)] There exists $C>0$ such that if $z\in C_c(\mathbb{R}^d)$ is supported in $\bigcap_{x\in \mathcal{J}}\Lambda_{\delta,x}$ for some $\delta\geq 0$, then for all $t\geq 0$ 
			\begin{equation*}
				\sup_{x\in \mathcal{J}}\left|(S^*_{\vartheta,\delta,x}(t)z)(y)\right|\leq Ce^{\tilde{c}_\theta t\delta^2}(e^{t\nabla\cdot a_{\theta}\nabla}|z|)(y),\quad y\in\R^d.
			\end{equation*}
			\item [(ii)] If $z\in L^2(\mathbb{R}^d)$, then as $\delta \rightarrow 0$ for all $t>0$
			\begin{equation*}
				\sup_{x\in \mathcal{J}}\Norml S^*_{\vartheta,\delta,x}(t)(z|_{\Lambda_{\delta,x}})-e^{t\nabla\cdot a_{\theta}\nabla}z\Normr_{L^2(\mathbb{R}^d)}\rightarrow0.
			\end{equation*}
		\end{itemize}
		\end{lemma}
		\begin{proof}
		    (i). By \eqref{eq:diagonalizable} and noting that the function $y\mapsto \exp(-(a_\theta^{-1}b_\theta)\cdot (\delta y+x)/2)$ is uniformly upper and lower bounded on $\bigcap_{x\in\mathcal{J}}\Lambda_{\delta,x}$, we get 
            \begin{equation*}
				\sup_{x\in \mathcal{J}}\left|(S^*_{\vartheta,\delta,x}(t)z)(y)\right|\lesssim  e^{t\delta^2\tilde{c}_{\theta}}(\tilde{S}_{\vartheta,\delta,x}(t)|z|)(y),\quad y\in\R^d.
			\end{equation*}
		    It is therefore enough to prove the claim with respect to $\tilde{S}_{\theta,\delta,x}$ and with $|z|$ instead of $z$. By the classical Feynman-Kac formulas (cf. \cite[Chapter 4.4]{karatzas_brownian_1998}, the anisotropic case is an easy generalisation, which can also be obtained by a change of variables leading to a diagonal diffusivity matrix $a_{\theta}$, which corresponds to $d$ scalar heat equations) we have with a process $Y_t=y+\at^{1/2}\tilde{W}_t$ and a $d$-dimensional Brownian motion $\tilde{W}$, all defined on another probability space with expectation and probability operators $\tilde{\E}_y$, $\tilde{\P}_y$, that $(e^{t\nabla\cdot a_{\theta}\nabla}z)(y)=\tilde{\mathbb{E}}_y[z(Y_t)]$ and 			$\tilde{S}_{\vartheta,\delta,x}(t)z(y)=\tilde{\mathbb{E}}_y\left[z(Y_t)\mathbf{1}\left(t<\tau_{\delta,x}\right)\right]$ with the stopping times $\tau_{\delta,x}:=\inf\{t\geq0:Y_t\notin\Lambda_{\delta,x}\}$. The claim follows now from
		    \begin{equation*}
				\sup_{x\in \mathcal{J}}(\tilde{S}_{\theta,\delta,x}(t)|z|)(y)\leq \tilde{\mathbb{E}}_y[|z(Y_t)|]=(e^{t\nabla\cdot a_{\theta}\nabla}|z|)(y).
			\end{equation*}
		    
		    (ii). By an approximation argument it is enough to consider $z\in C_c(\bar{\Lambda})$ and $0<\delta\leq\delta'$ such that $z$ is supported in $\Lambda_{\delta',x}$, hence $z|_{\Lambda_{\delta,x}}=z$ for all such $\delta$. Compactness of $\mathcal{J}$ according to Assumption \ref{assump:mainAssump}(iii) guarantees for sufficiently small $\delta$ the existence of a ball with centre $0$ and radius $\rho\delta^{-1}$ for some $\rho>0$, contained in $\bigcap_{x\in\mathcal{J}}\Lambda_{\delta,x}$. With this and the representation formulas in (i), combined with the Cauchy-Schwarz inequality, we have for all $y\in\R^d$ 
			\begin{align*}
			    & \sup_{x\in\mathcal{J}}|(\tilde{S}_{\vartheta,\delta,x}(t)z)(y)-(e^{t\nabla\cdot a_{\theta}\nabla}z)(y)|^2=\sup_{x\in\mathcal{J}}|\tilde{\mathbb{E}}_y\left[z(Y_t)\mathbf{1}(\tau_{\delta,x}\leq t)\right]|^2\nonumber\\
				& \leq \sup_{x\in\mathcal{J}}\tilde{\E}_y[z^2(Y_t)]\tilde{\mathbb{P}}_{y}(\tau_{\delta,x}\leq t)\leq (e^{t\nabla\cdot a_{\theta}\nabla}z^2)(y)\tilde{\mathbb{P}}_{y}(\max_{0\leq s\leq t} |Y_s| \geq \rho\delta^{-1})\nonumber\\
				& \leq (e^{t\nabla\cdot a_{\theta}\nabla}z^2)(y)\tilde{\mathbb{P}}_{y}(\max_{0\leq s\leq t} |\tilde{W}_s| \geq \tilde{\rho}\delta^{-1})\lesssim (e^{t\nabla\cdot a_{\theta}\nabla}z^2)(y)(\delta t^{1/2} e^{-\delta^{-2}t^{-1}})\rightarrow 0
			\end{align*} 
			as $\delta\rightarrow 0$ for another constant $\tilde{\rho}$, concluding by \cite[equation (2.8.4)]{karatzas_brownian_1998}. Since $\norm{e^{t\nabla\cdot a_{\theta}\nabla}z^2}_{L^1(\R^d)}\leq \norm{z}^2_{L^2(\R^d)}$, dominated convergence proves the claim when $b_{\theta}=0$, $c_{\theta}=0$. The general case is then an easy consequence of the last display and \eqref{eq:diagonalizable}.
		\end{proof}

	We require frequently quantitative statements on the decay of the action of the semigroup operators $S^*_{\theta,\delta,x}(t)$ as $t\rightarrow\infty$ when applied to functions of a certain smoothness and integrability. This is well-known for an analytic semigroup, but is shown here to hold true for all $\delta$  and uniformly in $x\in\mathcal{J}$.
	
	\begin{lemma}\label{lem:semigroupProp}
	    Let $0\leq \delta \leq 1$, $t>0$, $x\in \mathcal{J}$ and $1< p\leq \infty$. Moreover, let $z\in L^p(\Lambda_{\delta,x})$ if $1<p<\infty$ and $z\in C(\Lambda_{\delta,x})$ with $z=0$ on $\partial \Lambda_{\delta,x}$ if $p=\infty$. Then it holds with implied constants not depending on $x$:
	    \begin{align*}
	        \norm{A^*_{\theta,\delta,x}S^*_{\theta,\delta,x}(t)z}_{L^p(\Lambda_{\delta,x})}\lesssim t^{-1}\norm{z}_{L^p(\Lambda_{\delta,x})}.
	    \end{align*}
	\end{lemma}
	\begin{proof}
	    Apply first the scaling in Lemma \ref{rescaledsemigroup} in reverse order such that with $1<p\leq \infty$
	    \begin{align*}
	        \norm{A^*_{\theta,\delta,x}S^*_{\theta,\delta,x}(t)z}_{L^p(\Lambda_{\delta,x})} = \delta^{d(1/2-1/p)+2} \Norml{A^*_{\theta}S^*_{\theta}(t\delta^{2})z_{\delta,x}}\Normr_{L^{p}(\Lambda)}.
	    \end{align*}
	    If $p<\infty$, by the semigroup property for analytic semigroups in \cite[Theorem V.2.1.3]{amann1995linear}, the $L^p(\Lambda)$-norm is up to a constant upper bounded by $(t\delta^{2})^{-1}\norm{z_{\delta,x}}_{L^p(\Lambda)}$, and the claim follows. The same proof applies to $p=\infty$, noting that $A^*_{\theta}$ generates an analytic semigroup on $\{u\in C(\Lambda),u=0$ on $\partial\Lambda\}$, cf. \cite[Theorem 7.3.7]{pazy_semigroups_1983}.
	\end{proof}
	
	The proof for the next result relies on the Bessel-potential spaces $H_0^{s,p}(\Lambda_{\delta,x})$, $1<p<\infty$, $s\in\R$, defined for $\delta>0$ as the domains of the fractional Dirichlet-Laplacian $(-\Delta_{\delta,x})^{s/2}$ on $\Lambda_{\delta,x}$ with norms $\norm{\cdot}_{H^{s,p}(\Lambda_{\delta,x})}=\norm{(-\Delta_{\delta,x})^{s/2}\cdot}_{L^p(\Lambda_{\delta,x})}$, see \cite{debussche_regularity_2015} for details and also Section \ref{sec:RKHS:computations} below. Since $a_{\theta}$ is positive definite, the norms $\norm{\cdot}_{H^{s,p}(\Lambda_{\delta,x})}$ are equivalently generated by the fractional powers of $-\tilde{A}_{\theta,\delta,x}$ \cite[Theorem 16.15]{yagi_abstract_2009}.
	
	\begin{lemma}
		\label{boundS*u}
		Let $0< \delta\leq 1$, $t>0$, $1<p\leq 2$ and grant Assumption \ref{assump:mainAssump}(iii). Let $z\in H_0^{s}(\mathbb{R}^d)$, $s\geq 0$, be compactly supported in $\bigcap_{x\in \mathcal{J}}\Lambda_{\delta,x}$, suppose that $V_{\delta,x}:L^p(\Lambda_{\delta,x})\rightarrow H_0^{-s,p}(\Lambda_{\delta,x})$ are bounded linear operators with $\norm{V_{\delta,x}z}_{H^{-s,p}(\Lambda_{\delta,x})}\leq V_{\operatorname{op}}\norm{z}_{L^p(\Lambda_{\delta,x})}$ for some $V_{\operatorname{op}}$ independent of $\delta$, $x$. Then for $1< p\leq2$ and  $\gamma=(1/p-1/2)d/2$ there exists a constant $C>0$, depending on $p$ and $s$ such that 
		\begin{align*}
		\sup_{x\in \mathcal{J}}\Norml S^*_{\vartheta,\delta,x}(t) V_{\delta,x}z \Normr_{L^2(\Lambda_{\delta,x})}
			& \leq Ce^{\tilde{c}_{\theta}t\delta^2}\sup_{x\in \mathcal{J}}\left(\norm{V_{\delta,x}z}_{L^2(\Lambda_{\delta,x})}\wedge (V_{\operatorname{op}}t^{-s/2-\gamma}\norm{z}_{L^{p}(\Lambda_{\delta,x})})\right).
		\end{align*}
		If $s=0$, then this holds also for $p=1$.
		\begin{proof}
		   Set $u=V_{\delta,x}z$, $v=U_{\theta,\delta,x}u$. The $U_{\theta,\delta,x}$ are bounded operators on $L^2(\Lambda_{\delta,x})$ uniformly in $\delta\geq 0$ and $x\in\mathcal{J}$ and thus by \eqref{eq:diagonalizable} 
		    \begin{align}
		        \Norml S^*_{\vartheta,\delta,x}(t) u \Normr_{L^2(\Lambda_{\delta,x})}
		            & \lesssim e^{\tilde{c}_{\theta}t\delta^2}\Norml \tilde{S}_{\vartheta,\delta,x}(t) v \Normr_{L^2(\Lambda_{\delta,x})}.\label{eq:bound_1}
		    \end{align}
		    Let first $s=0$ such that $H_0^{-s,p}(\R^d)=L^p(\R^d)$. Ellipticity and symmetry of $a_{\theta}$ show $\norm{e^{t\nabla\cdot a_{\theta}\nabla}|v|}_{L^2(\R^d)}\leq \norm{e^{tC'\Delta_0}|v|}_{L^2(\R^d)}$ for a constant $C'>0$ (use either \cite{sheu_estimates_1991} or argue that the semigroup $e^{t\nabla\cdot a_{\theta}\nabla}$ on acts on $L^2(\R^d)$ as a multiplication operator in the Fourier domain according to \cite[Section VI.5]{EngNag00}, which can be upper bounded by the identity operator). Approximating $u$ by continuous and compactly supported functions, we thus find from Lemma \ref{FeynmanKac}(i) and hypercontractivity of the heat kernel on $\mathbb{R}^d$ uniformly in $x\in\mathcal{J}$
			\begin{align*}
				\Norml S^*_{\vartheta,\delta,x}(t) u \Normr_{L^2(\Lambda_{\delta,x})}
				& \lesssim e^{\tilde{c}_{\theta}t\delta^2} \Norml e^{C't\Delta_0}|v| \Normr_{L^2(\R^d)} \lesssim  e^{\tilde{c}_{\theta}t\delta^2}t^{-\gamma}\norm{u}_{L^p(\mathbb{R}^d)}\lesssim e^{\tilde{c}_{\theta}t\delta^2}t^{-\gamma}\norm{z}_{L^p(\R^d)}.
			\end{align*}
			This yields the result for $s=0$. These inequalities hold also for $p=1$, thus proving the supplement of the statement. For $s>0$ and $p>1$ note first that by \cite[Proposition 17(i)]{altmeyer_parameterSemi_2020} we have $\norm{(-t\tilde{A}_{\theta,\delta,x})^{s/2}\tilde{S}_{\theta,\delta,x}(t)z}_{L^2(\Lambda_{\delta,x})}\lesssim \norm{z}_{L^2(\Lambda_{\delta,x})}$. Inserting this and then the last display with $u$ replaced by $(-\tilde{A}_{\theta,\delta,x})^{-s/2} v$ into  \eqref{eq:bound_1} we get
			\begin{align*}
				& \Norml S^*_{\vartheta,\delta,x}(t) u \Normr_{L^2(\Lambda_{\delta,x})} 
				 \lesssim e^{\tilde{c}_{\theta}t\delta^2}\Norml{(-\tilde{A}_{\theta,\delta,x})^{s/2}\tilde{S}_{\theta,\delta,x}(t)(-\tilde{A}_{\theta,\delta,x})^{-s/2}v}\Normr_{L^2(\Lambda_{\delta,x})}\\
				& \quad \lesssim e^{\tilde{c}_{\theta}t\delta^2}t^{-s/2}\Norml{\tilde{S}_{\theta,\delta,x}(t/2)(-\tilde{A}_{\theta,\delta,x})^{-s/2}v}\Normr_{L^2(\Lambda_{\delta,x})}\lesssim e^{\tilde{c}_{\theta}t\delta^2}t^{-s/2-\gamma} \Norml{v}\Normr_{H^{-s,p}(\Lambda_{\delta,x})},
			\end{align*}
			uniformly in $x\in\mathcal{J}$. Note that the  $U_{\theta,\delta,x}$ also induce a family of multiplication operators on $H^{s,p}_0(\R^d)$ for $s\geq 0$ with operator norms uniformly bounded in $x\in\mathcal{J}$, cf. \cite[Theorem 2.8.2]{Triebel1983book}. By duality and restriction this transfers to $H^{s,p}_0(\Lambda_{\delta,x})$ for general $s$ according to \cite[Theorem 3.3.2]{Triebel1983book}. Hence,
			\begin{align*}
			    \Norml S^*_{\vartheta,\delta,x}(t) u \Normr_{L^2(\Lambda_{\delta,x})} \lesssim  e^{\tilde{c}_{\theta}t\delta^2}t^{-s/2-\gamma}\Norml{u}\Normr_{H^{-s,p}(\Lambda_{\delta,x})}\lesssim e^{\tilde{c}_{\theta}t\delta^2}t^{-s/2-\gamma} V_{\operatorname{op}}\norm{z}_{L^{p}(\Lambda_{\delta,x})}.\hspace*{9mm}\qedhere
			\end{align*}
		\end{proof}
	\end{lemma}	
 
	\subsubsection{Covariance structure of multiple local measurements}

	\begin{lemma}\label{lem:covFun}
	\begin{enumerate}[label=(\roman*)]
	    \item If $X_0=0$, then the Gaussian process from \eqref{eq:X_GP} has mean zero and covariance function
	    \begin{align*}
		 &\operatorname{Cov}(\sc{X(t)}{z},\sc{X(t')}{z'}) = \int_0^{t\wedge t'} \sc{ S^*_{\theta}(t-s)z}{ S^*_{\theta}(t'-s)z'}ds.
	    \end{align*}
	    \item If $X_0$ is the stationary initial condition from Lemma \ref{lem:X0}, then the Gaussian process from \eqref{eq:X_GP} has mean zero and covariance function
	    \begin{align*}
		 &\operatorname{Cov}(\sc{X(t)}{z},\sc{X(t')}{z'}) = \int_0^{\infty} \sc{ S^*_{\theta}((t-t')+s)z}{ S^*_{\theta}(s)z'}ds,\quad t\geq t'.
	    \end{align*}
	\end{enumerate}
	\end{lemma}
	\begin{proof}
	    Part (i) follows from \eqref{eq:X_GP} and Itô’s isometry \cite[Proposition 4.28]{da_prato_stochastic_2014}. For part (ii) we conclude in the same way from noting that the stationary solution given by $\sc{X(t)}{z}=\int_{-\infty}^{t}\sc{S^*_{\theta}(t-s)z}{dW(s)}$ has mean zero.
	\end{proof}

    Introduce for $i,j=1,\dots,p$
		\begin{align*}
		   \Psi_{\theta}(A_i^*K,A_j^*K) &= \frac{1}{2}\sc{(-\nabla\cdot a_{\theta}\nabla)^{-1/2}A_i^*K}{(-\nabla\cdot a_{\theta}\nabla)^{-1/2}A_j^*K}_{L^2(\R^d)},
		\end{align*}
    which is well-defined under Assumption \ref{assump:mainAssump}. by the discussion before Theorem \ref{thm:clt}.

	\begin{lemma}
		\label{ConvFisher} Grant Assumption \ref{assump:mainAssump} and let $X_0=0$. 
		We have as $\delta\rightarrow 0$ 
		    \begin{align*}
		       \delta^{-2+n_i+n_j}(MT)^{-1}\sum_{k=1}^M\int_{0}^{T}\E\left[\sc{X(t)}{A_i^*K_{\delta,x_k}}\sc{X(t)}{A_j^*K_{\delta,x_k}}\right]dt \rightarrow \Psi_{\theta}(A_i^*K,A_j^*K).
		    \end{align*}
		\begin{proof}
		   Fix $i,j$ with $n_i+n_j>2-d$. Then, applying Lemma \ref{lem:covFun}(i), the scaling from Lemma \ref{rescaledsemigroup} and changing variables give \begin{align*}
		       \delta^{-2+n_i+n_j}(MT)^{-1}\sum_{k=1}^M\int_{0}^{T}\E\left[\sc{X(t)}{A_i^*K_{\delta,x_k}}\sc{X(t)}{A_j^*K_{\delta,x_k}}\right]dt = 
			    \int_0^{\infty}f_{\delta}(t')dt'
		   \end{align*}
		   with
		    \begin{align*}
			    f_{\delta}(t')=(MT)^{-1}\sum_{k=1}^M\sc{S^*_{\theta,\delta,x_k}(t')A^*_{i}K}{S^*_{\theta,\delta,x_k}(t')A^*_{j}K}_{L^2(\Lambda_{\delta,x_k})}\int_0^T\mathbf{1}_{\{0\leq t'\leq t\delta^{-2}\}}dt.
			\end{align*}
			Consider now the differential operators $V_{\delta,x_k}=A^*_{i}$. If $D^m$ is a composition of $m$ partial differential operators, then Theorem 1.43 of \cite{yagi_abstract_2009} yields that $D^m$ is a bounded linear operator from $L^{p}(\Lambda)$ to $H_0^{-m,p}(\Lambda)$, implying $\Norml{D^m K_{\delta,x_k}}\Normr_{H^{-m,p}(\Lambda)}\lesssim \delta^{-m}\Norml{K_{\delta,x_k}}\Normr_{L^{p}(\Lambda)}$. Since $(D^m K)_{\delta,x_k}=\delta^m D^m K_{\delta,x_k}$, changing variables gives $\Norml{D^m K}\Normr_{H^{-m,p}(\Lambda_{\delta,x_k})}\lesssim \Norml{K}\Normr_{L^{p}(\Lambda_{\delta,x_k})}$. From this we find  $\norm{V_{\delta,x_k}K}_{H^{-n_i,p}(\Lambda_{\delta,x})}\leq \norm{K}_{L^p(\Lambda_{\delta,x_k})}$, $\norm{V_{\delta,x_k}K}_{L^2(\Lambda_{\delta,x_k})}\lesssim \norm{K}_{H^{n_i}(\R^d)}$, and Lemma \ref{boundS*u}  shows for $0\leq t'\leq T\delta^{-2}$, $\varepsilon>0$ and all sufficiently small $\delta> 0$
			\begin{equation}
				\sup_{x\in\mathcal{J}}\norm{S^*_{\vartheta,\delta,x}(t') A^*_{i}K}_{L^2(\Lambda_{\delta,x})} \lesssim 1\wedge (t')^{-n_i/2-d/4+\varepsilon}.
				\label{eq:fisher_5}
			\end{equation}
            By the Cauchy-Schwarz inequality we get $|f_{\delta}(t')|\lesssim 1\wedge (t')^{-n_i/2-n_j/2-d/2+2\varepsilon}$. In particular, taking  $\epsilon$ so small that $n_i+n_j>2-d-4\epsilon$ yields $\sup_{\delta> 0}|f_{\delta}|\in L^1([0,\infty))$. Lemma \ref{FeynmanKac}(ii), Lemma \ref{rescaledsemigroup}(ii) and continuity of the $L^2$-scalar product show now pointwise for all $t'>0$ that $f_{\delta}(t')\rightarrow \sc{e^{2t'\nabla\cdot a_{\theta}\nabla}A^*_iK}{A^*_jK}_{L^2(\R^d)}$. Conclude by the dominated convergence theorem and $\int_0^\infty\sc{e^{2t'\nabla\cdot a_{\theta}\nabla}A_i^*K}{A_j^*K}_{L^2(\R^d)} dt'=\Psi_{\theta}(A_i^*K,A_j^*K)$.
		\end{proof}
	\end{lemma}
	
	\begin{lemma}
		\label{ConvouterVar}
		Grant Assumption \ref{assump:mainAssump} and let $X_0=0$. If $n_i+n_j> 2-d$ for $i,j=1,\dots,p$, then $\sup_{x\in\mathcal{J}}\operatorname{Var}(\int_0^T\sc{X(t)}{A^*_iK_{\delta,x}}\sc{X(t)}{A^*_jK_{\delta,x}}dt)=o(\delta^{4-2n_i-2n_j})$.
		\begin{proof}
			Applying the scaling from Lemma \ref{rescaledsemigroup} and using Wicks theorem \cite[Theorem 1.28]{janson_gaussian_1997} we have for $x\in\mathcal{J}$
			\begin{align*}
			  &\delta^{2n_i+2n_j}\operatorname{Var}\Big(\int_0^T\sc{X(t)}{A_i^*K_{\delta,x}}\sc{X(t)}{A_j^* K_{\delta,x}}dt\Big)\\
			    &\quad =\operatorname{Var}\Big(\int_0^T\sc{X(t)}{(A^*_{i}K)_{\delta,x}}\sc{X(t)}{(A^*_{j}K)_{\delta,x}}dt\Big) = V_1 + V_2
			\end{align*}
			with $V_1 = V_{\delta,x}(A^*_{i}K, A^*_{i}K, A^*_{j}K, A^*_{j}K)$, $V_2 = V_{\delta,x}(A^*_{i} K, A^*_{j} K, A^*_{j} K, A^*_{i} K)$, 
			and where for $v,v',z,z'\in L^2(\Lambda_{\delta,x})$ 
			\begin{align*}
			V_{\delta,x}(v,v',z,z') =	\int_0^T \int _0^{T} \E[\sc{X(t)}{v_{\delta,x}}\sc{X(t')}{v'_{\delta,x}}]\E[\sc{X(t)}{z_{\delta,x}}\sc{X(t')}{z'_{\delta,x}}]dt'dt.
			\end{align*}
			We only upper bound $V_1$, the arguments for $V_2$ are similar. Set $f_{i,j}(s,s')=\sc{S^*_{\theta,\delta,x}(s+s')A^*_{i}K}{S^*_{\theta,\delta,x}(s')A^*_{j}K}_{L^2(\Lambda_{\delta,x})}$. Using Lemma \ref{lem:covFun}(i) and the scaling in Lemma \ref{rescaledsemigroup} we have
			\begin{align*}
			    V_1 = 2\delta^6\int_0^{T}\int_0^{t\delta^{-2}}&\Big(\int_0^{t\delta^{-2}-s}f_{i,i}(s,s')ds'\Big)\Big(\int_0^{t\delta^{-2}-s}f_{j,j}(s,s'')ds''\Big)dsdt,
			\end{align*}
			cf. \cite[Proof of Proposition A.9]{altmeyer_nonparametric_2020}. 
			From \eqref{eq:fisher_5} and the Cauchy-Schwarz inequality we infer 
			\begin{align*}
			    \sup_{x\in\mathcal{J}}|f_{i,i}(s,s')f_{j,j}(s,s'')| \lesssim &(1\wedge s^{-(n_i+n_j)/2-d/2+2\varepsilon}) (1\wedge s'^{-n_i/2-d/4+\varepsilon})(1\wedge s''^{-n_j/2-d/4+\varepsilon})
			\end{align*}
			for $\varepsilon>0$, which gives
			\begin{align*}
			     \sup_{x\in\mathcal{J}} \left|V_1\right|
			     &\lesssim \delta^6 \int_0^{T\delta^{-2}} (1\wedge s^{-n_i/2-n_j/2-d/2+2\varepsilon}) ds \int_0^{T\delta^{-2}} (1\wedge s'^{-n_i/2-d/4+\varepsilon})ds'\\
			     &\quad\quad\cdot\int_0^{T\delta^{-2}} (1\wedge s''^{-n_j/2-d/4+\varepsilon})ds''\\ &\lesssim \delta^6(1\vee \delta^{n_i+n_j+d-2-4\varepsilon})(1\vee\delta^{n_i+d/2-2-2\varepsilon})(1\vee\delta^{n_j+d/2-2-2\varepsilon}).
			\end{align*}
			Without loss of generality let $n_i\leq n_j$. For $\varepsilon$ small enough, we can ensure $\delta^{n_i+n_j+d-2-4\varepsilon}\leq  1$, as $n_i+n_j>2-d$. In $d\leq 2$ only the pairs $(n_i,n_j)\in\{(0,0),(0,1)\}$ are excluded, and in every case the claimed bound holds. The same applies to $d\geq 3$ for all pairs $(n_i,n_j)$.
		\end{proof}
	\end{lemma}

    \subsubsection{Proof of Theorem \ref{thm:clt}}
    
    \begin{proof}
	   We begin with the observed Fisher information. Suppose first $X_0=0$. Under Assumption \ref{assump:mainAssump} we find that $n_i+n_j >2-d$ for all $i,j=1,\dots,p$ in all dimensions $d\geq 1$. It follows from Lemmas \ref{ConvFisher} and \ref{ConvouterVar} that
		\begin{align*}
		      (\rho_{\delta}\mathcal{I}_{\delta}\rho_{\delta})_{ij} &=\delta^{-2+n_i+n_j}M^{-1}\sum_{k=1}^M\int_{0}^{T}\sc{X(t)}{A_i^*K_{\delta,x_k}}\sc{X(t)}{A_j^*K_{\delta,x_k}}dt\\
		    &=T\Psi_{\theta}(A^*_iK,A^*_jK) + o_{\P}(1) = (\Sigma_{\theta})_{ij} + o_{\P}(1).
		\end{align*}
        This yields for $X_0=0$ the wanted convergence $\rho_{\delta} \mathcal{I}_{\delta}\rho_{\delta} \stackrel{\P}\rightarrow \Sigma_{\theta}$. In order to extend this to the general $X_0$ from Assumption \ref{assump:mainAssump}, let $\bar{X}$ be defined as $X$, but starting in $\bar{X}(0)=0$ such that for $v\in L^2(\Lambda)$, $\sc{X(t)}{v}=\sc{\bar{X}(t)}{v}+\sc{S_{\vartheta}(t)X_0}{v}$. If $\bar{\mathcal{I}}_{\delta}$ is the observed Fisher information corresponding to $\bar{X}$, then by the Cauchy-Schwarz inequality, with $v_{i}=\sup_{k}\delta^{2n_i-2}\int_{0}^{T}\sc{X_0}{S^*_{\theta}(t)A^*_i K_{\delta,x_k}}^2 dt$,
		\begin{align*}
			& |(\rho_{\delta}\mathcal{I}_{\delta}\rho_{\delta})_{ij}-(\rho_{\delta}\bar{\mathcal{I}}_{\delta}\rho_{\delta})_{ij}|
		\lesssim (\rho_{\delta}\bar{\mathcal{I}}_{\delta}\rho_{\delta})_{ii}^{1/2}
			v_{j}^{1/2}
			+ (\rho_{\delta}\bar{\mathcal{I}}_{\delta}\rho_{\delta})_{jj}^{1/2}
			v_{i}^{1/2}
			+ 
			v_{i}^{1/2}v_{j}^{1/2}.
		\end{align*}
		By the first part, $(\rho_{\delta}\bar{\mathcal{I}}_{\delta}\rho_{\delta})_{ii}$ is bounded in probability and Assumption \ref{assump:mainAssump}(iv) gives $v_{i}=o_{\P}(1)$ for all $i$. From this obtain again the convergence of the observed Fisher information. Regarding the invertibility of $\Sigma_{\theta}$, let  $\lambda\in\R^p$ such that
    	\begin{align*}
    	    0 = \sum_{i,j=1}^p \lambda_i\lambda_j(\Sigma_{\theta})_{ij} & = T \Psi_{\theta}\Big(\sum_{i=1}^p\lambda_i A^*_{i}K,\sum_{i=1}^p\lambda_iA^*_iK\Big).
    	\end{align*}
    	By the definition of $\Psi_{\theta}$ this implies $e^{t\nabla\cdot a_{\theta}\nabla}(\sum_{i=1}^p\lambda_i A^*_iK)=0$ for all $t\geq 0$ and thus $\sum_{i=1}^p\lambda_i A^*_iK=0$. Since the functions $A^*_iK$ are linearly independent by Assumption \ref{assump:mainAssump}(i), conclude that $\Sigma_{\theta}$ is invertible. 
	
	We proceed next to the proof of the CLT. The augmented MLE and the statement of the limit theorem remain unchanged when $K$ is multiplied by a scalar factor. We can therefore assume without loss of generality that $\norm{K}_{L^2(\R^d)}=1$. By the basic error decomposition \eqref{eq:basicError} and because $\Sigma_{\theta}$ is invertible, this means 
			\begin{align}
				(\rho_{\delta} \mathcal{I}_{\delta}\rho_{\delta})^{1/2}\rho_{\delta}^{-1} (\hat{\theta}_{\delta}-\theta) = (\rho_{\delta}\mathcal{I}_{\delta}\rho_{\delta})^{-1/2} \Sigma_{\theta}^{1/2}(\Sigma_{\theta}^{-1/2}\rho_{\delta}\mathcal{M}_{\delta}).\label{eq:CLT}
			\end{align}			
	Note that $\mathcal{M}_{\delta}=\mathcal{M}_{\delta}(T)$ corresponds to a $p$-dimensional continuous and square integrable martingale $(\mathcal{M}_{\delta}(t))_{0\leq t\leq T}$ with respect to the filtration $(\mathcal{F}_t)_{0\leq t\leq T}$ evaluated at $t=T$. In view of Assumption \ref{assump:mainAssump}(iii) let $\delta\leq\delta'$ such that for $s,t\geq0$ and $k,k'$ with the Kronecker delta $\delta_{k,k'}$
		\begin{equation*}
		    \E[W_k(s)W_{k'}(t)] =(s\wedge t)\sc{K_{\delta,x_k}}{K_{\delta,x_{k'}}}=(s\wedge t)\delta_{k,k'}.
		\end{equation*}
	This means that the Brownian motions $W_k$ and $W_{k'}$ are independent for $k\neq k'$ and thus their quadratic co-variation process at $t$ is $[W_k,W_{k'}]_t=t\delta_{k,k'}$. From this infer that the quadratic co-variation process of the martingale $(\mathcal{M}_{\delta}(t))_{0\leq t\leq T}$ at $t=T$ for $\delta\leq\delta'$ is equal to
	\begin{align*}
	    [\mathcal{M}_{\delta}]_T = \sum_{k,k'=1}^M \int_0^T X^A_{\delta,k}(t)X^A_{\delta,k'}(t)^\top d[W_k,W_{k'}]_t=\mathcal{I}_{\delta}.
	\end{align*}
	Theorem \ref{thm:MartingaleCLT} now implies $\Sigma_{\theta}^{-1/2}\rho_{\delta}\mathcal{M}_{\delta}\stackrel{d}\rightarrow \mathcal{N}(0,I_{p\times p})$. Conclude in \eqref{eq:CLT} by $\rho_{\delta} \mathcal{I}_{\delta}\rho_{\delta} \stackrel{\P}\rightarrow \Sigma_{\theta}$ and Slutsky's lemma.
	\end{proof}

	\subsection{RKHS computations}\label{sec:RKHS:computations}
    
    The proofs of the RKHS results from Section \ref{Sec:RKHS:SPDE} are achieved by basic operations on RKHS, in particular under linear transformation (see, e.g.,
    \cite[Chapter 4]{MR3024389} or \cite[Chapter 12]{MR3967104}).
    
    Recall the stationary process $X$ in \eqref{eq:stochConv} and that $Ae_j=-\lambda_je_j$ with eigenvalues $0<\lambda_1\leq \lambda_2\leq \cdots$ and an orthonormal basis $(e_j)_{j\geq 1}$ of $\mathcal{H}$. The cylindrical Wiener process can be realised as $W=\sum_{j\geq 1}e_j\beta_j$ for independent scalar Brownian motions $\beta_j$ and we obtain 
	\begin{align}
	    X(t)=\sum_{j\geq 1}\int_{-\infty}^t e^{-\lambda_j(t-t')}d\beta_j(t')e_j=\sum_{j\geq 1}Y_j(t)e_j,\label{eq:series_decomp}
	\end{align}
    with independent stationary Ornstein-Uhlenbeck processes $Y_j$ satisfying
    \begin{align*}
        dY_{j}(t) = -\lambda_jY_j(t)dt+d\beta_j(t).
    \end{align*}
    For a sequence $(\mu_j)$ of non-decreasing, positive real numbers, take $\mathcal{H}_1$ to be the closure of $\mathcal{H}$ under the norm 
    \begin{align*}
        \norm{z}^2_{\mathcal{H}_1} = \sum_{j\geq 1}\frac{1}{\mu_j^2}\sc{z}{e_j}^2_{\mathcal{H}},
    \end{align*}
    such that $\mathcal{H}$ is continuously embedded in $\mathcal{H}_1$. If for $0\leq t\leq T$
    \begin{align}\label{eq:existence:SPDE}
        \int_{-\infty}^t\norm{S(t')}^2_{\operatorname{HS}(\mathcal{H},\mathcal{H}_1)}dt' 
        & = \sum_{j\geq 1}\int_{-\infty}^t \norm{S(t')e_j}^2_{\mathcal{H}_1}dt'= \sum_{j\geq 1}\frac{1}{\mu_j^2}\int_{-\infty}^t e^{-2\lambda_jt'}dt'<\infty,
    \end{align}
    then we conclude by \cite[Theorem 5.2]{da_prato_stochastic_2014} that the law of $X$ induces a Gaussian measure on the Hilbert space $L^2([0,T];\mathcal{H}_1)$. A first universal choice is given by $\mu_j=j$ for all $j\geq 1$. Moreover, if $A$ is a second order elliptic differential operator, then Weyl’s law \cite[Lemma 2.3]{shimakura_partial_1992} says that the $\lambda_j$ are positive real numbers of the order $j^{2/d}$, meaning that the choice $\mu_j=\lambda_j^{s/2}$ is possible whenever $s\geq 0$ and $s+1>d/2$. In this case, $\mathcal{H}_1$ corresponds to a Sobolev space of negative order $-s$ induced by the eigensequence $(\lambda_j,e_j)_{j\geq 1}$.

    Let us introduce some background on the RKHS of a centred Gaussian random variable $Z$, defined on a separable Hilbert space $\mathcal{Z}$. Its covariance operator $C_Z$ is necessarily positive self-adjoint and trace-class. This means, by the spectral theorem, there exist strictly positive eigenvalues $(\sigma_j^2)_{j\geq 1}$ and an associated orthonormal system of eigenvectors $(u_j)_{j\geq 1}$ such that $C_Z=\sum_{j\geq 1}\sigma_j^2(u_j\otimes u_j)$. Associate with $Z$ (or rather with the induced centred Gaussian measure) the so-called kernel or RKHS $(H_Z,\|\cdot\|_{Z})$, where 
	\begin{align}\label{eq:RKHS:Hilbert:space}
	    H_Z=\{h\in \mathcal{Z}:\|h\|_{Z}<\infty\},\qquad \|h\|_{Z}^2=\sum_{j\geq 1}\frac{\sc{u_j}{h}_{\mathcal{Z}}^2}{\sigma_j^2}
	\end{align}
	(see, e.g., \cite[Example 4.2]{MR3024389} and also \cite[Chapters 4.1 and 4.3]{MR3024389} and \cite[Chapter 3.6]{gine_mathematical_2016} for other characterizations of the RKHS of a Gaussian measure or process).
	Alternatively, we have $H_Z=C_Z^{1/2}\mathcal{Z}$ and $\|h\|_{Z}=\norm{C_Z^{-1/2}h}_{\mathcal{Z}}$ for $h\in H_Z$. A useful tool to compute the RKHS is the fact that the RKHS behaves well under linear transformation. More precisely, if $L:\mathcal{Z}\rightarrow \mathcal{Z}'$ is a bounded linear operator between Hilbert spaces, then the image $L(Z)$ is a centred Gaussian random variable with RKHS $L(H_Z)$ and norm $\norm{h}_{ L(Z)}=\inf\{\norm{f}_{Z}:f\in L^{-1}h\}$ (see Proposition 4.1 in \cite{MR3024389} and also Chapter 3.6 in \cite{gine_mathematical_2016}).

    \subsubsection{RKHS of an Ornstein-Uhlenbeck process}
    We start by computing the RKHS $(H_{Y_j},\norm{\cdot}_{Y_{j}})$ of the processes $Y_j$. We show that the RKHS is equal to the set $H$ from \mbox{Theorem \ref{thm:upper:bound:RKHS:norm:M>1}}, and therefore independent of $j$, while the corresponding norm depends on $\lambda_j$.
    
    \begin{lemma}\label{lemma:RKHS_norm_OU_process}
       For every $j\geq 1$ we have $H_{Y_j}=H$ and
        \begin{align}\label{eq:RKHS_norm_OU_process}
             \norm{h}^2_{Y_j} 
                &= \lambda_j^2 \norm{h}^2_{L^2([0,T])} + \lambda_j(h^2(T)+h^2(0)) + \norm{h'}^2_{L^2([0,T])}.
        \end{align}
    \end{lemma}
    \begin{proof}
        By Example 4.4 in \cite{MR3024389}, a scalar Brownian motion $(\beta(t))_{0\leq t\leq T}$ starting in zero has RKHS $H_{\beta}=\{h:h(0)=0, h\text{ absolutely continuous, }h,h'\in L^2([0,T])\}$ with norm $\norm{h}_{\beta}^2=\int_0^T(h'(t))^2\,dt$. Moreover, the Brownian motion $(\bar \beta(t))_{0\leq t\leq T}$ with $\bar \beta(t)=X_0+\beta(t)$, where $X_0$ is a standard Gaussian random variable independent of $(\beta(t))_{0\leq t\leq T}$ has RKHS $$H_{\bar \beta}=\{\alpha+h:\alpha\in\R,h\in H_{\beta}\}=H,\quad \norm{h}_{\bar \beta}^2=\int_0^T(h'(t))^2\,dt+h^2(0),$$
        as can be seen from Proposition 4.1 in \cite{MR3024389} or Example 12.28 in \cite{MR3967104}. To compute the RKHS of $Y_j$ we now proceed similarly as in Example 4.8 in \cite{MR3024389}. Define the bounded linear map $L:L^2([0,e^{2\lambda_j T}-1])\rightarrow L^2([0,T])$, $(Lf)(t)=(2\lambda_j)^{-1/2}e^{-\lambda_j t}f(e^{2\lambda_j t}-1)$. Then we have $L\bar \beta=Y_j$ in distribution and $L$ is bijective with inverse $L^{-1}h(s)=\sqrt{2\lambda_j (s+1)}h((2\lambda_j)^{-1}\log (s+1))$, $0\leq s\leq e^{2\lambda_j T}-1$. By Proposition 4.1 in \cite{MR3024389} (see also the discussion after \eqref{eq:RKHS:Hilbert:space}), we conclude that $H_{Y_j}=L(H_{\bar \beta})=L(H)=H$ with
        \begin{align*}
            \|h\|_{Y_j}^2&=\|L^{-1}h\|_{\bar \beta}^2 =\int_{0}^{e^{2\lambda_j T}-1}\Big(\frac{d}{ds}\sqrt{2\lambda_j (s+1)}h\Big(\frac{1}{2\lambda_j}\log (s+1)\Big)\Big)^2\,ds+2\lambda_j h^2(0)\\
            &=\int_0^T(\lambda_j h(t)+h'(t))^2\,dt+2\lambda_j h^2(0)\\           &=\lambda_j^2\int_0^Th^2(t)\,dt+\lambda_j(h^2(T)+h^2(0))+\int_0^T(h'(t))^2\,dt. \hspace*{40mm}\qedhere
        \end{align*}
    \end{proof}

    \subsubsection{RKHS of the SPDE}
    
    We compute next the RKHS of the process $X$. Let us start with the following series representation, which is independent of $\mathcal{H}_1$.
    
    \begin{lemma}\label{lem:RKHS:SPDE}
        The RKHS $(H_{X},\|\cdot\|_{X})$ of the process $X$ in \eqref{eq:stochConv} satisfies $$H_X=\Big\{h=\sum_{j\geq 1}h_je_j:h_j\in H,\|h\|_{X}<\infty\Big\}\quad\text{and}\quad \|h\|_{X}^2=\sum_{j\geq 1}\|h_j\|_{Y_j}^2.$$
    \end{lemma}
  
    Note that $h\in L^2([0,T];\mathcal{H})$ if and only if $h=\sum_{j\geq 1}h_je_j$ with $h_j\in L^2([0,T])$ and $\sum_{j\geq 1}\|h_j\|_{L^2([0,T])}^2<\infty$. In this case we have $h_j=\langle h,e_j\rangle$ for all $j\geq 1$. Moreover, since the $\lambda_j$ are bounded from below by a positive constant, we conclude that $H_X$ is indeed a subspace of $L^2([0,T];\mathcal{H})$.
    
    \begin{proof}[Proof of Lemma \ref{lem:RKHS:SPDE}]
        Choose $\mu_j=j$ for all $j\geq 1$. Then $X=\sum_{j\geq 1}j^{-1}Y_j\tilde e_j$ with orthonormal basis $\tilde e_j=je_j$ of $\mathcal{H}_1$ and the covariance operator $C_X$ of $X$ is given by
        \begin{align*}
            &C_X:L^2([0,T];\mathcal{H}_1)\rightarrow L^2([0,T];\mathcal{H}_1),\quad\sum_{j\geq 1}f_j\tilde e_j\mapsto \sum_{j\geq 1}j^{-2}(C_{Y_j}f_j)\tilde e_j.
        \end{align*}        
        with $C_{Y_j}:L^2([0,T])\rightarrow L^2([0,T])$ being the covariance operator of $Y_j$. Hence, using the definition of the RKHS given after \eqref{eq:RKHS:Hilbert:space}, the RKHS of $X$ consists of all elements of the form
        \begin{align*}
            h=C_X^{1/2}f=\sum_{j\geq 1}j^{-1}(C_{Y_j}^{1/2}f_j)\tilde e_j=\sum_{j\geq 1}(C_{Y_j}^{1/2}f_j) e_j
        \end{align*}
        with $f=\sum_{j\geq 1}f_j\tilde e_j\in L^2([0,T];\mathcal{H}_1)$ and we have
        \begin{align*}
            \|h\|_{X}^2&=\|C_X^{-1/2}h\|_{L^2([0,T];\mathcal{H}_1)}^2=\|f\|_{L^2([0,T];\mathcal{H}_1)}^2=\int_0^T\|f(t)\|_{\mathcal{H}_1}^2dt=\sum_{j\geq 1}\|f_j\|_{L^2([0,T])}^2<\infty.
        \end{align*}
        Using Lemma \ref{lemma:RKHS_norm_OU_process}, we can write $h=\sum_{j\geq 1}h_je_j$ with $h_j=C_{Y_j}^{1/2}f_j\in H$ and $$\|h_j\|_{Y_j}^2=\|C_{Y_{j}}^{-1/2}h_j\|_{L^2([0,T])}^2=\|f_j\|_{L^2([0,T])}^2.$$
        Inserting this above, the claim follows.
    \end{proof}
    
    \begin{proof}[Proof of Theorem \ref{thm:RKHS:SPDE}] 
    We first show 
    \begin{align}\label{eq:variation:H:X}
        H_X=\Big\{h=\sum_{j\geq 1}h_je_j:h_j\in H,\sum_{j\geq 1}(\lambda_j^2 \norm{h_j}^2_{L^2([0,T])}+\norm{h'_j}^2_{L^2([0,T])})<\infty\Big\},
    \end{align}
    meaning that the middle term in the squared RKHS norm $\|\cdot\|_{Y_j}^2$ can be dropped. By the calculus rules for Sobolev functions (cf.~\cite[Theorem 4.4]{MR3409135}), we have
    \begin{align*}
        \forall s,u\in[0,T]:\quad h_j^2(s)-h_j^2(u)=2\int_{u}^sh_j'(t)h_j(t)\,dt.
    \end{align*}
    Fix $j\geq 1$ for the moment and choose $t_0\in[0,T]$ such that $h_j^2(t_0)=T^{-1}\norm{h_j}^2_{L^2([0,T])}$. Then 
    \begin{align*}
        \lambda_j(h_j^2(T)+h_j^2(0))&=2\Big(\int_{t_0}^0+\int_{t_0}^T\Big)h_j'(t)h_j(t)\,dt+2\lambda_j\norm{h_j}^2_{L^2([0,T])}\\
        &\leq 2\lambda_j\norm{h_j'}_{L^2([0,T])}\norm{h_j}_{L^2([0,T])}+2\lambda_j\norm{h_j}^2_{L^2([0,T])}\\
        &\leq 2\lambda_j^2\norm{h_j}^2_{L^2([0,T])}+\norm{h_j'}^2_{L^2([0,T])}+\norm{h_j}^2_{L^2([0,T])},
    \end{align*}
    where we also used the Cauchy-Schwarz inequality, the fact that $T\geq 1$, and the inequality $2ab\leq a^2+b^2$, $a,b\in\mathbb{R}$. Summing over $j\geq 1$, we get
    \begin{align}\label{eq:middle:term:upper:bound}
        \sum_{j\geq 1}\lambda_j(h_j^2(T)+h_j^2(0))\leq \sum_{j\geq 1}(2\lambda_j^2 \norm{h_j}^2_{L^2([0,T])}+\norm{h'_j}^2_{L^2([0,T])}+\norm{h_j}^2_{L^2([0,T])}),
    \end{align}
    from which \eqref{eq:variation:H:X} follows, taking into account the discussion after Lemma \ref{lem:RKHS:SPDE}.

    Let us now write
        \begin{align*}
	        \tilde H_X 
	            & = \{h\in L^2([0,T];\mathcal{H}):A h,h'\in L^2([0,T];\mathcal{H})\}
	    \end{align*}
	    and 
	    \begin{align*}
	        \|h\|_{\tilde{H}_X}^2 
                &=\|A h\|_{L^2([0,T];\mathcal{H})}^2
                +\|h'\|_{L^2([0,T];\mathcal{H})}^2+\langle-Ah(0), h(0)\rangle_{\mathcal{H}}+\langle-Ah(T), h(T)\rangle_{\mathcal{H}}.
	    \end{align*}
	By the RKHS computations in Lemma \ref{lem:RKHS:SPDE} it remains to check that $H_X=\tilde{H}_X$ and $\norm{h}_X=\norm{h}_{\tilde H_X}$ for all $h\in H_X$. First, let $h=\sum_{j\geq 1}h_je_j\in H_X$. Then $h$ is absolutely continuous with
	    \begin{align}
	        h'&=\sum_{j\geq 1}h_j'e_j\in L^2([0,T];\mathcal{H}),\quad
	        Ah=-\sum_{j\geq 1}\lambda_jh_je_j\in L^2([0,T];\mathcal{H}).\label{eq:characterization:RKHS:SPDE}
	    \end{align}
	    Hence, $h\in\tilde H_X$ and therefore $H_X\subseteq\tilde{H}_X$. To see the second claim in \eqref{eq:characterization:RKHS:SPDE}, set $h^{(m)}=\sum_{j= 1}^mh_je_j$ and $g^{(m)}=-\sum_{j= 1}^m\lambda_jh_je_j$ for $m\geq 1$. Then, $h^{(m)}(t)$ and $g^{(m)}(t)$ are in $\mathcal{H}$ for all $t\in[0,T]$ and we have $Ah^{(m)}=g^{(m)}$ because $(\lambda_j,e_j)_{j\geq 1}$ is an eigensequence of $-A$. Moreover $h^{(m)}(t)\rightarrow h(t)$ and $Ah^{(m)}(t)=g^{(m)}(t)\rightarrow g(t)=\sum_{j\geq 1}\lambda_jh_j(t)e_j$ for a.e.~$t$. Since $A$ is closed, we conclude that $Ah(t)=g(t)$ for a.e.~$t$.
     
     Next, let $h\in\tilde H_X$. Then we can write $h=\sum_{j\geq 1}h_je_j$ with $h_j=\langle h,e_j\rangle\in L^2([0,T])$. Using also \cite[Proposition A.22]{liu_stochastic_2015}, the $h_j$ are absolutely continuous with $h_j'=\sc{h}{e_j}'=\sc{h'}{e_j}\in L^2([0,T])$. Hence, $h_j\in H$ for all $j\geq 1$. Moreover, the relations in \eqref{eq:characterization:RKHS:SPDE} continue to hold, as can be seen from the identities $\sc{Ah(t)}{e_j}=\lambda_jh_j(t)$ and $\sc{h'}{e_j}=h_j'$, and we have
     \begin{align}
         \|A h\|_{L^2([0,T];\mathcal{H})}^2&=\sum_{j\geq 1}\lambda_j^2 \norm{h_j}^2_{L^2([0,T])},\quad
         \|h'\|_{L^2([0,T];\mathcal{H})}^2=\sum_{j\geq 1}\norm{h'_j}^2_{L^2([0,T])}.\label{eq:norm:series:closed}
     \end{align}
     Hence, $h\in H_X$ and therefore also $\tilde H_X\subseteq H_X$. We conclude that $\tilde H_X= H_X$ and that the norms coincide, where the latter follows from \eqref{eq:norm:series:closed} and \eqref{eq:characterization:RKHS:SPDE}. Moreover, inserting 
     \begin{align*}&
     \langle-Ah(0), h(0)\rangle_{\mathcal{H}}+\langle-Ah(T), h(T)\rangle_{\mathcal{H}}\\
     &\qquad \leq  2\norm{Ah}^2_{L^2([0,T];\mathcal{H})} +
     \norm{h'}^2_{L^2([0,T];\mathcal{H})}
    +\norm{h}^2_{L^2([0,T];\mathcal{H})},
     \end{align*}
     the upper RKHS norm bound follows, as can be seen from \eqref{eq:middle:term:upper:bound}, \eqref{eq:characterization:RKHS:SPDE} and \eqref{eq:norm:series:closed}.
    \end{proof}
	    
     \subsubsection{RKHS of multiple measurements}
    
    In this section we deduce Theorem \ref{thm:upper:bound:RKHS:norm:M>1} from Theorem \ref{thm:RKHS:SPDE}. This requires the $K_1,\dots,K_M$ to lie in the dual space $\mathcal{H}_1'$. When $A=\Delta$ this is a Sobolev space of order $s>d/2-1$ (see the beginning of Section \ref{sec:RKHS:computations}). In Section \ref{app:RKHS:measurements:approximation:argument}, we give a second more technical proof based on an approximation argument, which provides the claim under the weaker assumption $K_1,\dots,K_M\in \mathcal{D}(A)$.

    \begin{proof}[First proof of Theorem~\ref{thm:upper:bound:RKHS:norm:M>1}] 
        For a non-decreasing sequence $(\mu_j)$ of positive real numbers, take
        $$V_\mu=\{f\in \mathcal{H}:\norm{f}_{V_\mu}^2=\sum_{j\geq 1}\mu_j^2\sc{f}{e_j}^2_{\mathcal{H}}<\infty\},$$
        and take $\mathcal{H}_1=V_\mu'$ to be the closure of $\mathcal{H}$ under the norm
        \begin{align*}
        \norm{z}^2_{V_\mu'} = \sum_{j\geq 1}\frac{1}{\mu_j^2}\sc{z}{e_j}^2_{\mathcal{H}}.
        \end{align*}
        Then, $V_\mu$ is continuously embedded in $\mathcal{H}$ and $(V_\mu,\mathcal{H},V_\mu')$ forms a Gelfand triple, i.e., $\mathcal{H}$ is identified with its dual and thus $\mathcal{H}$ is also continuously embedded in $ V_\mu'$. Moreover, we can extend $\sc{f}{g}=\sc{f}{g}_{\mathcal{H}}$ to pairs $f\in V_\mu$ and $g\in V_\mu'$ and we have the (generalised) Cauchy-Schwarz inequality
        \begin{align}\label{eq:generalized:CS}
            |\sc{f}{g}|\leq \norm{f}_{V_\mu}\norm{g}_{V_\mu'}.
        \end{align}
        We choose the sequence $(\mu_j)$ such that \eqref{eq:existence:SPDE} holds, meaning that $X$ can be considered as a Gaussian random variable in $L^2([0,T];V_\mu')$. For $K_1,\dots,K_M\in V_\mu$, consider the linear map
        \begin{align*}
            &L:L^2([0,T];V_\mu')\rightarrow L^2([0,T])^M,\quad Lf(t)= (\sc{K_k}{f(t)})_{k=1}^M, t\in[0,T].
        \end{align*}
        Then, $LX=X_K$ in distribution. Using \eqref{eq:generalized:CS}, it is easy to see that $L$ is a bounded operator with norm bounded by $(\sum_{k=1}^M\norm{K_k}_{V_\mu}^2)^{1/2}$:
        \begin{align*}
            \sum_{k=1}^M\int_0^T\sc{K_k}{f(t)}^2dt&\leq 
            \Big(\sum_{k=1}^M\norm{K_k}_{V_\mu}^2\Big)\norm{f}_{L^2([0,T];V_\mu')}^2.
        \end{align*}
        Next, we show that $L(H_X)=H^M$. First, for $(h_k)_{k=1}^M\in H^M$, the function 
        \begin{align}\label{eq:inverse:image}
        f=\sum_{k,l=1}^M G^{-1}_{k,l} K_kh_l\in H_{X}\quad\text{satisfies}\quad Lf=(h_k)_{k=1}^M.
       \end{align}
       Hence $H^ M \subseteq L(H_X)$. To see the reverse inclusion, let $f\in H_X$. Set $(h_k)_{k=1}^M=Lf$ such that $h_k(t)= \sc{K_k}{f(t)}$. By the definition of $H_X$ and properties of the Bochner integral (see, e.g., \cite[Proposition A.22]{liu_stochastic_2015}), the $h_k$ are absolutely continuous with derivatives $h_k'(t)=\sc{K_k}{f'(t)}$, and we have 
       \begin{align*}
           \int_0^T (h'_k(t))^2dt\leq \norm{K_k}_{\mathcal{H}}^2\int_0^T\norm{f'(t)}_{\mathcal{H}}^2dt=\norm{K_k}_{\mathcal{H}}^2\norm{f'}_{L^2([0,T];\mathcal{H})}^2<\infty.
       \end{align*}
       We get $h_k\in H$ for all $k=1,\dots,M$. Hence, $L(H_X)\subseteq H^M$ and therefore $L(H_{X})=H^M$. It remains to prove the bound for the norm. Using \eqref{eq:inverse:image}, the behavior of the RKHS under linear transformation (see \cite[Proposition 4.1]{MR3024389}) and Theorem \ref{thm:RKHS:SPDE}, we have
       \begin{align*}
           \norm{(h_k)_{k=1}^M}_{X_K}^2\leq \norm{\sum_{k,l=1}^M & G^{-1}_{k,l} K_kh_l}_{X}^2\leq 3\|\sum_{k,l=1}^M G^{-1}_{k,l}AK_kh_l\|_{L^2([0,T];\mathcal{H})}^2\\
           &+\|\sum_{k,l=1}^M G^{-1}_{k,l} K_kh_l\|_{L^2([0,T];\mathcal{H})}^2+2\| \sum_{k,l=1}^M G^{-1}_{k,l}K_kh_l'\|_{L^2([0,T];\mathcal{H})}^2.
       \end{align*}
       Using the definition of $G_A$, the last display becomes
        \begin{align}
            \norm{(h_k)_{k=1}^M}_{X_K}^2 &\leq 3\int_0^T \sum_{k,l=1}^M(G^{-1}G_{A} G^{-1})_{kl}h_k(t)h_l(t)\,dt+\int_0^T  \sum_{k,l=1}^M(G^{-1})_{kl}h_k(t)h_l(t)\,dt\nonumber\\
            &\quad+2\int_0^T  \sum_{k,l=1}^M(G^{-1})_{kl}h_k'(t)h_l'(t)\,dt,\label{eq:upper:bound:RKHS:norm:M>1}
        \end{align}
        and the claim follows from standard results for the operator norm of symmetric matrices.
    \end{proof}

      \begin{proof}[Proof of Corollary \ref{cor:upper:bound:RKHS:norm:M>1_Laplace}]
        Since the Laplace operator $\Delta$ is negative and self-adjoint, the stochastic convolution \eqref{eq:stochConv} is just the weak solution in \eqref{eq:weakSolution} and $\mathcal{H}=L^2(\Lambda)$. If $(K_k)_{k=1}^M=(K_{\delta,x_k})_{k=1}^M$ with $\norm{K}_{L^2(\R^d)}=1$, then $K_1,\dots,K_M$ have disjoint supports and satisfy the assumptions of Theorem \ref{thm:upper:bound:RKHS:norm:M>1} with $G=I_{M\times M}$ and $G_\Delta$ being a diagonal matrix with $(G_\Delta)_{kk}=\norm{\Delta K_{\delta,x_k}}^2_{L^2(\R^d)}$. By construction and the Cauchy-Schwarz inequality, we have $\norm{K_{\delta,x_k}}=1$ and $\norm{\Delta K_{\delta,x_k}}\leq \delta^{-2}\norm{\Delta K}_{L^2(\R^d)}$. From Theorem \ref{thm:upper:bound:RKHS:norm:M>1}, we obtain the RKHS $H_{X_K}=H^M$ of $X_K$ with the claimed upper bound on its norm, where we also used that $\delta^2\leq \norm{\Delta K}_{L^2(\R^d)}$ by assumption. 
    \end{proof}

	\subsection{Proof of the lower bounds}\label{sec:proof:optimality:result}
	
    In this section, we give the main steps of the proof of Theorem \ref{thm:lower:bound:M>1}, which follows a Gaussian route. First, we combine a classical Gaussian lower bound (based on two hypotheses) with arguments from the Feldman-Hajek theorem to formulate a lower bound scheme that is expressed in terms of covariance operators and RKHS norms. Second, we invoke the RKHS computations from Section \ref{Sec:RKHS:SPDE} to further reduce our analysis to $L^2$-distances of the involved (cross-)covariance kernels and their first and second derivatives. Finally, we use semigroup perturbation arguments to compute these distances in the setting of Assumption~\ref{assu:lowerBound}. The proofs of three key lemmas are deferred to the appendix.
 
	\subsubsection{Gaussian minimax lower bounds}\label{sec:general_lower_bound_setting}
	
    Let $(\mathbb{P}_{\vartheta})_{\vartheta\in\Theta}$ be a family of probability measures defined on the same measurable space with a parameter set $\Theta\subseteq \mathbb{R}^p$. For $\theta^0,\theta^1\in\Theta$, the (squared) Hellinger distance between $\mathbb{P}_{\theta^0}$ and $\mathbb{P}_{\theta^1}$ is defined by $H^2(\mathbb{P}_{\vartheta^0},\mathbb{P}_{\theta^1})=\int (\sqrt{\mathbb{P}_{\theta^0}}-\sqrt{\mathbb{P}_{\theta^1}})^2$ (see, e.g. \cite[Definition 2.3]{tsybakov_introduction_2008}). Moreover, if $\theta^0,\vartheta^1\in\Theta$ satisfy
    \begin{align}\label{eq_Hellinger_lower_bound_cond}\
    H^2(\mathbb{P}_{\theta^0},\mathbb{P}_{\vartheta^1})\leq 1,
    \end{align}
    then we have the lower bound
    \begin{align}\label{eq_Hellinger_lower_bound}
    \inf_{\hat\vartheta}\max_{\vartheta\in\{\vartheta^0,\vartheta^1\}}\mathbb{P}_\vartheta\Big(|\hat\vartheta-\vartheta|\geq \frac{|\vartheta^0-\vartheta^1|}{2}\Big)\geq \frac{1}{4}\frac{2-\sqrt{3}}{4}=:c_3,
    \end{align}
    where the infimum is taken over all $\R^p$-valued estimators~$\hat\vartheta$ and $|\cdot|$ denotes the Euclidean norm. For a proof of this lower bound, see~\cite[Theorem~2.2(ii)]{tsybakov_introduction_2008}.
    
    Next, let $\mathbb{P}_{\vartheta^0}$ and $\mathbb{P}_{\vartheta^1}$ be two Gaussian measures defined on a separable Hilbert space $\mathcal{Z}$ with expectation zero and positive self-adjoint trace-class covariance operators $C_{\vartheta^0}$ and $C_{\vartheta^1}$, respectively. By the spectral theorem, there exist strictly positive eigenvalues $(\sigma_j^2)_{j\geq 1}$ and an associated orthonormal system of eigenvectors $(u_j)_{j\geq 1}$ such that $C_{\vartheta^0}=\sum_{j\geq 1}\sigma_j^2(u_j\otimes u_j)$. Given the Gaussian measure $\mathbb{P}_{\vartheta^0}$, we can associate the RKHS $(H_{\theta^0},\|\cdot\|_{H_{\theta^0}})$ of $\mathbb{P}_{\vartheta^0}$ given by $H_{\theta^0}=\{h\in \mathcal{Z}:\|h\|_{H_{\vartheta^0}}<\infty\}$ and $\|h\|_{H_{\vartheta^0}}^2=\sum_{j\geq 1}\sigma_j^{-2}\sc{u_j}{h}_{\mathcal{Z}}^2$ (cf.~the beginning of Section \ref{sec:RKHS:computations}). Combining \eqref{eq_Hellinger_lower_bound_cond} with the RKHS machinery, we get the following lower bound.
	
	\begin{lemma}\label{lem:Gaussian:lower:bound}
	    In the above Gaussian setting, suppose that $(u_j)_{j\geq 1}$ is an orthonormal basis of $\mathcal{Z}$ and that 
	    \begin{align}\label{eq:Gaussian:lower:bound:condition}
	    \sum_{j\geq 1}\sigma_j^{-2}\|(C_{\vartheta^1}-C_{\vartheta^0})u_{j}\|_{H_{\vartheta^0}}^2
	    \leq 1/2.
	    \end{align}
	    Then the lower bound in \eqref{eq_Hellinger_lower_bound} holds, that is
	    \begin{align*}
	        \inf_{\hat\vartheta}\max_{\vartheta\in\{\vartheta^0,\vartheta^1\}}\mathbb{P}_\vartheta\Big(|\hat\vartheta-\vartheta|\geq \frac{|\vartheta^0-\vartheta^1|}{2}\Big)\geq c_3.
	    \end{align*}
	\end{lemma}
	
	Lemma \ref{lem:Gaussian:lower:bound} is a consequence of the proof of the Feldman-Hajek theorem \cite[Theorem 2.25]{da_prato_stochastic_2014} in combination with basic properties of the Hellinger distance and the minimax risk. A proof is given in Section \ref{proof:lem:Gaussian:lower:bound}.
    
    \subsubsection{Proof of Theorem \ref{thm:lower:bound:M>1}}
    
     Our goal is to apply Lemma \ref{lem:Gaussian:lower:bound} and Corollary \ref{cor:upper:bound:RKHS:norm:M>1_Laplace} to the Gaussian process $X_{\delta}$ under Assumption \ref{assu:lowerBound}. We assume without loss of generality that $\norm{K}_{L^2(\R^d)}=1$.
     We choose $\theta^0=(1,0,0)$ and $\theta^1\in\Theta_1\cup\Theta_2\cup\Theta_3$, meaning that the null model is $A_{\vartheta^0}=\Delta$ and the alternatives are  $A_{\theta^1}=\theta^1_1\Delta+\theta^1_2(\nabla \cdot b)+\theta^1_3$ for $\theta^1\in\R^3$, where $\theta^1$ lies in one of the parameter classes $\Theta_1$, $\Theta_2$ or $\Theta_3$. For $\vartheta\in\{\vartheta^0,\vartheta^1\}$, let $\P_{\theta,\delta}$ be the law of $X_{\delta}$ on $\mathcal{Z}=L^2([0,T])^M$, let  $C_{\vartheta,\delta}$ be its covariance operator, and let $(H_{\vartheta,\delta},\norm{\cdot}_{H_{\vartheta,\delta}})$ be the associated RKHS. For $(f_k)_{k=1}^M\in L^2([0,T])^M$, we have $C_{\theta,\delta}(f_k)_{k=1}^M=(\sum_{l=1}^MC_{\theta,\delta,k,l}f_l)_{k=1}^M$ with (cross-)covariance operators defined by
	\begin{align*}
	    &C_{\theta,\delta,k,l}:L^2([0,T])\rightarrow L^2([0,T]),\\ &C_{\theta,\delta,k,l}f_l(t)=\E_{\theta} [\sc{X_{\delta,l}}{f_l}_{L^2([0,T])}X_{\delta,k}(t)],\qquad 0\leq t\leq T
	\end{align*}
	(see also Section \ref{proof:lem:seriesBound:M>1} for more details). By stationarity of $X_\delta$ under Assumption~\ref{assu:lowerBound} 
	\begin{align*}
        C_{\theta,\delta,k,l}f_l(t)
        &=\int_0^t c_{\theta,\delta,k,l}(t-t')f_l(t')\,dt'+\int_t^T c_{\theta,\delta,l,k}(t'-t)f_l(t')\,dt',\quad 0\leq t\leq T
    \end{align*}
    with covariance kernels $c_{\theta,\delta,k,l}(t)= \E_{\theta} [X_{\delta,k}(t)X_{\delta,l}(0)]$, $0\leq t\leq T$.
    Following the notation of Section \ref{sec:general_lower_bound_setting}, let $(\sigma_j^2)_{j\geq 1}$ be the strictly positive eigenvalues of $C_{\vartheta^0,\delta}$ and let $(u_j)_{j\geq 1}$ with $u_j=(u_{j,k})_{k=1}^m\in L^2([0,T])^M$ be a corresponding orthonormal system of eigenvectors. By Corollary \ref{cor:upper:bound:RKHS:norm:M>1_Laplace}, we have $H_{\vartheta^0,\delta}=H^M$ as sets. Since $H^M$ is dense in $L^2([0,T])^M$, $(u_j)_{j\geq 1}$ forms an orthonormal basis of $L^2([0,T])^M$. This means that the first assumption of Lemma \ref{lem:Gaussian:lower:bound} is satisfied. To verify the second assumption in \eqref{eq:Gaussian:lower:bound:condition}, we will use the bound for the RKHS norm in Corollary \ref{cor:upper:bound:RKHS:norm:M>1_Laplace}.
	
	\begin{lemma}\label{lem:seriesBound:M>1}
        In the above setting, we have
        \begin{align*}
            &\sum_{j=1}^\infty \sigma_j^{-2}\norm{(C_{\theta^0,\delta}-C_{\theta^1,\delta}) u_{j}}^2_{H_{\vartheta^0,\delta}}\\
            &\leq cT  \sum_{k,l=1}^M\Big(\frac{\norm{\Delta K}^4_{L^2(\R^d)}}{\delta^{8}}\norm{c_{\theta^0,\delta,k,l}-c_{\theta^1,\delta,k,l}}^2_{L^2([0,T])}+\norm{c_{\theta^0,\delta,k,l}''-c_{\theta^1,\delta,k,l}''}^2_{L^2([0,T])}\Big)
        \end{align*}
        for all $\delta^2\leq \norm{\Delta K}_{L^2(\R^d)}$ and all $T\geq 1$, where $c>0$ is an absolute constant.
    \end{lemma}
    
    The proof of Lemma \ref{lem:seriesBound:M>1} can be found in Section \ref{proof:lem:seriesBound:M>1}. Moreover, combining Lemma \ref{lem:covFun}(ii) with perturbation arguments for semigroups, we prove the following upper bound in Section \ref{proof:lem:concrete_lower_bound}.
    
    \begin{lemma}\label{lem:concrete_lower_bound}
        In the above setting let $\theta^1=(\theta_1,\theta_2,\theta_3)\in\Theta_1\cup\Theta_2\cup\Theta_3$ with $M\geq 1$. Then there exists a constant $c>0$, depending only on $K$ such that
    	\begin{align*}
    	    &\sum_{k, l=1}^M\left(\delta^{-8}\norm{c_{\theta^0,\delta,k,l}-c_{\theta^1,\delta,k,l}}^2_{L^2([0,T])}+\norm{c_{\theta^0,\delta,k,l}''-c_{\theta^1,\delta,k,l}''}^2_{L^2([0,T])}\right)\\
    	    &\qquad\leq cM(\delta^{-2}(1-\theta_1)^2 + \theta_2^2 + \delta^2 \theta_3^2).
    	\end{align*}
    \end{lemma}
    Choosing consecutively 
	\begin{align*}
	    \vartheta^1&=(\theta_1,0,0)\in\Theta_1, \qquad \theta_1=1+c_2\frac{\delta}{\sqrt{TM}},\nonumber\\
	    \vartheta^1&=(1,\theta_2,0)\in\Theta_2, \qquad \theta_2=c_2\frac{1}{\sqrt{TM}},\nonumber\\
	    \vartheta^1&=(1,0,\theta_3)\in\Theta_3, \qquad \theta_3=c_2\min\Big(1,\frac{\delta^{-1}}{\sqrt{TM}}\Big),
	\end{align*}
        Theorem \ref{thm:lower:bound:M>1} follow from Lemma \ref{lem:Gaussian:lower:bound} in combination with Lemmas \ref{lem:seriesBound:M>1} and \ref{lem:concrete_lower_bound}.
	\qed

 \appendix
	
	\section{Additional proofs}\label{app:additional:proofs}
	
	\subsection{Additional proofs from Section \ref{sec:main}}\label{sec:remainingProofs}
	
	The proof of invertibility of the observed Fisher information is classical when the solution process is a multivariate Ornstein-Uhlenbeck process \cite{kuchler2006exponential}, but requires a different proof for the Itô processes $X_{\delta,k}^A$.
	
	 \begin{proof}[Proof of Lemma \ref{lem:FisherInvertible}]
	 It is enough to show that the first summand with $k=1$ in the definition of the observed Fisher information is $\P$-almost surely positive definite. By a density argument we can assume without loss of generality that $K\in C^{\infty}_{c}(\R^d)$. Define a symmetric matrix $\beta\in\R^{p\times p}$, $\beta_{ij}=\sc{A^*_i K_{\delta,x_1}}{ A^*_{j} K_{\delta,x_1}}$ and suppose for $\lambda\in\R^p$ that
	 \begin{equation*}
	     0=\sum_{i,j=1}^p \lambda_i \lambda_{j}\beta_{ij}=\norm{\sum_{i=1}^p\lambda_i A_i^* K_{\delta,x_1}}^2.
	 \end{equation*} 
	 By linear independence this yields $\lambda=0$, and so $\beta$ is invertible. It follows that 
	 \begin{equation*}
	     dX_{\delta,1}^A(t) = (\sc{X(t)}{A^*_{\theta}A^*_iK_{\delta,x_1}})_{i=1}^p dt + \beta^{1/2} d\bar{W}(t)
	 \end{equation*}
	 with a $p$-dimensional Brownian motion $\bar{W}(t)=\beta^{-1/2}(\sc{W(t)}{A_i^* K_{\delta,x_1}})_{i=1}^p$. Then  $Y=\beta^{-1/2}X^A_{\delta,1}$ satisfies $dY(t) = \alpha(t)dt + d\bar{W}(t)$ for some $p$-dimensional Gaussian process $\alpha$. Invertibility of $\int_0^T X_{\delta,1}^A(t)X_{\delta,1}^A(t)^\top dt$ is equivalent to the invertibility of $\int_0^T Y(t)Y(t)^{\top}dt$. Applying first the innovation theorem, cf. \cite[Theorem 7.18]{liptser_statistics_2001}, componentwise and then the Girsanov theorem for multivariate diffusions, this is further equivalent to the $\P$-almost sure invertibility of $\int_0^T \bar{W}(t)\bar{W}(t)^{\top}dt$. The result is now obtained from noting that the determinant of the $p\times p$ dimensional random matrix $(\bar{W}(t_1),\dots,\bar{W}(t_p))$ is $\P$-almost surely not zero for any pairwise different time points $t_1,\dots,t_p$, because $\bar{W}$ has independent increments.
	\end{proof}
	
	\begin{proof}[Proof of Lemma \ref{lem:X0}]
 		 	 It is enough to prove the claim for $A^*_i\in\{1,D_j,D_{jk}\}$ with $n_i\in\{0,1,2\}$.   Let $u_i=\delta^{-n_i}(v_i)_{\delta,x}$ for $v_i=A^*_{i}K$. Suppose first $X_0\in L^p(\Lambda)$. The scaling in Lemma \ref{rescaledsemigroup},  the Hölder inequality and Lemma  \ref{boundS*u} applied to $\delta=1$, $s=0$, yield for $1/p+1/q=1$
 		\begin{align*}
 		   & \sup_{x\in\mathcal{J}} \sc{X_0}{S^*_{\theta}(t)u_i}^2
 		     \lesssim \norm{S_{\theta}(t)X_0}_{L^{p}(\Lambda)}^2\sup_{x\in\mathcal{J}} \norm{u_i}^2_{L^q(\Lambda)}
 		     \lesssim \delta^{d(1-2/p)-2n_i}\norm{X_0}_{L^{p}(\Lambda)}^2,
 		\end{align*}
 		The same Lemmas applied to $s=n_i$ also show for $\varepsilon,\varepsilon'>0$
 		\begin{align*}
 		   & \sup_{x\in\mathcal{J}}\int^T_{\varepsilon'}\sc{X_0}{S^*_{\theta}(t)u_i}^2dt 
 		    \leq \norm{X_0}^2\sup_{x\in\mathcal{J}}\int^T_{\varepsilon'}\norm{S^*_{\theta}(t)u_i}^2 dt \\
 		 & \quad \lesssim \delta^{-2n_i}\sup_{x\in\mathcal{J}}\int^T_{\varepsilon'}\norm{S^*_{\theta,
 		    \delta,x}(t\delta^{-2})v_i}^2_{L^2(\Lambda_{\delta,x})} dt
 		    \lesssim \delta^{-2n_i}\int_{\varepsilon'}^T (t\delta^{-2})^{-n_i-d/2+\varepsilon}dt.
 		\end{align*}	
 	    Assumption \ref{assump:mainAssump}(ii) implies $1-n_k-d/2<0$, and so the last line is of order $O((\varepsilon')^{1-n_i-d/2+\varepsilon}\delta^{d-2\varepsilon})$. After splitting up the integral we conclude
 	    \begin{align*}
 	        \sup_{x\in\mathcal{J}}\int^T_{0}\sc{X_0}{S^*_{\theta}(t)u_i}^2 dt \lesssim \varepsilon' \delta^{d(1-2/p)-2n_i} + (\varepsilon')^{1-n_i-d/2+\varepsilon}\delta^{d-2\varepsilon}.
 	    \end{align*} 
 		Choosing $\varepsilon'=\delta^{\frac{2n_i+2d/p-2\varepsilon}{n_i+d/2-\varepsilon}}$ yields the order $O(\delta^{h(p)-\varepsilon''})$ with the function $h(p)=d(1-2/p)+2(n_i+d/p)/(n_i+d/2)-2n_i$ for any $\varepsilon''>0$.  We get $h(2)=2-2n_i$ and $h'(p)>0$.  From this obtain the claim when $X_0\in L^p(\Lambda)$.  
 		
 		Let now $X_0 = \int_{-\infty}^0 S_{\theta}(-t')dW(t')$ and $c_{\theta} \leq 0$. By Itô's isometry, the $\delta$-scaling and changing variables we get 
 		\begin{align*}
 			\E[\sc{X_0}{S^*_{\theta}(t)u_i}^2] & = \int_0^\infty \norm{S^*_{\theta}(t'+t)u_i}^2dt'= \delta^{2-2n_i}\int_0^\infty \norm{S_{\theta,\delta,x}^*(t'+t\delta^{-2})v_i}^2_{L^2(\Lambda_{\delta,x})} dt'.
 		\end{align*}
 		By Lemma  \ref{boundS*u} and $\tilde{c}_{\theta}\leq 0$  the integral is uniformly bounded in $x\in\mathcal{J}$ and $0\leq t\leq T$ and converges to zero by dominated convergence, because the integrand does so as $\delta\rightarrow 0$. From this obtain the claim in the stationary case.
 		\end{proof}
	
	\begin{theorem}
		\label{thm:MartingaleCLT}
		Let $M_{\delta}=(M_{\delta}(t))_{t\geq 0}$ be a family of continuous $p$-dimensional square integrable martingales with respect to the filtered probability space $(\Omega,\mathcal{F},(\mathcal{F}_t),\mathbb{P})$, with $M_{\delta}(0)=0$ and with quadratic covariation processes $([M_{\delta}]_{t})_{t\geq 0}$. If $T>0$ is such that 
	\[
		[M_{\delta}]_T \stackrel{\P}{\rightarrow} I_{p\times p}, \quad \delta \rightarrow 0,
	\]		
	then we have the convergence in distribution
	\[
		M_{\delta}(T)\stackrel{d}{\rightarrow}\mathcal{N}(0,I_{p\times p}),\quad \delta \rightarrow 0.
	\]
	\end{theorem}
		\begin{proof}
			For $x\in\mathbb{R}^p$ the process $Y_{\delta}(t)=x^{\top}M_{\delta}(t)x$ defines a one dimensional continuous martingale with respect to $(\mathcal{F}_t)$ with $Y_{\delta}(0)=0$ and with quadratic variation $$[Y_{\delta}]_T \stackrel{\mathbb{P}}{\rightarrow} x^{\top}x,\quad \delta \rightarrow 0.$$An application of the Dambis-Dubins-Schwarz theorem (\cite[Theorem 3.4.6]{karatzas_brownian_1998}) shows $Y_{\delta}(t)=w_{\delta}([Y_{\delta}]_t)$ with scalar Brownian motions $(w_{\delta}(t))_{t \geq 0}$, which are possibly defined on an extension of the underlying probability space. From the last display Slutsky's lemma implies the joint weak convergence $(w_{\delta},[Y_{\delta}]_T)\stackrel{d}{\rightarrow}(w_0,x^{\top}x)$ on the product Borel sigma algebra of $C([0,\infty))\times \R$, where $C([0,\infty))$ is endowed with the uniform topology on compact subsets of $[0,\infty)$, and where $w_0$ is another scalar Brownian motion. The continuous mapping theorem with respect to $(f,t)\mapsto\phi(f,t)=f(t)$ yields then the result, noting that $w_0(x^{\top}x)$ has distribution $\mathcal{N}(0,x^{\top}x)$.
		\end{proof}
		
	\subsection{Proof of Proposition \ref{prop: reactiond2}}
	\label{sec: reaction}
	
	\begin{proof}
	Note first that $A_{\theta}=\Delta+\theta$ corresponds to $A_1=1$, $A_0=\Delta$ and with observed Fisher information $\mathcal{I}_{\delta}=\sum_{k=1}^M\int_0^T \sc{X(t)}{K_{\delta,x_k}}^2dt$. In particular, Assumptions \ref{assump:mainAssump}(i), (iii) and (iv) hold. 
	
     As in the proof of Theorem \ref{thm:clt}, we can suppose that $\norm{K}_{L^2(\R^d)}=1$. Recall from the basic decomposition \eqref{eq:basicError}, $	    \hat{\theta}_{\delta} = \theta + \mathcal{I}_{\delta}^{-1} \mathcal{M}_{\delta}$ and from the proof of Theorem \ref{thm:clt} that $\mathcal{M}_{\delta}=\mathcal{M}_{\delta}(T)$ for a square integrable martingale $(\mathcal{M}_{\delta}(t))_{0\leq t\leq T}$, whose quadratic variation at $t=T$ coincides with $\mathcal{I}_{\delta}$. We show below 
	    \begin{align}
	        \log(\delta^{-1})^{-1}\mathcal{I}_{\delta} = O_{\P}(1),\quad (\log(\delta^{-1})^{-1}\mathcal{I}_{\delta})^{-1}=O_{\P}(1).\label{eq:I_log}
	    \end{align} 
	    A well-known result about tail properties of square integrable martingales (e.g., \cite[3.8]{whitt_martingale_FCLT_2007}) therefore implies $\mathcal{M}_\delta=O_{\P}(\log(\delta^{-1})^{1/2})$, and we conclude from the basic decomposition that $\hat{\theta}_{\delta}=\theta + O_{\P}(\log(\delta^{-1})^{-1/2})$ as claimed. 
	    
	    For \eqref{eq:I_log} it is enough to show that $\mathcal{I}_\delta/\E[\mathcal{I}_\delta]\stackrel{\P}{\rightarrow}1$ and $\log(\delta^{-1})\E[\mathcal{I}_\delta]\asymp1$, which in turn holds if for some $c,C>0$, independent of $\delta$,
	    \begin{equation}
	    \label{eq:logbound}
	        c\leq \log(\delta^{-1})^{-1}\E[\mathcal{I}_\delta]\leq C,\quad \log(\delta^{-1})^{-2}\operatorname{Var}(\mathcal{I}_\delta)=o(1).
	    \end{equation}
	    As in the proofs of Lemmas \ref{ConvFisher}, \ref{ConvouterVar} and using their notation we compute
	    \begin{align*}
	       \E[\mathcal{I}_\delta]&\leq M\delta^{2}\int_0^T\int_0^{t\delta^{-2}}\sup_{x\in\mathcal{J}}\norm{S^*_{\theta,\delta,x}(t')K}_{L^2(\Lambda_{\delta,x})}^2dt'dt,\\
	       \operatorname{Var}(\mathcal{I}_\delta)&\lesssim M^2\sup_{x\in\mathcal{J}}\operatorname{Var}\left(\int_0^T\sc{X(t)}{K_{\delta,x}}^2dt\right)\\
	       &=M^2\sup_{x\in\mathcal{J}}4\delta^6\int_0^{T}\int_0^{t\delta^{-2}}\left(\int_0^{t\delta^{-2}-s}f_{1,1}(s,s')ds'\right)^2dsdt.
	    \end{align*}
	    By the supplement in Lemma \ref{boundS*u} we find in $d=2$ that \begin{align}
	       \sup_{x\in\mathcal{J}}\norm{S^*_{\theta,\delta,x}(t)K}_{L^2(\Lambda_{\delta,x})}\lesssim 1\wedge t^{-1/2},\quad t\geq 0, \delta\geq 0.\label{eq:semgroupBound_d2} 
	    \end{align}
	    Plugging this into the last display and using $M\delta^{2}\lesssim 1$ provides us with
	    \begin{align*}
	        \E[\mathcal{I}_\delta]
	        &\lesssim\int_0^T\int_0^{t\delta^{-2}}(1\wedge s^{-1})dsdt
	        \lesssim \int_0^T\log(t\delta^{-2})dt\lesssim \log(\delta^{-1}),\\
	        \operatorname{Var}(\mathcal{I}_\delta)
	        &\lesssim M^2\delta^6\left(\int_0^{T\delta^{-2}}(1\wedge (t')^{-1})dt'\right)\left(\int_0^{T\delta^{-2}}(1\wedge t^{-1/2})dt\right)^2\lesssim\log(\delta^{-1}).
	    \end{align*}
	    We are thus left with showing $ \E[\mathcal{I}_\delta]\gtrsim \log(\delta^{-1})$. First, note that $\norm{S^*_{\theta,\delta,x}(t)K}_{L^2(\R^2)}\geq e^{-T|\theta|}\norm{\tilde{S}_{\theta,\delta,x}(t)K}_{L^2(\R^2)}$ and decompose
	    \begin{align*}
	        &\norm{\tilde{S}_{\theta,\delta,x}(t)K}^2_{L^2(\R^2)} = \sc{\tilde{S}_{\theta,\delta,x}(t)K}{\tilde{S}_{\theta,\delta,x}(t)K}_{L^2(\R^2)}\\
	        &\quad = \norm{e^{t\Delta_0}K}^2_{L^2(\R^2)} + \sc{e^{t\Delta_0}K+\tilde{S}_{\theta,\delta,x}(t)K}{\tilde{S}_{\theta,\delta,x}(t)K-e^{t\Delta_0}K}_{L^2(\R^2)}.
	    \end{align*}
	    Recalling $K\geq 0$, the inner product here is uniformly in $x\in\mathcal{J}$ up to a universal constant upper bounded by
	    \begin{align*}
	        \sc{e^{t\Delta_0}K}{(e^{t\Delta_0}K^2)^{1/2}}_{L^2(\R^2)}(\delta t^{1/2} e^{-\delta^{-2}t^{-1}})^{1/2}\leq \norm{e^{t\Delta_0}K}_{L^2(\R^2)}\norm{e^{t\Delta_0}K^2}_{L^1(\R^2)}\delta t^{1/2},
	    \end{align*}
	    concluding by the Cauchy-Schwarz inequality and $\sup_{x\geq 0}x e^{-x}\lesssim 1$ in the last inequality. Since $\norm{e^{t\Delta_0}K^2}_{L^1(\R^2)}=\norm{K}^2_{L^2(\R^2)}$ and using \eqref{eq:semgroupBound_d2}, it thus follows for some $C>0$ that
	    \begin{align*}
	     &\norm{\tilde{S}_{\theta,\delta,x}(t)K}^2_{L^2(\R^2)} \geq  \norm{e^{t\Delta_0}K}^2_{L^2(\R^2)} - C(1\wedge t^{-1/2})\delta t^{1/2}.
	    \end{align*}
	    Hence, using $M\delta^2\gtrsim 1$,
	    \begin{align*}
	        \E[\mathcal{I}_\delta]
	        &\quad \geq  \int_0^{\delta^{-1}}\norm{e^{t\Delta_0}K}_{L^2(\R^2)}^2dt-\int_0^{\delta^{-1}}C\delta dt=\int_0^{\delta^{-1}}\norm{e^{t\Delta_0}K}_{L^2(\R^2)}^2dt-C.
	    \end{align*}
	    Suppose without loss of generality that the support of $K$ is contained in the unit ball $B_1(0)$. Writing $e^{t\Delta_0}K=q_t*K$ as convolution with the heat kernel $q_t(x)=(4\pi t)^{-1}\exp(-|x|^2/(4t))$ we have
	    \begin{align*}	                    \norm{e^{t\Delta_0}K}_{L^2(\R^2)}^2=\int_{\R^2}\left(\int_{B_1(0)}q_t(y-x)K(x)dx\right)^2dy.
	    \end{align*}
	    The heat kernel $q_t(x)$ is decreasing as $|x|\rightarrow\infty$. Hence, for $x\in B_1(0),$ we bound $q_t(y-x)\geq q_t(y+y/|y|)$ for any $y\in\R^2\setminus\{0\}$. Plugging this into the preceding display yields by $K\geq 0$
	    \begin{align*}
	        \norm{e^{t\Delta_0}K}_{L^2(\R^2)}^2&\geq\int_{\R^2}\left(\int_{B_1(0)}q_t(y+y/|y|)K(x)dx\right)^2dy\\
	        &=\norm{K}_{L^1(\R^2)}^2\int_{\R^2}q_t(y+y/|y|)^2dy
	        \gtrsim t^{-1}.
	    \end{align*}
	    In all, we conclude that $\E[\mathcal{I}_{\delta}]\gtrsim C'\int_0^{\delta^{-1}}t^{-1}dt-C$ for $C'>0$, implying the wanted lower bound in \eqref{eq:logbound}.
	\end{proof}

	\subsection{Proof of Lemma \ref{lem:Gaussian:lower:bound}}\label{proof:lem:Gaussian:lower:bound}
	By definition, we have
	\begin{align*}
	    \sum_{j\geq 1}\sigma_j^{-2}\|(C_{\vartheta^1}-C_{\vartheta^0})u_{j}\|_{H_{\vartheta^0}}^2=\sum_{j,k\geq 1} \sigma^{-2}_j\sigma^{-2}_{k}\sc{u_j}{(C_{\vartheta^1}-C_{\vartheta^0})u_{k}}_{\mathcal{Z}}^2.
	\end{align*}
	Combining this with \eqref{eq:Gaussian:lower:bound:condition} and the fact that $(u_j)_{j\geq 1}$ is an orthonormal basis of $\mathcal{Z}$, the infinite matrix $(\sc{u_j}{(C_{\vartheta^1}-C_{\vartheta^0})u_{k}}_{\mathcal{Z}}/(\sigma_j\sigma_k))_{j,k=1}^\infty$ defines an Hilbert-Schmidt operator $S$ on $\mathcal{Z}$. Let $(v_j,\mu_j)_{j\geq 1}$ be an eigensequence of $S$ with $(v_j)_{j\geq 1}$ being an orthonormal basis of $\mathcal{Z}$. Since $\|S\|_{\operatorname{HS}(\mathcal{Z})}^2\leq 1/2$ by \eqref{eq:Gaussian:lower:bound:condition}, we have $\tau_j:=\mu_j+1\in[1-2^{-1/2},1+2^{-1/2}]$ for all $j\geq 1$. Now, let $\{\xi_j\}_{j\geq 1}$ be a sequence of independent standard Gaussian random variables. Then the series 
	\begin{align*}
	    \sum_{j\geq 1}\xi_jC_{\vartheta^0}^{1/2}v_j\qquad\text{and}\qquad\sum_{j\geq 1}\sqrt{\tau_j}\xi_jC_{\vartheta^0}^{1/2}v_j
	\end{align*}
	converge a.s.~and their laws coincide with those of $\mathbb{P}_{\vartheta^0}$ and $\mathbb{P}_{\vartheta^1}$, respectively (see, e.g., \cite[Proof of Theorem 2.25]{da_prato_stochastic_2014} or \cite[Pages 166-167]{KukushAlexander2020GMiH}). By standard properties of the Hellinger distance (see, e.g., Equation (A.4) in \cite{reis_asymptotic_2011}), we have
	\begin{align}
	    H^2\Big(\bigotimes_{j\geq 1}\mathcal{N}(0,1),\bigotimes_{j\geq 1}\mathcal{N}(0,\tau_j)\Big)&\leq \sum_{j\geq 1}H^2(\mathcal{N}(0,1),\mathcal{N}(0,\tau_j))\nonumber\\
	    &\leq 2\sum_{j\geq 1}(\tau_j-1)^2=2\|S\|_{\operatorname{HS}(\mathcal{Z})}^2\leq 1.\label{eq:bound:HD:product:measure}
	\end{align}
	Moreover, defining $Q_{\vartheta^0}=\bigotimes_{j\geq 1}\mathcal{N}(0,1)$, $Q_{\vartheta^1}=\bigotimes_{j\geq 1}\mathcal{N}(0,\tau_j)$ and the measurable map $\mathcal{T}:\R^\infty\rightarrow \mathcal{Z}$ by $\mathcal{T}(\{\alpha_j\})=\sum_{j\geq 1}\alpha_jC_{\vartheta^0}^{1/2}v_j$ if the limit exists and $T(\{\alpha_j\})=0$ otherwise, the image measures satisfy $Q_{\vartheta^0}\circ \mathcal{T}^{-1}=\mathbb{P}_{\vartheta^0}$ and $Q_{\vartheta_1}\circ \mathcal{T}^{-1}=\mathbb{P}_{\vartheta_1}$. Finally, by the transformation formula, the minimax risk in \eqref{eq_Hellinger_lower_bound} can be written as $\inf_{\hat\theta}\max_{\theta\in\{\vartheta^0,\vartheta^1\}}Q_\theta(|\hat\vartheta\circ \mathcal{T}-\vartheta|\geq|\vartheta^0-\vartheta^1|/2)$, where the infimum is taken over all measurable functions from $\mathcal{Z}$ to $\R^p$. Allowing for general estimators depending on the whole coefficent vector in $\R^\infty$, the claim follows from \eqref{eq:bound:HD:product:measure} and \eqref{eq_Hellinger_lower_bound} applied to the product measures $Q_{\vartheta^0}$ and $Q_{\vartheta^1}$.\qed
	
	\subsection{Proof of Lemma \ref{lem:seriesBound:M>1}}\label{proof:lem:seriesBound:M>1}
    Let us recall some simple facts on the space $\mathcal{Z}=L^2([0,T])^M$ and a bounded linear operator $I:\mathcal{Z}\rightarrow \mathcal{Z}$. First, $\mathcal{Z}$ is a Hilbert space equipped with the inner product $\langle (f_k)_{k=1}^M,(g_k)_{k=1}^M\rangle=\sum_{j=1}^M\langle f_j,g_j\rangle_{L^2([0,T])}$. Second, $I$ can be represented by linear operators $I_{k,l}:L^2([0,T])\rightarrow L^2([0,T])$ such that $I(f_k)_{k=1}^M=(\sum_{l=1}^MI_{k,l}f_l)_{k=1}^M$. Finally, $I$ is a Hilbert-Schmidt operator if and only if all $I_{jk}$ are Hilbert-Schmidt operators and we have
    \begin{align*}
    \|I\|_{\operatorname{HS}(\mathcal{Z})}^2=\sum_{k,l=1}^M\|I_{k,l}\|_{\operatorname{HS}(L^2([0,T]))}^2,
    \end{align*}
    where $\norm{\cdot}_{\operatorname{HS}(L^2([0,T]))}$ denotes the Hilbert-Schmidt norm on $L^2([0,T])$. Recall also that $(\sigma_j^2)_{j\geq 1}$ are the strictly positive eigenvalues of $C_{\vartheta^0,\delta}$ and that $(u_j)_{j\geq 1}$ with $u_j=(u_{j,k})_{k=1}^m\in \mathcal{Z}$ is a corresponding orthonormal basis of eigenvectors. We first prove a more general version of Lemma \ref{lem:seriesBound:M>1}.

	\begin{lemma}\label{lem:seriesBound:M>1:version}
        Grant Assumption \ref{assu:lowerBound}. Consider an integral operator $I=(I_{k,l})_{k,l=1}^M:\mathcal{Z}\rightarrow \mathcal{Z}$, $I_{k,l}f(t) = \int_0^t \kappa_{k,l}(t-t')f(t')dt'+\int_t^T \kappa_{l,k}(t'-t)f(t')dt'$ with square integrable and twice continuously differentiable functions $\kappa_{k,l}$ satisfying $\kappa_{k,l}(0)=\kappa_{l,k}(0)$ and $\kappa'_{k,l}(0)=-\kappa'_{l,k}(0)$ for all $1\leq k,l\leq M$. Then we have
        \begin{align*}
            \sum_{j=1}^\infty \sigma_j^{-2}\norm{I u_{j}}^2_{H_{\vartheta^0,\delta}}&\leq  \sum_{k,l=1}^M\left(240T\frac{\norm{\Delta K}^4_{L^2(\R^d)}}{\delta^{8}}\norm{\kappa_{k,l}}^2_{L^2([0,T])}+200T\norm{\kappa_{k,l}''}^2_{L^2([0,T])}\right)
        \end{align*}
        for all $\delta^2\leq \norm{\Delta K}_{L^2(\R^d)}$ and all $T\geq 1$.
    \end{lemma}
    
    \begin{proof}[Proof of Lemma \ref{lem:seriesBound:M>1:version}]
    We divide the proof into the cases of single and multiple measurements.
    
    \paragraph*{Case $M=1$} If $M=1$, then we consider an integral operator $I:L^2([0,T])\rightarrow L^2([0,T])$, $If(t) = \int_0^T \kappa(|t-t'|)f(t')dt'$ with some square integrable and twice continuously differentiable function $\kappa$ satisfying $\kappa'(0)=0$.
        Define the operators 
        \begin{align*}
            I'f(t) &= \int_0^T \operatorname{sign}(t-t')\kappa'(|t-t'|)f(t')dt',\\
            I''f(t) &= \int_0^T \kappa''(|t-t'|)f(t')dt'.
        \end{align*}
        We show first 
        \begin{align}\label{eq:properties:i':I''}
            (If)'(t) = I'f(t),\quad (If)''(t) = I''f(t).
        \end{align}
        Indeed, after splitting up the integral defining $If(t)$ it follows from the chain rule that
        \begin{align*}
            (If)'(t) 
                &= \left(\int_0^t \kappa(t-t')f(t')dt' + \int_t^T \kappa(t'-t)f(t')dt'\right)'\\
                &= \kappa(0)f(t) + \int_0^t\kappa'(t-t')f(t')dt' - \kappa(0)f(t) - \int_t^T \kappa'(t'-t)f(t')dt',\\
            (If)''(t) 
                &= \kappa'(0)f(t) + \int_0^t\kappa''(t-t')f(t')dt' + \kappa'(0)f(t) + \int_t^T \kappa''(t'-t)f(t')dt',
        \end{align*}
        from which \eqref{eq:properties:i':I''} follows by inserting the assumption $\kappa'(0)=0$. Thus, Corollary~\ref{cor:upper:bound:RKHS:norm:M>1_Laplace} (applied with $M=1$) and \eqref{eq:properties:i':I''} yield for all $j\geq 1$
        \begin{align*}
            \norm{I u_{j}}^2_{H_{\vartheta^0,\delta}}
             & \quad \leq  4\delta^{-4} \norm{\Delta K}^2_{L^2(\R^d)} \norm{I u_{j}}^2_{L^2([0,T])} + 2\norm{(I u_{j})'}^2_{L^2([0,T])}\\
              &\quad =  4\delta^{-4} \norm{\Delta K}^2_{L^2(\R^d)} \norm{I u_{j}}^2_{L^2([0,T])} + 2\norm{I' u_{j}}^2_{L^2([0,T])}
        \end{align*}
        for all $\delta^2\leq \norm{\Delta K}_{L^2(\R^d)}$ and all $T\geq 1$. By construction, $I$ is symmetric, while $I'$ is anti-symmetric, implying that 
         $\sc{Iu_j}{u_{j'}}_{L^2([0,T])}^2=\sc{u_j}{Iu_{j'}}_{L^2([0,T])}^2$ and $\sc{I'u_{j}}{u_{j'}}_{L^2([0,T])}^2=\sc{u_{j}}{I'u_{j'}}_{L^2([0,T])}^2$ for all $j,j'\geq 1$. Combining this with Parseval's identity, we get
        \begin{align*}
                \norm{I u_{j}}^2_{H_{\vartheta^0,\delta}}&\leq   \quad\sum_{j'=1}^\infty\big(4\delta^{-4} \norm{\Delta K}^2_{L^2(\R^d)} \sc{u_{j}}{Iu_{j'}}^2_{L^2([0,T])} + 2\sc{u_{j}}{I'u_{j'}}^2_{L^2([0,T])}\big).
        \end{align*}
        Multiplying the right-hand side with $\sigma_{j}^{-2}$ and summing over $j\geq 1$ yields
        \begin{align*}
            \sum_{j'=1}^\infty\big(4\delta^{-4} \norm{\Delta K}^2_{L^2(\R^d)} \norm{Iu_{j'}}^2_{H_{\vartheta^0,\delta}} + 2\norm{I'u_{j'}}^2_{H_{\vartheta^0,\delta}}\big),
        \end{align*}
        as can be seen from \eqref{eq:RKHS:Hilbert:space}.
        Applying again Corollary~\ref{cor:upper:bound:RKHS:norm:M>1_Laplace} and the definition of the Hilbert-Schmidt norm, we arrive at
        \begin{align*}
            \sum_{j=1}^\infty \sigma_j^{-2}\norm{Iu_j}^2_{H_{\vartheta^0,\delta}} 
            &\leq 16\frac{\norm{\Delta K}^4_{L^2(\R^d)}}{\delta^{8}} \norm{I}^2_{\operatorname{HS}(L^2([0,T]))}\\
            &+ 16\frac{\norm{\Delta K}^2_{L^2(\R^d)}}{\delta^{4}} \norm{I'}^2_{\operatorname{HS}(L^2([0,T]))} + 4\norm{I''}^2_{\operatorname{HS}(L^2([0,T]))}
        \end{align*}
        for all $\delta^2\leq \norm{\Delta K}_{L^2(\R^d)}$ and all $T\geq 1$. Inserting
	    \begin{align}
	    &\norm{I}^2_{\operatorname{HS}(L^2([0,T]))} 
	        = \int_0^T\int_0^T \kappa^2(|t-t'|)dtdt' \leq 2T\norm{\kappa}^2_{L^2([0,T])},\nonumber\\
	        &\norm{I'}^2_{\operatorname{HS}(L^2([0,T]))}\leq 2T \norm{\kappa'}^2_{L^2([0,T])},\quad  \norm{I''}^2_{\operatorname{HS}(L^2([0,T]))}\leq 2T \norm{\kappa''}^2_{L^2([0,T])},\label{eq:HS:l2:kernel}
	\end{align}
	we get
        \begin{align}
            \sum_{j}^\infty \sigma_j^{-2}\norm{Iu_j}^2_{H_{\vartheta^0,\delta}} 
            &\leq 32T\frac{\norm{\Delta K}^4_{L^2(\R^d)}}{\delta^{8}} \norm{\kappa}^2_{L^2([0,T])}\nonumber\\
            &+ 32T\frac{\norm{\Delta K}^2_{L^2(\R^d)}}{\delta^{4}} \norm{\kappa'}^2_{L^2([0,T])} + 8T\norm{\kappa''}^2_{L^2([0,T])}\label{eq:lem:seriesBound}
        \end{align}
	for all $\delta^2\leq \norm{\Delta K}_{L^2(\R^d)}$ and all $T\geq 1$.
	The claim now follows from an interpolation inequality (see, e.g., \cite{MR2759829}). To get precise constants with respect to $T$, we give a self-contained argument. By partial integration and the fact that $\kappa'(0)=0$, we have
	\begin{align*}
	    \int_0^T(\kappa'(t))^2\,dt=-\int_0^T\kappa''(t)\kappa(t)\,dt+\kappa'(T)\kappa(T).
	\end{align*}
	Let $t_0\in[0,T]$ such that $\kappa(t_0)=T^{-1}\int_0^T\kappa(t)\,dt$. Then, by the Cauchy-Schwarz inequality, we have
	\begin{align*}
	    \kappa^2(T)&=2\int_{t_0}^T\kappa'(t)\kappa(t)\,dt+\Big(T^{-1}\int_0^T\kappa(t)\,dt\Big)^2\\
	    &\leq 2\|\kappa\|_{L^2([0,T])}\|\kappa'\|_{L^2([0,T])}+T^{-1}\|\kappa\|_{L^2([0,T])}^2
	\end{align*}
	and similarly 
	\begin{align*}
	    (\kappa'(T))^2&=(\kappa'(T))^2-(\kappa'(0))^2\\
	    &=\int_0^T 2\kappa''(t)\kappa'(t)\,dt\leq 2\|\kappa'\|_{L^2([0,T])}\|\kappa''\|_{L^2([0,T])}.
	\end{align*}
	Combining these estimates, using also the Cauchy-Schwarz inequality, the fact that $T\geq 1$ and the inequality $2ab\leq \epsilon a^2+\epsilon^{-1}b^2$, $\epsilon>0$, $a,b\in\R$ consecutively with $\epsilon\in\{1/4, 1/\sqrt{2},1\}$, we get
	\begin{align*}
	    \|\kappa'\|_{L^2([0,T])}^2&\leq \|\kappa\|_{L^2([0,T])}\|\kappa''\|_{L^2([0,T])}\\
	    &+2\sqrt{\|\kappa\|_{L^2([0,T])}\|\kappa''\|_{L^2([0,T])}}\|\kappa'\|_{L^2([0,T])}\\
	    &+\sqrt{2}\sqrt{\|\kappa'\|_{L^2([0,T])}\|\kappa''\|_{L^2([0,T])}}\|\kappa\|_{L^2([0,T])}\\
	    &\leq  \|\kappa\|_{L^2([0,T])}\|\kappa''\|_{L^2([0,T])}\\
	    &+4\|\kappa\|_{L^2([0,T])}\|\kappa''\|_{L^2([0,T])}+\|\kappa'\|_{L^2([0,T])}^2/4\\
	    &+\|\kappa'\|_{L^2([0,T])}\|\kappa\|_{L^2([0,T])}/2+\|\kappa\|_{L^2([0,T])}\|\kappa''\|_{L^2([0,T])}\\
	    &\leq 6\|\kappa\|_{L^2([0,T])}\|\kappa''\|_{L^2([0,T])}+\|\kappa'\|_{L^2([0,T])}^2/2+\|\kappa\|_{L^2([0,T])}^2/4
	\end{align*}
	and thus
	\begin{align*}
	    \|\kappa'\|_{L^2([0,T])}^2\leq 12\|\kappa\|_{L^2([0,T])}\|\kappa''\|_{L^2([0,T])}+\|\kappa\|_{L^2([0,T])}^2/2.
	\end{align*}
	Using again the inequality $2ab\leq a^2+b^2$, $a,b\in\R$, we conclude that 
	\begin{align*}
	    &32T\frac{\norm{\Delta K}^2_{L^2(\R^d)}}{\delta^{4}} \norm{\kappa'}^2_{L^2([0,T])}\\
	    &\leq 192T\frac{\norm{\Delta K}^4_{L^2(\R^d)}}{\delta^{8}}\|\kappa\|_{L^2([0,T])}^2+192T \norm{\kappa''}^2_{L^2([0,T])}+16T\frac{\norm{\Delta K}^2_{L^2(\R^d)}}{\delta^{4}} \norm{\kappa}^2_{L^2([0,T])}\\
	    &\leq 208T\frac{\norm{\Delta K}^4_{L^2(\R^d)}}{\delta^{8}}\|\kappa\|_{L^2([0,T])}^2+192T \norm{\kappa''}^2_{L^2([0,T])}
	\end{align*}
	where we used the inequality $\delta^2\leq \norm{\Delta K}_{L^2(\R^d)}$ in the last step.
	Inserting this into \eqref{eq:lem:seriesBound}, the claim follows.
    
    \paragraph*{Case $M>1$} We now extend the result to the general case $M>1$. Define the operators $I'=(I'_{k,l})_{k,l=1}^M$ and $I''=(I''_{k,l})_{k,l=1}^M$ by 
        \begin{align*}
            I_{k,l}'f(t) &= \int_0^t\kappa_{k,l}'(t-t')f(t')dt'-\int_t^T\kappa_{l,k}'(t'-t)f(t')dt',\\
            I_{k,l}''f(t) &=\int_0^t\kappa_{k,l}''(t-t')f(t')dt'+\int_t^T\kappa_{l,k}''(t'-t)f(t')dt'.
        \end{align*}
        Using that $\kappa_{k,l}(0)=\kappa_{l,k}(0)$ and $\kappa'_{k,l}(0)=-\kappa'_{l,k}(0)$, we have
        \begin{align*}
            (I_{k,l}f_l)'(t) = I_{k,l}'f_l(t),\quad (I_{k,l}f_l)''(t) = I_{k,l}''f_l(t),
        \end{align*}
        as can be seen by proceeding similarly as in the case $M=1$. Hence, we get 
        \begin{align*}
            (I(f_k)_{k=1}^M)'=I'(f_k)_{k=1}^M\quad\text{ and }\quad(I(f_k)_{k=1}^M)''=I''(f_k)_{k=1}^M.
        \end{align*}
        Thus, Corollary~\ref{cor:upper:bound:RKHS:norm:M>1_Laplace} again yields for all $j\geq 1$
        \begin{align*}
           \norm{I (u_{j,k})_{k=1}^M}^2_{H_{\vartheta^0,\delta}} 
             &  \leq 4\delta^{-4} \norm{\Delta K}^2_{L^2(\R^d)} \norm{I (u_{j,k})_{k=1}^M}^2_{L^2([0,T])^M} + 2\norm{I' (u_{j,k})_{k=1}^M}^2_{L^2([0,T])^M}.
        \end{align*}
        Next, by construction, we have $(I_{k,l})^*=I_{l,k}$ and $(I'_{k,l})^*=-I'_{l,k}$, implying that $I$ is symmetric, while $I'$ is anti-symmetric. Combining this with Parseval's identity, we get
        \begin{align*}
            \norm{I (u_{j,k})_{k=1}^M}^2_{H_{\vartheta^0,\delta}}&\leq 4\delta^{-4} \norm{\Delta K}^2_{L^2(\R^d)}\sum_{j'=1}^\infty \sc{(u_{j,k})_{k=1}^M}{I(u_{j',k})_{k=1}^M}^2_{L^2([0,T])^M}\\
             & +2\sum_{j'=1}^\infty \sc{(u_{j,k})_{k=1}^M}{I'(u_{j',k})_{k=1}^M}^2_{L^2([0,T])^M}.
        \end{align*}
        Multiplying this with $\sigma_{j}^{-2}$ and summing over $j$ yields
        \begin{align*}
            &\sum_{j=1}^\infty \sigma_j^{-2}\norm{I (u_{j,k})_{k=1}^M}^2_{H_{\vartheta^0,\delta}}\\
            &\leq \sum_{j'=1}^\infty\big(4\delta^{-4} \norm{\Delta K}^2_{L^2(\R^d)} \norm{I(u_{j',k})_{k=1}^M}^2_{H_{\vartheta^0,\delta}} + 2\norm{I'(u_{j',k})_{k=1}^M}^2_{H_{\vartheta^0,\delta}}\big).
        \end{align*}
        Applying again Theorem \ref{thm:upper:bound:RKHS:norm:M>1}, we arrive at
         \begin{align*}
            \sum_{j=1}^\infty \sigma_j^{-2}\norm{I (u_{j,k})_{k=1}^M}^2_{H_{\vartheta^0,\delta}}&\leq 16\frac{\norm{\Delta K}^4_{L^2(\R^d)}}{\delta^{8}} \norm{I}^2_{\operatorname{HS}(L^2([0,T])^M)}\\
            &+ 16\frac{\norm{\Delta K}^2_{L^2(\R^d)}}{\delta^{4}} \norm{I'}^2_{\operatorname{HS}(L^2([0,T])^M)} + 4\norm{I''}^2_{\operatorname{HS}(L^2([0,T])^M)}.
        \end{align*}
        Inserting
        \begin{align*}
	    \norm{I^{(j)}}^2_{\operatorname{HS}(L^2([0,T])^M)}=\sum_{k,l=1}^M\norm{I_{k,l}^{(j)}}_{\operatorname{HS}(L^2([0,T]))}^2,\qquad j=0,1,2,
	    \end{align*}
	    with
	    \begin{align*}
	    \norm{I_{k,k}^{(j)}}_{\operatorname{HS}(L^2([0,T]))}^2=\int_0^T\int_0^T (\kappa_{k,k}^{(j)}(|t-t'|))^2dtdt' \leq 2T\norm{\kappa_{k,k}^{(j)}}^2_{L^2([0,T])}
	    \end{align*}
        for all $k=1,\dots,M$ and
	    \begin{align*}
	    &\norm{I_{k,l}^{(j)}}_{\operatorname{HS}(L^2([0,T]))}^2+\norm{I_{l,k}^{(j)}}_{\operatorname{HS}(L^2([0,T]))}^2\\
	    &=\int_{0}^T\int_0^t (\kappa_{k,l}^{(j)}(t-t'))^2\,dt'dt+\int_{0}^T\int_t^T (\kappa_{l,k}^{(j)}(t'-t))^2\,dt'dt\\
	    &+\int_{0}^T\int_0^t (\kappa_{l,k}^{(j)}(t-t'))^2\,dt'dt+\int_{0}^T\int_t^T (\kappa_{k,l}^{(j)}(t'-t))^2\,dt'dt\\
	    &=\int_{0}^T\int_0^T (\kappa_{k,l}^{(j)}(|t-t'|))^2\,dt'dt+\int_{0}^T\int_0^T (\kappa_{l,k}^{(j)}(|t'-t|))^2\,dt'dt\\
	    &\leq 2T\norm{\kappa_{k,l}^{(j)}}^2_{L^2([0,T])}+2T\norm{\kappa_{l,k}^{(j)}}^2_{L^2([0,T])}
	    \end{align*}
	    for all $1\leq k\neq l\leq M$ (here, we used $I_{l,k}^{(0)}=I_{l,k}$, $I_{l,k}^{(1)}=I_{l,k}'$ and $I_{l,k}^{(2)}=I_{l,k}''$ and similar notation for the derivatives of $\kappa$), we arrive at
        \begin{align*}
            \sum_{j=1}^\infty \frac{\norm{I u_{j}}^2_{H_{\vartheta^0,\delta}}}{\sigma_j^{2}}&\leq  32T\frac{\norm{\Delta K}^4_{L^2(\R^d)}}{\delta^{8}}\sum_{k,l=1}^M\norm{\kappa_{k,l}}^2_{L^2([0,T])}\\
            &+32T\frac{\norm{\Delta K}^2_{L^2(\R^d)}}{\delta^{4}}\sum_{k,l=1}^M\norm{\kappa_{k,l}'}^2_{L^2([0,T])}+8T\sum_{k,l=1}^M\norm{\kappa_{k,l}''}^2_{L^2([0,T])}
        \end{align*}
        for all $\delta^2\leq \norm{\Delta K}_{L^2(\R^d)}$ and all $T\geq 1$. The claim now follows as in the case $M=1$ by an interpolation argument.
        \end{proof}
        
        \begin{proof}[Proof of Lemma \ref{lem:seriesBound:M>1}]
        We apply Lemma \ref{lem:seriesBound:M>1:version} to $I=C_{\vartheta^0,\delta}-C_{\vartheta^1,\delta}:\mathcal{Z}\rightarrow \mathcal{Z}$. This means that  $I=(I_{k,l})_{k,l=1}^M$ with $I_{k,l}f(t) = \int_0^t \kappa_{k,l}(t'-t)f(t')dt'+\int_t^T \kappa_{l,k}(t-t')f(t')dt'$ for the integral kernels $\kappa_{k,l}(t)=c_{\vartheta^0,\delta,k,l}(t)-c_{\vartheta^1,\delta,k,l}(t)$, $0\leq t\leq T$,  with 
        \begin{align}
    	c_{\theta,\delta,k,l}(t)= \E_{\theta} [X_{\delta,k}(t)X_{\delta,l}(0)] = \int_{0}^\infty \sc{S_{\theta}^*(t+t')K_{\delta,x_k}}{S_{\theta}^*(t')K_{\delta,x_l}}dt',\label{eq:cov_kernel}
        \end{align}
        where the last equality follows from Lemma \ref{lem:covFun}(ii). From \eqref{eq:cov_kernel}, we immediately infer $\kappa_{k,l}(0)=\kappa_{l,k}(0)$. The first two derivatives of the cross-covariance integral kernels for $\theta\in\{\vartheta^0,\vartheta^1\}$ are
    \begin{align}
        c'_{\theta,\delta,k,l}(t)&=\int_0^\infty\sc{A^*_\theta S^*_\theta(t+t')K_{\delta,x_k}}{S^*_\theta(t')K_{\delta,x_l}}dt',\label{eq:c_prime}\\
        c''_{\theta,\delta,k,l}(t)&=\int_0^\infty\sc{(A^*_\theta)^2S^*_\theta(t+t') K_{\delta,x_k}}{S^*_\theta(t')K_{\delta,x_l}}dt'.\label{eq:c_double_prime}
    \end{align}
    We note that
     \begin{align*}
         &c'_{\theta,\delta,k,l}(0)+c'_{\theta,\delta,l,k}(0)\\
         &\quad=\int_0^\infty(d/dt')\sc{S^*_\theta(t') K_{\delta,x_k}}{S^*_\theta(t')K_{\delta,x_l}}dt'=-\sc{K_{\delta,x_k}}{K_{\delta,x_l}}
     \end{align*}
    is independent of $\theta$, and hence, $\kappa'_{k,l}(0)+\kappa'_{l,k}(0)=0$.
    \end{proof}
    
    \subsection{Proof of Lemma \ref{lem:concrete_lower_bound}}\label{proof:lem:concrete_lower_bound}
    It is sufficient to upper bound the $L^2$-norms of the $\kappa_{k,l}$. Indeed, the proof below for this remains valid if $K_{\delta,x_k}$ is replaced by $\delta^{-4}(A_{\theta,\delta,x_k}^2K)_{\delta,x_k}$ and so it yields also the wanted bound on the $L^2$-norm of $\kappa_{k,l}''$, cf.~\eqref{eq:c_double_prime}. As in \eqref{eq:diagonalizable} we have $$S^*_{\vartheta^0}(t)=e^{t\Delta},\quad S^*_{\vartheta^1}(t)= U_{\vartheta^1}^{-1}e^{t\theta_1\Delta}U_{\vartheta^1}e^{t\tilde{c}_{\vartheta^1}}$$ with $U_{\vartheta^1}(x)= e^{-(2\theta_1)^{-1}\theta_2b\cdot x}$ and  $\tilde{c}_{\vartheta^1} = \theta_3 - (4\theta_1)^{-1}\theta^2_2\leq 0$. From Lemma \ref{rescaledsemigroup}, we also have  $$S_{\vartheta^0,\delta,x_k}(t)=e^{t\Delta_{\delta,x_k}},\quad S^*_{\vartheta^1,\delta,x_k}(t)=U_{\vartheta^1,\delta,x_k}^{-1}e^{t\theta_1\Delta_{\delta,x_k}}U_{\vartheta^1,\delta,x_k}e^{t\delta^2\tilde{c}_{\vartheta^1}}$$ with $U_{\vartheta^1,\delta,x_k}(x)=U_{\vartheta^1}(x_k+\delta x)$. Note that $$e^{t\theta_1\Delta}=U_{\vartheta^1}(x_k)^{-1}e^{t\theta_1\Delta}U_{\vartheta^1}(x_k).$$ 
    
    To get started, let $1\leq k,l\leq M$ and decompose $\kappa_{k,l}=\sum_{j=1}^6\kappa_{k,l}^{(j)}$ with
    \begin{align*}
        \kappa_{k,l}^{(1)}(t) &= \int_0^\infty\sc{(e^{(t+t')\Delta}-e^{(t+t')\theta_1\Delta}) K_{\delta,x_k}}{e^{t'\Delta}K_{\delta,x_l}}dt',\\ 
        \kappa_{k,l}^{(2)}(t) &= \int_0^\infty\sc{e^{(t+t')\theta_1\Delta} K_{\delta,x_k}}{(e^{t'\Delta}-e^{t'\theta_1\Delta})K_{\delta,x_l}}dt',\\ 
        \kappa_{k,l}^{(3)}(t) &= \int_0^\infty\sc{U_{\vartheta^1}(x_k)^{-1}e^{(t+t')\theta_1\Delta}(U_{\vartheta^1}(x_k)-U_{\vartheta^1}e^{\tilde{c}_{\vartheta^1}(t+t')})K_{\delta,x_k}}{e^{t'\theta_1\Delta}K_{\delta,x_l}}dt',\\ 
        \kappa_{k,l}^{(4)}(t) &= \int_0^\infty\sc{(U_{\vartheta^1}(x_k)^{-1}-U_{\vartheta^1}^{-1})U_{\theta^1}S^*_{\theta^1}(t+t')K_{\delta,x_k}}{e^{t'\theta_1\Delta}K_{\delta,x_l}}dt',\\
        \kappa_{k,l}^{(5)}(t) &= \int_0^\infty\sc{S^*_{\vartheta^1}(t+t')K_{\delta,x_k}}{U_{\vartheta^1}(x_k)^{-1}e^{t'\theta_1\Delta}(U_{\vartheta^1}(x_k)-U_{\vartheta^1}e^{\tilde{c}_{\vartheta^1}t'})K_{\delta,x_l}}dt',\\
         \kappa_{k,l}^{(6)}(t) &= \int_0^\infty\sc{S^*_{\vartheta^1}(t+t')K_{\delta,x_k}}{(U_{\vartheta^1}(x_k)^{-1}-U_{\vartheta^1}^{-1})e^{t'\theta_1\Delta}U_{\vartheta^1}e^{\tilde{c}_{\vartheta^1}t'}K_{\delta,x_l}}dt'.
    \end{align*}
    We only show $\sum_{1\leq k, l\leq M}\norm{\kappa_{k,l}^{(j)}}^2_{L^2([0,T])}\leq c\delta^{8}M(\delta^{-2}(1-\theta_1)^2 + \theta_2^2 + \delta^2 \theta_3^2)$ for $j=1,3,4$. The proof that the same bound holds for $j=2,5,6$ follows from similar arguments and is therefore skipped. Diagonal (i.e., $k=l$) and off-diagonal (i.e., $k\neq l$) terms are treated separately. Set $K_{k,l}=K(\cdot+\delta^{-1}(x_k-x_l))$.
    \paragraph*{Case $j=1$}
     The scaling in Lemma \ref{rescaledsemigroup} and changing variables yield
    \begin{align*}
	    \kappa_{k,l}^{(1)}(t\delta^2)
	        &= \delta^2 \int_0^\infty\sc{(e^{(t+t')\Delta_{\delta,x_k}}-e^{(t+t')\theta_1\Delta_{\delta,x_k}})K}{e^{t'\Delta_{\delta,x_k}}K_{k,l}}_{L^2(\Lambda_{\delta,x_k})}dt'\\
	        &= \delta^2 \sc{\int_0^\infty(e^{(t+2t')\Delta_{\delta,x_k}}-e^{(t\theta_1+(t'(1+\theta_1))\Delta_{\delta,x_k}})dt'\,K}{K_{k,l}}_{L^2(\Lambda_{\delta,x_k})}\\
	        &= \frac{\delta^2}{2}\sc{(e^{t\Delta_{\delta,x_k}}-2(1+\theta_1)^{-1}e^{t\theta_1\Delta_{\delta,x_k}})(-\Delta_{\delta,x_k})^{-1}K}{K_{k,l}}_{L^2(\Lambda_{\delta,x_k})}\\
	        &= \frac{\delta^2}{2}\sc{(e^{t\Delta_{\delta,x_k}}-e^{t\theta_1\Delta_{\delta,x_k}})(-\Delta_{\delta,x_k})^{-1}K}{K_{k,l}}_{L^2(\Lambda_{\delta,x_k})}\\
	        &\quad - \frac{\delta^2(1-\theta_1)}{2(1+\theta_1)}\sc{e^{t\theta_1\Delta_{\delta,x_k}}(-\Delta_{\delta,x_k})^{-1}K}{K_{k,l}}_{L^2(\Lambda_{\delta,x_k})}.
	\end{align*}
	With $$e^{t\theta_1\Delta_{\delta,x_k}}-e^{t\Delta_{\delta,x_k}}=(1-\theta_1)\int_0^t e^{s(\theta_1-1)\Delta_{\delta,x_k}}ds\, e^{t\Delta_{\delta,x_k}}(-\Delta_{\delta,x_k}),$$ as follows from the variation of parameters formula, see \cite[p.~162]{EngNag00}, the identity $e^{t\Delta_{\delta,x_k}}=e^{(t/2)\Delta_{\delta,x_k}}e^{(t/2)\Delta_{\delta,x_k}}$ and the Cauchy-Schwarz inequality, we get
	\begin{align}
	    & \left|\sc{(e^{t\Delta_{\delta,x_k}}-e^{t\theta_1\Delta_{\delta,x_k}})(-\Delta_{\delta,x_k})^{-1}K}{K_{k,l}}_{L^2(\Lambda_{\delta,x_k})}\right|\nonumber\\
	    &= |1-\theta_1| \left|\int_0^t\sc{ e^{s(\theta_1-1)\Delta_{\delta,x_k}}e^{(t/2)\Delta_{\delta,x_k}}K}{e^{(t/2)\Delta_{\delta,x_k}}K_{k,l}}_{L^2(\Lambda_{\delta,x_k})}ds\right|\label{eq:kappa_1}\\
	    &\lesssim |1-\theta_1|t \norm{e^{(t/2)\Delta_{\delta,x_k}}K}_{L^2(\Lambda_{\delta,x_k})}\norm{e^{(t/2)\Delta_{\delta,x_k}}K_{k,l}}_{L^2(\Lambda_{\delta,x_k})}.\nonumber
	\end{align}
	In the same way, and using $K=\Delta^2 \tilde{K}$, \begin{align}
	   &\left|\sc{e^{t\theta_1\Delta_{\delta,x_k}}(-\Delta_{\delta,x_k})^{-1}K}{K_{k,l}}_{L^2(\Lambda_{\delta,x_k})}\right|\label{eq:kappa_2} \\
	    &\quad \leq  \norm{e^{(t/2)\theta_1\Delta_{\delta,x_k}}\Delta\tilde{K}}_{L^2(\Lambda_{\delta,x_k})}\norm{e^{(t/2)\theta_1\Delta_{\delta,x_k}}K_{k,l}}_{L^2(\Lambda_{\delta,x_k})}.\nonumber
	\end{align}
	Lemma \ref{boundS*u} therefore gives $\kappa_{k,l}^{(1)}(t\delta^2)\lesssim \delta^2 |1-\theta_1| (1\wedge t^{-1-d/2+\varepsilon})$ for any $\varepsilon>0$. Changing variables one more time and recalling that $\theta_1\geq 1$ already proves for the sum of diagonal terms  that $\sum_{1\leq k\leq M}\norm{\kappa_{k,k}^{(1)}}^2_{L^2([0,T])}\lesssim \delta^{6}(1-\theta_1)^2M$, and the implied constant depends only on $K$. 
	
	With respect to the off-diagonal terms, by exploring the different supports of $K$ and $K(\cdot+\delta^{-1}(x_k-x_l))$, we can obtain a second bound for $\kappa_{k,l}^{(1)}$. First, Lemma \ref{lem:semigroupProp} gives 
    \begin{align}
        &\sup_{y\in \operatorname{supp}K}\left|(e^{t\Delta_{\delta,x_k}}\Delta^2 \tilde{K}_{k,l})(y)\right| \leq \norm{e^{t\Delta_{\delta,x_k}}\Delta^2 \tilde{K}_{k,l}}_{L^\infty(\Lambda_{\delta,x_k})}\nonumber\\
        &\quad \lesssim \norm{\Delta^2 \tilde{K}_{k,l}}_{L^\infty(\R^d)}\wedge (t^{-2}\norm{\tilde{K}_{k,l}}_{L^\infty(\R^d)})\lesssim 1\wedge t^{-2},\label{eq:kappa_7}
    \end{align}
    while on the other hand Lemma \ref{FeynmanKac}(i) shows
    \begin{align}
	        &\sup_{y\in \operatorname{supp}K}|(e^{t\Delta_{\delta,x_k}} K_{k,l})(y)|\lesssim \sup_{y\in \operatorname{supp}K}|(e^{t\Delta_{0}} |K_{k,l}|)(y)|\nonumber\\
	        &= \sup_{y\in \operatorname{supp}K}\int_{\R^d}(4\pi t)^{-d/2}\exp(-|x-y|^2/(4t))| K_{k,l}(x)|\,dx\nonumber\\
	    &\leq (4\pi t)^{-d/2}e^{-c'\frac{|x_k-x_l|^2}{\delta^2t}}\| K\|_{L^1(\R^d)}
	    \lesssim t^{-d/2}e^{-c'\frac{|x_k-x_l|^2}{\delta^2 t}},\label{eq:kappa_3}
	 \end{align}
	for some $c'>0$. By applying the Hölder inequality and using the results from the last two displays we thus obtain for \eqref{eq:kappa_1} the upper bound
	\begin{align*}
	        &(\theta_1-1)t \norm{K}_{L^1(\R)}  \sup_{0\leq s\leq t,y\in \operatorname{supp}K}|(e^{(t+s(\theta_1-1))\Delta_{\delta,x_k}} K_{k,l})(y)|^{1/2+1/2}\\
	        &\lesssim (\theta_1-1)  \sup_{y\in \operatorname{supp}K}|(e^{t\Delta_{0}} |K_{k,l}|)(y)|^{1/2}\lesssim (\theta_1-1) t^{-d/4}e^{-(c'/2)\frac{|x_k-x_l|^2}{\delta^2 t}}.
	 \end{align*}
	The same upper bound (up to the factor $\theta_1-1$ and with $c'$ instead of $c'/2$) holds in \eqref{eq:kappa_2}. Hence, together with the bound $\kappa_{k,l}^{(1)}(t\delta^2)\lesssim \delta^2 (\theta_1-1) t^{-(1+d/2-\epsilon d/4)(1-\epsilon)^{-1}})$ from above (for sufficiently small $\epsilon$) we get
	\begin{align}
	    |\kappa^{(1)}_{k,l}(t\delta^2)|&\lesssim \delta^2 (\theta_1-1)\min\Big(t^{-(1+d/2-\epsilon d/4)(1-\epsilon)^{-1}}, t^{-d/4}e^{-(c'/2)\frac{|x_k-x_l|^2}{\delta^2 t}}\Big)\nonumber\\
	    &\leq \delta^2 (\theta_1-1)t^{-1-d/2}e^{-\epsilon'\frac{|x_k-x_l|^2}{\delta^2 t}}\label{eq:kappa_8}
	\end{align}
	for $\epsilon'=c'\epsilon/2$,
	where we have used the inequality $\min(a,b)\leq a^{1-\epsilon}b^\epsilon$ valid for $a,b\geq 0$. Applying the bound $$\int_0^{\infty}t^{-p-1}e^{-a/t}dt=a^{-p}\int_0^{\infty}t^{-p-1}e^{-1/t}dt\lesssim a^{-p}$$ to $p=1+d>0$ and $a=2\epsilon'\delta^{-2}|x_k-x_l|^2$ this means
	\begin{align}
	    \int_0^T \kappa_{k,l}^{(1)}(t)^2dt &\lesssim \delta^6(1-\theta_1)^2 \int_0^{\infty} t^{-2-d}e^{-2\epsilon'\frac{|x_k-x_l|^2}{\delta^2 t}}dt\nonumber\\
	    &\lesssim \delta^6(1-\theta_1)^2\frac{\delta^{2+2d}}{|x_k-x_l|^{2+2d}}.\label{eq:kappa_9}
	\end{align}
	Recalling that the $x_k$ are $\delta$-separated we obtain from Lemma \ref{lem:sum:inverse:packing} below that
	\begin{align*} 
	    &\sum_{1\leq k\neq l\leq M}\norm{\kappa_{k,l}^{(1)}}^2_{L^2([0,T])}
	    \lesssim \delta^{6}(1-\theta_1)^2\sum_{k=1}^M\sum_{l=1,l\neq k}^M \frac{\delta^{2+2d}}{|x_k-x_l|^{2+2d}}\lesssim \delta^{6} (1-\theta_1)^2 M.
	\end{align*}
 Together with the bounds for the diagonal terms this yields in all $\sum_{1\leq k, l\leq M}\norm{\kappa_{k,l}^{(1)}}^2_{L^2([0,T])}\leq c\delta^{6}M(1-\theta_1)^2$ for a constant $c$ depending only on~$K$. 
 
 \paragraph*{Case $j=3$}
 We begin again with the scaling from Lemma \ref{rescaledsemigroup} and changing variables such that with the multiplication operators $V_{t,t',\delta}(x)=1-e^{\tilde{c}_{\vartheta^1}\delta^2(t+t')-(2\theta_1)^{-1}\theta_2\delta b\cdot   x}$
    \begin{align}
        &\kappa_{k,l}^{(3)}(t\delta^2) = \delta^2\int_0^\infty\sc{e^{(t+t')\theta_1\Delta_{\delta,x_k}}V_{t,t',\delta}K}{e^{t'\theta_1\Delta_{\delta,x_k}}K_{k,l}}_{L^2(\Lambda_{\delta,x_k})}dt'\nonumber\\
        &\quad= \delta^2\int_0^\infty\sc{e^{(t/2+t')\theta_1\Delta_{\delta,x_k}}V_{t,t',\delta}K}{e^{(t/2+t')\theta_1\Delta_{\delta,x_k}}K_{k,l}}_{L^2(\Lambda_{\delta,x_k})}dt'.\label{eq:kappa_4}
    \end{align}
    Since $K$ is compactly supported and $\tilde{c}_{\theta^1}\leq 0$, $V_{t,t',\delta}$ can be extended to smooth multiplication operators with operator norms bounded by $v_{t,t',\delta}=-\tilde{c}_{\vartheta^1}\delta^2(t+t')+(2\theta_1)^{-1}\delta \theta_2$. Recalling $K=\Delta^2\tilde{K}$,  Lemma \ref{boundS*u} gives for any $\epsilon>0$
    \begin{align}
        & |\kappa_{k,l}^{(3)}(t\delta^2)|\leq  \delta^2 \int_0^\infty \norm{e^{(t/2+t')\theta_1\Delta_{\delta,x_k}}V_{t,t',\delta}K}_{L^2(\Lambda_{\delta,x_k})}\norm{e^{(t/2+t')\theta_1\Delta_{\delta,x_k}}K_{k,l}}_{L^2(\Lambda_{\delta,x_k})}dt'\nonumber\\
        & \quad \leq  \delta^2  \int_0^\infty v_{t,t',\delta}(1\wedge (t+t')^{-4-d/2+\epsilon})dt'\lesssim \delta^3 (-\tilde{c}_{\vartheta^1}\delta+ (2\theta_1)^{-1}\theta_2)(1\wedge t^{-1-d/2})\nonumber\\
        & \quad  \leq \delta^3 (\delta |\theta_3| + \theta_2)(1\wedge t^{-1-d/2}),\label{eq:kappa_5}
    \end{align}
    recalling in the last line that $\theta_1\geq 1$ and $\theta_2\leq 1$.     Changing variables therefore proves for the sum of diagonal terms $\sum_{1\leq k\leq M}\norm{\kappa_{k,k}^{(3)}}^2_{L^2([0,T])}\lesssim \delta^8 (\delta |\theta_3| + \theta_2)^2 M$.
    
    With respect to the off-diagonal terms we have 
    \begin{align*}
        &\kappa_{k,l}^{(3)}(t\delta^2) = \delta^2\int_0^\infty\sc{V_{t,t',\delta}K}{e^{(t+2t')\theta_1\Delta_{\delta,x_k}}K_{k,l}}_{L^2(\Lambda_{\delta,x_k})}dt'.
    \end{align*}
    Write $K=\Delta\bar{K}$ for some compactly supported $\bar{K}$ and note that
    \begin{align*}
        V_{t,t',\delta}K = V_{t,t',\delta}\Delta\bar{K} = \Delta (V_{t,t',\delta}\bar{K})-(\Delta V_{t,t',\delta})\bar{K} - 2\nabla V_{t,t',\delta}\cdot \nabla \bar{K}.
    \end{align*}
    Similar to \eqref{eq:kappa_7} we find from Lemma \ref{lem:semigroupProp} 
    \begin{align*}
        &\sup_{y\in \operatorname{supp}\bar{K}}\left|(\Delta e^{t\Delta_{\delta,x_k}} K_{k,l})(y)\right| \lesssim \norm{\Delta K_{k,l}}_{L^\infty(\R^d)}\wedge (t^{-3}\norm{\tilde{K}_{k,l}}_{L^\infty(\R^d)})\lesssim 1\wedge t^{-3}.
    \end{align*}
    Together with the Hölder inequality and \eqref{eq:kappa_3} this provides us for sufficiently small $\epsilon>0$ with
    \begin{align*}
        &\delta^2\left|\int_0^\infty\sc{\Delta(V_{t,t',\delta}\bar{K})}{e^{(t+2t')\theta_1\Delta_{\delta,x_k}}K_{k,l}}_{L^2(\Lambda_{\delta,x_k})}dt'\right|\\
        &=\delta^2\left|\int_0^\infty\sc{V_{t,t',\delta}\bar{K}}{\Delta e^{(t+2t')\theta_1\Delta_{\delta,x_k}}K_{k,l}}_{L^2(\Lambda_{\delta,x_k})}dt'\right|\\
        & \lesssim \delta^2 \int_0^\infty v_{t,t',\delta}\norm{\bar{K}}_{L^1(\R^d)}\sup_{y\in \operatorname{supp}K}\left|(\Delta e^{(t+2t')\theta_1\Delta_{\delta,x_k}} K_{k,l})(y)\right|dt'\\
        &\lesssim \delta^2 \int_0^\infty v_{t,t',\delta}(1\wedge(t+2t'))^{-3(1-\varepsilon)}\sup_{y\in \operatorname{supp}K}\left|(e^{(t+2t')\theta_1\Delta_{\delta,x_k}}\Delta K_{k,l})(y)\right|^{\epsilon}dt'\\
        & \lesssim \delta^3 (\delta |\theta_3| +\theta_2) e^{-\epsilon c'\frac{|x_k-x_l|^2}{\delta^2 t}} \int_0^{\infty} (1\wedge (t')^{-1-\epsilon d/2})dt'
        \lesssim \delta^3 (\delta |\theta_3| + \theta_2) e^{-\epsilon c'\frac{|x_k-x_l|^2}{\delta^2 t}}.
    \end{align*}
    Next, using $\theta_2\leq 1$, we have $\norm{\Delta V_{t,t',\delta}}_{L^{\infty}(\Lambda_{\delta,x_k})}+\norm{|\nabla V_{t,t',\delta}|}_{L^{\infty}(\Lambda_{\delta,x_k})}\lesssim\delta\theta_2$, and so analogously to the computations in the last display
    \begin{align*}
        &\delta^2\left|\int_0^\infty\sc{(\Delta V_{t,t',\delta})\bar{K}+2\nabla V_{t,t',\delta}\cdot\nabla \bar{K}}{e^{(t+2t')\theta_1\Delta_{\delta,x_k}}K_{k,l}}_{L^2(\Lambda_{\delta,x_k})}dt'\right|\\
        & \lesssim \delta^3 \theta_2 \int_0^\infty (\norm{\bar{K}}_{L^1(\R^d)}+\norm{|\nabla\bar{K}|}_{L^1(\R^d)})\sup_{y\in \operatorname{supp}K}\left|( e^{(t+2t')\theta_1\Delta_{\delta,x_k}} K_{k,l})(y)\right|dt'\\
        &\lesssim \delta^3 \theta_2 e^{-\epsilon c'\frac{|x_k-x_l|^2}{\delta^2 t}}.
    \end{align*}
    In all, this means $|\kappa_{k,l}^{(3)}(t\delta^2)| \lesssim \delta^3 (\delta |\theta_3| + \theta_2) e^{-\epsilon c'\frac{|x_k-x_l|^2}{\delta^2 t}}$. Arguing as for \eqref{eq:kappa_8} and \eqref{eq:kappa_9}, as well as using  \eqref{eq:kappa_5} we conclude that $|\kappa^{(3)}_{k,l}(t\delta^2)|\lesssim \delta^3 (\delta |\theta_3| + \theta_2) t^{-1/2-d/2}e^{-\epsilon' \frac{|x_k-x_l|^2}{\delta^2 t}}$ for some $\epsilon'>0$ and $$\int_0^T \kappa_{k,l}^{(3)}(t)^2dt \lesssim \frac{\delta^{8+2d}(\delta |\theta_3| + \theta_2)^2}{|x_k-x_l|^{2d}}.$$ 
    We thus get for diagonal and off-diagonal terms that  $\sum_{1\leq k, l\leq M}\norm{\kappa_{k,l}^{(3)}}^2_{L^2([0,T])}\leq c\delta^8 (\delta |\theta_3| +  \theta_2)^2 M$ for a constant $c$ depending only on $K$.
	
	\paragraph*{Case $j=4$}
	As in the previous cases we have
	\begin{align*}
	    \kappa_{k,l}^{(4)}(t\delta^2) = \delta^2 \int_0^\infty \sc{(e^{-\delta (\theta_2/\theta_1)b\cdot x}-1)S^*_{\vartheta^1,\delta,x_k}(t+t')K}{e^{t'\theta_1\Delta_{\delta,x_k}}K_{k,l}}_{L^2(\Lambda_{\delta,x_k})}dt'.
	\end{align*}
	Using the Cauchy-Schwarz inequality, Lemma \ref{FeynmanKac}(i) and Lemma \ref{boundS*u} with $K=\Delta^2\tilde{K}$ we get for any $\epsilon>0$, and recalling that $\theta_1\geq 1$,
	\begin{align}
	    \sc{(e^{-\delta (\theta_2/\theta_1)b\cdot x}-1)S^*_{\vartheta^1,\delta,x_k}(t+t')K}{e^{t'\theta_1\Delta_{\delta,x_k}}K_{k,l}}_{L^2(\Lambda_{\delta,x_k})}\nonumber\\
	    \lesssim \delta \theta_2(1\wedge (t+t')^{-2-d/4+\epsilon})\norm{|x|e^{t'\theta_1\Delta_{0}}|K_{k,l}|}_{L^2(\R^d)}.\label{eq:kappa_11}
	\end{align}
    Note that $K_{k,l}\in C^1_c(\R^d)$ such that $|K_{k,l}|\in H^{1,\infty}(\R^d)$ and $\nabla |K_{k,l}|\in L^\infty(\R^d)$ with compact support. Using now \cite[Lemma A.2(ii)]{altmeyer_nonparametric_2020} to the extent that 
    \begin{align}
        x(e^{t'\theta_1\Delta_0}|K_{k,l}|)(x) =  (e^{t'\theta_1\Delta_0}(-2t'\theta_1\nabla|K_{k,l}|+x|K_{k,l}|))(x),\label{eq:kappa_10}
    \end{align}
    we find that the $L^2(\R^d)$-norm in \eqref{eq:kappa_11} is uniformly bounded in $t'>0$. Hence, $|\kappa_{k,l}^{(4)}(t\delta^2)|\lesssim \delta^3 \theta_2(1\wedge t^{-1/2-d/4-\epsilon})$ and changing variables shows for the sum of diagonal terms $\sum_{1\leq k\leq M}\norm{\kappa_{k,k}^{(4)}}^2_{L^2([0,T])}\lesssim \delta^{8} \theta_2^2 M$. Regarding the off-diagonal terms we have similarly for some $\bar{K}\in L^\infty(\R^d)$ having compact support
    \begin{align*}
	    &\left|\sc{(e^{-\delta (\theta_2/\theta_1)b\cdot x}-1)S^*_{\vartheta^1,\delta,x_k}(t+t')K}{e^{t'\theta_1\Delta_{\delta,x_k}}K_{k,l}}_{L^2(\Lambda_{\delta,x_k})}\right|\\
	    &=\left|\sc{K}{S_{\vartheta^1,\delta,x_k}(t+t')(e^{-\delta (\theta_2/\theta_1)b\cdot x}-1)e^{t'\theta_1\Delta_{\delta,x_k}}K_{k,l}}_{L^2(\Lambda_{\delta,x_k})}\right|\\
	    &\lesssim \delta \theta_2 \norm{K}_{L^1(\R^d)} \sup_{y\in\operatorname{supp}K} \left|\left(e^{(t+t')\theta_1\Delta_0}|x|e^{t'\theta_1\Delta_0}|K_{k,l}|\right)(y)\right|\\
	     &\lesssim \delta \theta_2 (1\vee t') \sup_{y\in\operatorname{supp}K} \left|\left(e^{(t+2t')\theta_1\Delta_0}|\bar{K}_{k,l}|\right)(y)\right|\\
	     &\lesssim \delta \theta_2 (1\vee t')t^{-d/2}e^{-c'\frac{|x_k-x_l|^2}{\delta^2 t}},
	\end{align*}
	using \eqref{eq:kappa_3}. Arguing as for \eqref{eq:kappa_8} and \eqref{eq:kappa_9} we then find from combining the last display with \eqref{eq:kappa_11} that $|\kappa^{(4)}_{k,l}(t\delta^2)|\lesssim \delta^3 \theta_2 t^{-1/2-d/4-\epsilon'}e^{-\epsilon' \frac{|x_k-x_l|^2}{\delta^2 t}}$ for some $\epsilon'>0$ and $$\int_0^T \kappa_{k,l}^{(4)}(t)^2dt \lesssim \frac{\delta^{8+d+4\epsilon'}\theta_2^2}{|x_k-x_l|^{4\epsilon'+d}},$$ and so in all, for diagonal and off-diagonal terms,  $\sum_{1\leq k, l\leq M}\norm{\kappa_{k,l}^{(4)}}^2_{L^2([0,T])}\leq c\delta^8 \theta_2^2 M$ for a constant $c$ depending only on $K$. \qed
	
	\begin{lemma}
	\label{lem:sum:inverse:packing}
	    Let $x_1,\dots,x_M$ be $\delta$-separated points in $\R^d$, and let $p>d$. Then we have
	    \begin{align*}
	        \sum_{k=2}^M\frac{1}{|x_1-x_k|^p}\leq C\delta^{-p},
	    \end{align*}
        where $C$ is a constant depending only on $d$ and $p$.
	\end{lemma}
	
	\begin{proof}
	    Since  $x_1,\dots,x_M$ are $\delta$-separated, the Euclidean balls $B(x_k,\delta/2)=\{y\in\R^d:|y-x_k|\leq \delta/2\}$ around the $x_k$ of radius $\delta/2$ are disjoint. Moreover, for $y\in B(x_k,\delta/2)$ and $k>1$, we have
	    \begin{align*}
	        |y-x_1|\leq |y-x_k|+|x_k-x_1|\leq \frac{\delta}{2}+|x_k-x_1|\leq \frac{3}{2}|x_k-x_1|,
	    \end{align*}
	    implying that
	    \begin{align*}
	        \frac{1}{|x_k-x_1|}\leq \frac{3}{2}\frac{1}{|y-x_1|}.
	    \end{align*}
	    We conclude that
	    \begin{align*}
	        \sum_{k=2}^M\frac{1}{|x_1-x_k|^p}&=\sum_{k=2}^M\frac{1}{\operatorname{vol}(B(x_k,\delta/2))}\int_{B(x_k,\delta/2)}\limits\frac{1}{|x_1-x_k|^p}dy\\
	        &\leq \frac{(3/2)^p}{(\delta/2)^d\operatorname{vol}(B(0,1))}\sum_{k=2}^M\int_{B(x_k,\delta/2)}\limits\frac{1}{|y-x_1|^p}dy\\
	        &\leq  \frac{(3/2)^p}{(\delta/2)^d\operatorname{vol}(B(0,1))}\int_{B(x_1,\delta/2)^c}\limits\frac{1}{|y-x_1|^p}dy\\
	        &=\frac{(3/2)^p}{(\delta/2)^d\operatorname{vol}(B(0,1))}\int_{B(0,\delta/2)^c}\limits\frac{1}{|y|^p}dy.
	    \end{align*}
	    Changing to polar coordinates, we arrive at
	    \begin{align*}
	        \sum_{k=2}^M\frac{1}{|x_1-x_k|^p}&\leq \frac{d(3/2)^p}{(\delta/2)^d}\int_{\delta/2}^\infty\frac{1}{t^p}t^{d-1}dt=\frac{d(3/2)^p}{(\delta/2)^p}\int_1^\infty s^{d-p-1}ds.
	    \end{align*}
	    Since $p>d$, the latter integral is finite, and the claim follows. 
	\end{proof}

   \subsection{Proof of Theorem \ref{thm:lower:bound:M>1:add:measurements}}
   Argue as in the proof of Theorem \ref{thm:lower:bound:M>1}, using slight modifications of Lemmas \ref{lem:seriesBound:M>1} and \ref{lem:concrete_lower_bound}. The key additional ingredient is an appropriate extension of Corollary~\ref{cor:upper:bound:RKHS:norm:M>1_Laplace}. For this, let $A_1,\dots,A_p$ be as in Section \ref{sec:setup}. We assume that 
    \begin{align}\label{eq:independence:ass}
        (A_i^*K)_{i=1}^{p}\text{ are linearly independent in }L^2(\Lambda).
    \end{align}
    Define 
    \begin{align*}
        H=\Big(\Big\langle\frac{A_i^*K}{\norm{A_i^*K}_{L^2(\Lambda)}},\frac{A_j^*K}{\norm{A_j^*K}_{L^2(\Lambda)}}\Big\rangle\Big)_{i,j=1}^p,
    \end{align*}
    and let $\lambda_{\min}=\lambda_{\min}(H)$ be the smallest eigenvalue of $H$. By \eqref{eq:independence:ass}, $H$ is non-singular, meaning that $\lambda_{\min}(H)>0$. Finally, let 
    \begin{align*}
        v=\sum_{i=1}^p\frac{\norm{\Delta A_i^*K}_{L^2(\Lambda)}^2}{\norm{A_i^*K}_{L^2(\Lambda)}^2}.
    \end{align*}
    
    \begin{corollary}\label{cor:lower:bound:RKHS:norm:M>1:mult:measurement}
        Let $(H_ {X_\delta},\|\cdot\|_{X_\delta})$ be the RKHS of the measurements $X_\delta$ with differential operator $A_\theta=\Delta$, where $X_\delta(t)=(\sc{X(t)}{K_{1k}},\dots,\sc{X(t)}{K_{pk}})_{k=1}^M$ and $K_{ik}=A_i^*K_{\delta,x_k}/\norm{A_i^*K_{\delta,x_k}}_{L^2(\Lambda)}$. Then we have $H_ {X_\delta}=(H^{p})^M$ and 
          \begin{align*}
              \|((h_{ip})_{i=1}^p)_{k=1}^{M}\|_{X_\delta}^2&\leq \frac{4vp}{\delta^4\lambda^2_{\min}}\sum_{k=1}^M\sum_{i=1}^p\|h_{ik}\|_{L^2([0,T])}^2+\frac{2}{\lambda^2_{\min}}\sum_{k=1}^M\sum_{i=1}^p\|h_{ik}'\|_{L^2([0,T])}^2
          \end{align*}
         for all $((h_{ip})_{i=1}^p)_{k=1}^{M}\in (H^{p})^M$, $\delta^2\leq \sqrt{v}$ and $T\geq 1$. 
    \end{corollary}

    \begin{proof}[Proof of Corollary \ref{cor:lower:bound:RKHS:norm:M>1:mult:measurement}]
        First, let $M=1$. Additionally to $H$, define 
        \begin{align*}
        H_{\Delta}&=\Big(\Big\langle\frac{\Delta A_i^*K}{\norm{A_i^*K}_{L^2(\Lambda)}},\frac{\Delta A_j^*K}{\norm{A_j^*K}_{L^2(\Lambda)}}\Big\rangle\Big)_{i,j=1}^p.
        \end{align*}
        By the Cauchy-Schwarz inequality, we have $\norm{H_\Delta}_{\operatorname{op}}\leq v$.
        Moreover, we have
        $G=H$ and $G_\Delta=\delta^{-4}H_\Delta$. Inserting these bounds into Theorem \ref{thm:upper:bound:RKHS:norm:M>1}, we obtain that 
        \begin{align*}
              \|(h_i)_{i=1}^p\|_{X_K}^2
              &\leq  \Big(\frac{3v}{\delta^4\lambda_{\min}^2}+\frac{1}{\lambda_{\min}}\Big)\sum_{i=1}^p\norm{h_i}^2_{L^2([0,T])}+\frac{2}{\lambda_{\min}}\sum_{i=1}^p\norm{h_i'}^2_{L^2([0,T])}.
          \end{align*}
          Next, let $M>1$. Then $G$ and $G_\Delta$ are block-diagonal with $M$ equal $p\times p$-blocks all of the above form and we get
          \begin{align*}
              \|((h_i)_{i=1}^p)_{k=1}^M\|_{X_K}^2
              &\leq  \frac{4v}{\delta^4\lambda_{\min}^2}\sum_{i=1}^p\sum_{k=1}^M\norm{h_{ik}}^2_{L^2([0,T])}+\frac{2}{\lambda_{\min}^2}\sum_{i=1}^p\sum_{k=1}^M\norm{h_{ik}'}^2_{L^2([0,T])},
          \end{align*}
          where we also used that $\lambda_{\min}\in(0,1]$ and $\delta^2\leq \sqrt{v}$. 
    \end{proof}

        \subsection{Second proof of Theorem~\ref{thm:upper:bound:RKHS:norm:M>1}}\label{app:RKHS:measurements:approximation:argument}
        
        In this Appendix, we prove Theorem~\ref{thm:upper:bound:RKHS:norm:M>1} under the weaker assumption $K_1,\dots,K_M\in \mathcal{D}(A)$. This is achieved by an additional approximation argument.
        Let $X_m(t)=\sum_{j\leq m}Y_j(t)e_j$, $0\leq t\leq T$, be the projection of $X(t)$ onto $V_m=\operatorname{span}\{e_1,\dots,e_m\}$, and taking values in $L^2([0,T];V_m)$. We start with the following consequence of Lemma \ref{lemma:RKHS_norm_OU_process}.
    
    \begin{lemma}\label{lem:RKHS:projected:SPDE}
        For every $m\geq 1$, the RKHS $(H_{X_m},\|\cdot\|_{X_m})$ of $X_m$ satisfies $$H_{X_m}=\Big\{h=\sum_{j=1}^mh_je_j:h_j\in H, 1\leq j\leq m\Big\}\quad\text{and}\quad \|h\|_{X_m}^2=\sum_{j=1}^m\norm{h_j}_{Y_j}^2.$$ 
        Moreover, we have $\|h\|_{X_m}=\|h\|_{X}$ with the latter norm defined in Theorem \ref{thm:RKHS:SPDE}.
    \end{lemma}

    \begin{proof}
        Since $L^2([0,T];V_m)$ is isomorphic to $L^2([0,T])^m$, it suffices to compute the RKHS of the coefficient vector $Y=(Y_1,\dots,Y_m)$. Using that $Y_1,\dots,Y_m$ are independent stationary Ornstein-Uhlenbeck processes, the vector $Y$ is a Gaussian process in $L^2([0,T])^m$ with expectation zero and covariance operator $\bigoplus_{j=1}^m C_{Y_j}$ with $C_{Y_j}:L^2([0,T])\rightarrow L^2([0,T])$ being the covariance operator of $Y_j$. Combining this with \eqref{eq:RKHS:Hilbert:space} and Lemma \ref{lemma:RKHS_norm_OU_process}, we conclude that the RKHS of $Y$ is equal to $H^m$ with norm $\sum_{j=1}^m\norm{h_j}_{Y_j}^2$. Translating this back to $X_m$, the first claim follows. The second one follows from $(\lambda_j,e_j)_{j=1}^\infty$ being an eigensystem of $-A$.
    \end{proof}
    
    \begin{proof}[Proof of Theorem \ref{thm:upper:bound:RKHS:norm:M>1}]
        The first step will be to compute the RKHS $(H_{X_{K,m}},\|\cdot\|_{X_{K,m}})$ of $X_{K,m}=(\sc{X_m}{K_k}_{\mathcal{H}})_{k=1}^M$. To this end define the bounded linear map
        \begin{align*}
            L:L^2([0,T];V_m)\rightarrow L^2([0,T])^M,f\mapsto (\sc{f}{K_k}_{\mathcal{H}})_{k=1}^M.
        \end{align*}
        Combining the fact that $LX_m=X_{K,m}$ in distribution with Proposition 4.1 in \cite{MR3024389} and Lemma \ref{lem:RKHS:projected:SPDE}, we obtain that $H_{X_{K,m}}=L(\{h:h=\sum_{j=1}^mh_je_j:h_j\in H\})$. This implies $H_{X_{K,m}}\subseteq H^M$. To see the reverse inclusion, let $P_m$ be the orthogonal projection of $L^2(\Lambda)$ onto $V_m=\operatorname{span}\{e_1,\dots,e_m\}$, and let $G_m=(\sc{P_mK_k}{P_mK_l}_{\mathcal{H}})_{k,l=1}^M$. Since $(e_j)_{j\geq 1}$ is an orthonormal basis of $\mathcal{H}$, $G_m$ tends (e.g., ~in operator norm) to $G$ as $m\rightarrow \infty$. Since $G$ is non-singular, we deduce that $G_m$ is non-singular for all $m$ large enough (which we assume from now on). Hence, for $(h_k)_{k=1}^M\in H^M$, we have that
        \begin{align}\label{eq:inverse:image:approximate}
        f=\sum_{k,l=1}^M (G_m)^{-1}_{k,l} P_mK_kh_l\in H_{X_m}\quad\text{satisfies}\quad Lf=(h_k)_{k=1}^M,
       \end{align}
       where we also used that $P_mK_k\in V_m$ for all $1\leq k\leq M$.
       Hence, $H^M\subseteq H_{X_{K,m}}$ and therefore $H_{X_{K,m}}=H^M$. Moreover, combining \eqref{eq:inverse:image:approximate} with Proposition 4.1 in \cite{MR3024389} and the fact that $A P_m=P_mA$, we get from Lemma \ref{lem:RKHS:projected:SPDE}
       \begin{align*}
           &\norm{(h_k)_{k=1}^M}_{X_{K,m}}^2\leq \|\sum_{k,l=1}^M (G_m)^{-1}_{k,l} P_mK_kh_l\|_{X}^2\\
           &\leq 3\| P_m\sum_{k,l=1}^M (G_m)^{-1}_{k,l}A K_kh_l\|_{L^2([0,T];\mathcal{H})}^2+\| P_m\sum_{k,l=1}^M (G_m)^{-1}_{k,l}K_kh_l\|_{L^2([0,T];\mathcal{H})}^2\\
           &+2\| P_m\sum_{k,l=1}^M (G_m)^{-1}_{k,l}K_kh_l'\|_{L^2([0,T];\mathcal{H})}^2.
        \end{align*}
        Letting $m$ go to infinity, in which case $(G_m)^{-1}$ converges to $G^{-1}$, and so by definition of $G_{A}$, the last display becomes
        \begin{align*}
            &\limsup_{m\rightarrow\infty}\norm{(h_k)_{k=1}^M}_{X_{K,m}}^2\\
            &\leq 3\int_0^T \sum_{k,l=1}^M(G^{-1}G_{A} G^{-1})_{kl}h_k(t)h_l(t)\,dt+\int_0^T  \sum_{k,l=1}^M(G^{-1})_{kl}h_k(t)h_l(t)\,dt\\
            &+2\int_0^T  \sum_{k,l=1}^M(G^{-1})_{kl}h_k'(t)h_l'(t)\,dt.
        \end{align*}
        Using standard results for the operator norm of symmetric matrices yields thus for $\limsup_{m\rightarrow\infty}\norm{(h_k)_{k=1}^M}_{X_{K,m}}^2$ the upper bound claimed in the statement of the theorem to hold for $\norm{h}_{X_K}^2$. 
        
        Next, we use the above results to compute the RKHS of $X_K=(\sc{X}{K_k}_{\mathcal{H}})_{k=1}^M$. First, let us argue that the RKHS of a single measurement $\sc{X}{K_k}_{\mathcal{H}}$ (as a set) equals $H$. Combining Girsanov's theorem for the Itô process $\sc{X}{K_k}_{\mathcal{H}}$ in \eqref{eq:Ito} with Feldman-Hajek's theorem, the RKHS of $\sc{X}{K_k}_{\mathcal{H}}$ starting in zero is $H_\beta$. Adding an independent Gaussian random variable with variance greater zero, we obtain that in the stationary case $\sc{X}{K_k}_{\mathcal{H}}$ has RKHS $H=H_{\bar \beta}$ (see also the proof of Lemma \ref{lemma:RKHS_norm_OU_process}). Now, consider the case $M>1$. By Proposition 4.1 in \cite{MR3024389}, each coordinate projection maps the RKHS of $X_K$ to the RKHS of a single measurement, thus to $H$ by the first step.
        Hence, we have $H_{X_K}\subseteq H^M$. It remains to show the reverse inclusion $H^M=H_{X_{K,m}}\subseteq H_{X_K}$. To see this, note that $$\sc{X_m}{K_k}_{\mathcal{H}}=\sum_{j=1}^m \sc{K_k}{e_j}_{\mathcal{H}} Y_{j}\quad\text{and}\quad\sc{X}{K_k}_{\mathcal{H}}=\sum_{j=1}^\infty \sc{K_k}{e_j}_{\mathcal{H}} Y_{j},$$ so that $X_K=X_{K,m}+(X_K-X_{K,m})$ can be written as a sum of two independent processes taking values in the Hilbert space $L^2([0,T])^M$. Letting $C_{K}$ and $C_{K,m}$ be the covariance operators of $X_{K}$ and $X_{K,m}$, respectively, this implies that $C_K=C_{K,m}+\tilde C$ with $\tilde C$ self-adjoint and positive. Combining this with the characterisation of the RKHS norm in Proposition 2.6.8 of \cite{gine_mathematical_2016}, we get $\norm{(h_k)_{k=1}^M}^2_{X_{K}}\leq \norm{(h_k)_{k=1}^M}^2_{X_{K,m}}$ and $H_{X_{K,m}}\subseteq H_{X_K}$ for all $m\geq 1$. Finally, inserting the upper bound on $\norm{(h_k)_{k=1}^M}^2_{X_{K,m}}$ derived above, the proof is complete.
        \end{proof}

\paragraph*{Acknowledgements}
We are grateful for the helpful comments from three anonymous
referees. The research of AT and MW has been partially funded by the Deutsche Forschungsgemeinschaft (DFG)- Project-ID 318763901 - SFB1294. AT further acknowledges financial support of Carlsberg Foundation Young Researcher Fellowship grant CF20-0604. RA gratefully acknowledges support by the European Research Council, ERC grant agreement 647812 (UQMSI).

	\bibliographystyle{imsart-number} 
	\bibliography{refs}

\begin{thebibliography}{62}

\bibitem{alonso2018modeling}
\begin{barticle}[author]
\bauthor{\bsnm{Alonso},~\bfnm{Sergio}\binits{S.}}, \bauthor{\bsnm{Stange},~\bfnm{Maike}\binits{M.}} \AND \bauthor{\bsnm{Beta},~\bfnm{Carsten}\binits{C.}}
(\byear{2018}).
\btitle{Modeling random crawling, membrane deformation and intracellular polarity of motile amoeboid cells}.
\bjournal{PloS one}
\bvolume{13}
\bpages{e0201977}.
\end{barticle}
\endbibitem

\bibitem{altmeyer_parameter_2020}
\begin{barticle}[author]
\bauthor{\bsnm{Altmeyer},~\bfnm{Randolf}\binits{R.}}, \bauthor{\bsnm{Bretschneider},~\bfnm{Till}\binits{T.}}, \bauthor{\bsnm{Janák},~\bfnm{Josef}\binits{J.}} \AND \bauthor{\bsnm{Reiß},~\bfnm{Markus}\binits{M.}}
(\byear{2022}).
\btitle{Parameter {Estimation} in an {SPDE} {Model} for {Cell} {Repolarisation}}.
\bjournal{SIAM/ASA Journal on Uncertainty Quantification}
\bvolume{10}
\bpages{179-199}.
\end{barticle}
\endbibitem

\bibitem{altmeyer_parameterSemi_2020}
\begin{barticle}[author]
\bauthor{\bsnm{Altmeyer},~\bfnm{Randolf}\binits{R.}}, \bauthor{\bsnm{Cialenco},~\bfnm{Igor}\binits{I.}} \AND \bauthor{\bsnm{Pasemann},~\bfnm{Gregor}\binits{G.}}
(\byear{2023}).
\btitle{Parameter estimation for semilinear {SPDEs} from local measurements}.
\bjournal{Bernoulli}
\bvolume{29}
\bpages{2035--2061}.
\end{barticle}
\endbibitem

\bibitem{altmeyer_nonparametric_2020}
\begin{barticle}[author]
\bauthor{\bsnm{Altmeyer},~\bfnm{Randolf}\binits{R.}} \AND \bauthor{\bsnm{Reiß},~\bfnm{Markus}\binits{M.}}
(\byear{2021}).
\btitle{Nonparametric estimation for linear {SPDEs} from local measurements}.
\bjournal{Annals of Applied Probability}
\bvolume{31}
\bpages{1--38}.
\end{barticle}
\endbibitem

\bibitem{amann1995linear}
\begin{bbook}[author]
\bauthor{\bsnm{Amann},~\bfnm{Herbert}\binits{H.}}
(\byear{1995}).
\btitle{Linear and quasilinear parabolic problems}.
\bseries{Volumne I: Abstract Linear Theory}.
\bpublisher{Birkh\"auser}.
\end{bbook}
\endbibitem

\bibitem{aspelmeier_modern_2015}
\begin{barticle}[author]
\bauthor{\bsnm{Aspelmeier},~\bfnm{Timo}\binits{T.}}, \bauthor{\bsnm{Egner},~\bfnm{Alexander}\binits{A.}} \AND \bauthor{\bsnm{Munk},~\bfnm{Axel}\binits{A.}}
(\byear{2015}).
\btitle{Modern statistical challenges in high-resolution fluorescence microscopy}.
\bjournal{Annual Reviews of Statistics and Its Application}
\bvolume{2}
\bpages{163--202}.
\end{barticle}
\endbibitem

\bibitem{backer_extending_2014}
\begin{barticle}[author]
\bauthor{\bsnm{Backer},~\bfnm{Adam~S.}\binits{A.~S.}} \AND \bauthor{\bsnm{Moerner},~\bfnm{William~E.}\binits{W.~E.}}
\btitle{Extending {Single}-{Molecule} {Microscopy} {Using} {Optical} {Fourier} {Processing}}.
\bjournal{The Journal of Physical Chemistry B}
\bvolume{118}
\bpages{8313--8329}.
\bdoi{10.1021/jp501778z}
\end{barticle}
\endbibitem

\bibitem{benth2022weak}
\begin{barticle}[author]
\bauthor{\bsnm{Benth},~\bfnm{Fred~Espen}\binits{F.~E.}}, \bauthor{\bsnm{Schroers},~\bfnm{Dennis}\binits{D.}} \AND \bauthor{\bsnm{Veraart},~\bfnm{Almut~ED}\binits{A.~E.}}
(\byear{2022}).
\btitle{A weak law of large numbers for realised covariation in a Hilbert space setting}.
\bjournal{Stochastic Processes and their Applications}
\bvolume{145}
\bpages{241--268}.
\end{barticle}
\endbibitem

\bibitem{bibinger_volatility_2020}
\begin{barticle}[author]
\bauthor{\bsnm{Bibinger},~\bfnm{Markus}\binits{M.}} \AND \bauthor{\bsnm{Trabs},~\bfnm{Mathias}\binits{M.}}
(\byear{2020}).
\btitle{Volatility estimation for stochastic {PDEs} using high-frequency observations}.
\bjournal{Stochastic Processes and their Applications}
\bvolume{130}
\bpages{3005--3052}.
\bdoi{10.1016/j.spa.2019.09.002}
\end{barticle}
\endbibitem

\bibitem{MR2759829}
\begin{bbook}[author]
\bauthor{\bsnm{Brezis},~\bfnm{Haim}\binits{H.}}
(\byear{2011}).
\btitle{Functional analysis, {S}obolev spaces and partial differential equations}.
\bpublisher{Springer}.
\bmrnumber{2759829}
\end{bbook}
\endbibitem

\bibitem{catania2006}
\begin{barticle}[author]
\bauthor{\bsnm{Catania},~\bfnm{Federico}\binits{F.}}, \bauthor{\bsnm{Massab{\`o}},~\bfnm{Marco}\binits{M.}} \AND \bauthor{\bsnm{Paladino},~\bfnm{Ombretta}\binits{O.}}
(\byear{2006}).
\btitle{Estimation of Transport and Kinetic Parameters Using Analytical Solutions of the {{2D}} Advection-Dispersion-Reaction Model}.
\bjournal{Environmetrics}
\bvolume{17}
\bpages{199--216}.
\end{barticle}
\endbibitem

\bibitem{chong_high-frequency_2020}
\begin{barticle}[author]
\bauthor{\bsnm{Chong},~\bfnm{Carsten}\binits{C.}}
(\byear{2020}).
\btitle{High-frequency analysis of parabolic stochastic {PDEs}}.
\bjournal{Annals of Statistics}
\bvolume{48}
\bpages{1143--1167}.
\end{barticle}
\endbibitem

\bibitem{cialenco_drift_2019}
\begin{barticle}[author]
\bauthor{\bsnm{Cialenco},~\bfnm{Igor}\binits{I.}}, \bauthor{\bsnm{Delgado-Vences},~\bfnm{Francisco}\binits{F.}} \AND \bauthor{\bsnm{Kim},~\bfnm{Hyun-Jung}\binits{H.-J.}}
(\byear{2020}).
\btitle{Drift {estimation} for {discretely} {sampled} {SPDEs}}.
\bjournal{Stochastics and Partial Differential Equations: Analysis and Computations}
\bvolume{8}
\bpages{895--920}.
\end{barticle}
\endbibitem

\bibitem{cialenco2021statistical}
\begin{barticle}[author]
\bauthor{\bsnm{Cialenco},~\bfnm{Igor}\binits{I.}}, \bauthor{\bsnm{Kim},~\bfnm{Hyun-Jung}\binits{H.-J.}} \AND \bauthor{\bsnm{Pasemann},~\bfnm{Gregor}\binits{G.}}
(\byear{2023}).
\btitle{Statistical analysis of discretely sampled semilinear {SPDEs}: a power variation approach}.
\bjournal{Stochastics and Partial Differential Equations: Analysis and Computations}.
\end{barticle}
\endbibitem

\bibitem{da_prato_stochastic_2014}
\begin{bbook}[author]
\bauthor{\bsnm{Da~Prato},~\bfnm{Giuseppe}\binits{G.}} \AND \bauthor{\bsnm{Zabczyk},~\bfnm{Jerzy}\binits{J.}}
(\byear{2014}).
\btitle{Stochastic equations in infinite dimensions}.
\bpublisher{Cambridge University Press}.
\end{bbook}
\endbibitem

\bibitem{debussche_regularity_2015}
\begin{barticle}[author]
\bauthor{\bsnm{Debussche},~\bfnm{Arnaud}\binits{A.}}, \bauthor{\bparticle{de} \bsnm{Moor},~\bfnm{Sylvain}\binits{S.}} \AND \bauthor{\bsnm{Hofmanová},~\bfnm{Martina}\binits{M.}}
(\byear{2015}).
\btitle{A {Regularity} {Result} for {Quasilinear} {Stochastic} {Partial} {Differential} {Equations} of {Parabolic} {Type}}.
\bjournal{SIAM Journal on Mathematical Analysis}
\bvolume{47}
\bpages{1590--1614}.
\bdoi{10.1137/130950549}
\end{barticle}
\endbibitem

\bibitem{egner2020}
\begin{bincollection}[author]
\bauthor{\bsnm{Egner},~\bfnm{Alexander}\binits{A.}}, \bauthor{\bsnm{Geisler},~\bfnm{Claudia}\binits{C.}} \AND \bauthor{\bsnm{Siegmund},~\bfnm{Ren{\'e}}\binits{R.}}
(\byear{2020}).
\btitle{{{STED Nanoscopy}}}.
In \bbooktitle{Nanoscale {{Photonic Imaging}}},
(\beditor{\bfnm{Tim}\binits{T.}~\bsnm{Salditt}}, \beditor{\bfnm{Alexander}\binits{A.}~\bsnm{Egner}} \AND \beditor{\bfnm{D.~Russell}\binits{D.~R.}~\bsnm{Luke}}, eds.).
\bseries{Topics in {{Applied Physics}}}
\bpages{3--34}.
\bpublisher{{Springer International Publishing}}, \baddress{{Cham}}.
\end{bincollection}
\endbibitem

\bibitem{EngNag00}
\begin{bbook}[author]
\bauthor{\bsnm{Engel},~\bfnm{Klaus-Jochen}\binits{K.-J.}} \AND \bauthor{\bsnm{Nagel},~\bfnm{Rainer}\binits{R.}}
(\byear{2000}).
\btitle{One-Parameter Semigroups for Linear Evolution Equations}.
\bpublisher{Springer}.
\end{bbook}
\endbibitem

\bibitem{evans_partial_2010}
\begin{bbook}[author]
\bauthor{\bsnm{Evans},~\bfnm{Lawrence~C.}\binits{L.~C.}}
(\byear{2010}).
\btitle{Partial {Differential} {Equations}}.
\bpublisher{American Mathematical Soc.}
\end{bbook}
\endbibitem

\bibitem{MR3409135}
\begin{bbook}[author]
\bauthor{\bsnm{Evans},~\bfnm{Lawrence~C.}\binits{L.~C.}} \AND \bauthor{\bsnm{Gariepy},~\bfnm{Ronald~F.}\binits{R.~F.}}
(\byear{2015}).
\btitle{Measure theory and fine properties of functions},
\bedition{Revised} ed.
\bpublisher{CRC Press, Boca Raton, FL}.
\bmrnumber{3409135}
\end{bbook}
\endbibitem

\bibitem{gaudlitz2022estimation}
\begin{barticle}[author]
\bauthor{\bsnm{Gaudlitz},~\bfnm{Sascha}\binits{S.}} \AND \bauthor{\bsnm{Rei{\ss}},~\bfnm{Markus}\binits{M.}}
(\byear{2023}).
\btitle{Estimation for the reaction term in semi-linear {SPDEs} under small diffusivity}.
\bjournal{Bernoulli}
\bvolume{29}
\bpages{3033--3058}.
\end{barticle}
\endbibitem

\bibitem{gine_mathematical_2016}
\begin{bbook}[author]
\bauthor{\bsnm{Giné},~\bfnm{Evarist}\binits{E.}} \AND \bauthor{\bsnm{Nickl},~\bfnm{Richard}\binits{R.}}
(\byear{2016}).
\btitle{Mathematical foundations of infinite-dimensional statistical models}.
\bpublisher{Cambridge University Press}.
\bdoi{10.1017/CBO9781107337862}
\end{bbook}
\endbibitem

\bibitem{gugushvili2020bayesian}
\begin{barticle}[author]
\bauthor{\bsnm{Gugushvili},~\bfnm{Shota}\binits{S.}}, \bauthor{\bsnm{Van Der~Vaart},~\bfnm{Aad}\binits{A.}} \AND \bauthor{\bsnm{Yan},~\bfnm{Dong}\binits{D.}}
(\byear{2020}).
\btitle{Bayesian linear inverse problems in regularity scales}.
\bjournal{Annales de l'Institut Henri Poincar{\'e}, Probabilit{\'e}s et Statistiques}
\bvolume{56}
\bpages{2081--2107}.
\end{barticle}
\endbibitem

\bibitem{hairer_introduction_2009}
\begin{barticle}[author]
\bauthor{\bsnm{Hairer},~\bfnm{Martin}\binits{M.}}
(\byear{2009}).
\btitle{An {Introduction} to {Stochastic} {PDEs}}.
\bjournal{arXiv:0907.4178}.
\end{barticle}
\endbibitem

\bibitem{hildebrandt_parameter_2019}
\begin{barticle}[author]
\bauthor{\bsnm{Hildebrandt},~\bfnm{Florian}\binits{F.}} \AND \bauthor{\bsnm{Trabs},~\bfnm{Mathias}\binits{M.}}
(\byear{2021}).
\btitle{Parameter estimation for {SPDEs} based on discrete observations in time and space.}
\bjournal{Electronic Journal of Statistics}
\bvolume{15}
\bpages{2716--2776}.
\end{barticle}
\endbibitem

\bibitem{hildebrandt2021nonparametric}
\begin{barticle}[author]
\bauthor{\bsnm{Hildebrandt},~\bfnm{Florian}\binits{F.}} \AND \bauthor{\bsnm{Trabs},~\bfnm{Mathias}\binits{M.}}
(\byear{2023}).
\btitle{Nonparametric calibration for stochastic reaction-diffusion equations based on discrete observations}.
\bjournal{Stochastic Processes and their Applications}
\bvolume{162}
\bpages{171--217}.
\end{barticle}
\endbibitem

\bibitem{huebner_asymptotic_1995}
\begin{barticle}[author]
\bauthor{\bsnm{Huebner},~\bfnm{Marianne}\binits{M.}} \AND \bauthor{\bsnm{Rozovskii},~\bfnm{Boris}\binits{B.}}
(\byear{1995}).
\btitle{On asymptotic properties of maximum likelihood estimators for parabolic stochastic {PDE}'s}.
\bjournal{Probability Theory and Related Fields}
\bvolume{103}
\bpages{143--163}.
\bdoi{10.1007/BF01204212}
\end{barticle}
\endbibitem

\bibitem{janson_gaussian_1997}
\begin{bbook}[author]
\bauthor{\bsnm{Janson},~\bfnm{Svante}\binits{S.}}
(\byear{1997}).
\btitle{Gaussian {Hilbert} {Spaces}}.
\bpublisher{Cambridge University Press}.
\end{bbook}
\endbibitem

\bibitem{karalashvili2011}
\begin{barticle}[author]
\bauthor{\bsnm{Karalashvili},~\bfnm{Maka}\binits{M.}}, \bauthor{\bsnm{Gro{\ss}},~\bfnm{Sven}\binits{S.}}, \bauthor{\bsnm{Marquardt},~\bfnm{Wolfgang}\binits{W.}}, \bauthor{\bsnm{Mhamdi},~\bfnm{Adel}\binits{A.}} \AND \bauthor{\bsnm{Reusken},~\bfnm{Arnold}\binits{A.}}
(\byear{2011}).
\btitle{Identification of Transport Coefficient Models in Convection-Diffusion Equations}.
\bjournal{SIAM Journal on Scientific Computing}
\bvolume{33}
\bpages{303--327}.
\end{barticle}
\endbibitem

\bibitem{karatzas_brownian_1998}
\begin{bbook}[author]
\bauthor{\bsnm{Karatzas},~\bfnm{Ioannis}\binits{I.}} \AND \bauthor{\bsnm{Shreve},~\bfnm{Steven}\binits{S.}}
(\byear{1998}).
\btitle{Brownian {Motion} and {Stochastic} {Calculus}}.
\bpublisher{Springer}.
\end{bbook}
\endbibitem

\bibitem{kato2013perturbation}
\begin{bbook}[author]
\bauthor{\bsnm{Kato},~\bfnm{Tosio}\binits{T.}}
(\byear{1995}).
\btitle{Perturbation theory for linear operators}.
\bpublisher{Springer}.
\end{bbook}
\endbibitem

\bibitem{kuchler2006exponential}
\begin{bbook}[author]
\bauthor{\bsnm{K{\"u}chler},~\bfnm{Uwe}\binits{U.}} \AND \bauthor{\bsnm{S{\o}rensen},~\bfnm{Michael}\binits{M.}}
(\byear{1997}).
\btitle{Exponential families of stochastic processes}.
\bpublisher{Springer}.
\end{bbook}
\endbibitem

\bibitem{KukushAlexander2020GMiH}
\begin{bbook}[author]
\bauthor{\bsnm{Kukush},~\bfnm{Alexander}\binits{A.}}
(\byear{2020}).
\btitle{Gaussian Measures in Hilbert Space: Construction and Properties}.
\bpublisher{Wiley}.
\end{bbook}
\endbibitem

\bibitem{kulaitis_what_2020}
\begin{barticle}[author]
\bauthor{\bsnm{Kulaitis},~\bfnm{Gytis}\binits{G.}}, \bauthor{\bsnm{Munk},~\bfnm{Axel}\binits{A.}} \AND \bauthor{\bsnm{Werner},~\bfnm{Frank}\binits{F.}}
(\byear{2021}).
\btitle{What is resolution? {A} statistical minimax testing perspective on super-resolution microscopy}.
\bjournal{Annals of Statistics}
\bvolume{49}
\bpages{2292--2312}.
\end{barticle}
\endbibitem

\bibitem{kutoyants_statistical_2013}
\begin{bbook}[author]
\bauthor{\bsnm{Kutoyants},~\bfnm{Yuri~A.}\binits{Y.~A.}}
(\byear{2004}).
\btitle{Statistical {Inference} for {Ergodic} {Diffusion} {Processes}}.
\bpublisher{Springer}.
\end{bbook}
\endbibitem

\bibitem{kriz_central_2018}
\begin{barticle}[author]
\bauthor{\bsnm{K{\v{r}}í{\v{z}}},~\bfnm{Pavel}\binits{P.}} \AND \bauthor{\bsnm{Maslowski},~\bfnm{Bohdan}\binits{B.}}
(\byear{2019}).
\btitle{Central {Limit} {Theorems} and {Minimum}-{Contrast} {Estimators} for {Linear} {Stochastic} {Evolution} {Equations}}.
\bjournal{Stochastics}
\bvolume{91}
\bpages{1109--1140}.
\bdoi{10.1080/17442508.2019.1576688}
\end{barticle}
\endbibitem

\bibitem{kwasnicki2017ten}
\begin{barticle}[author]
\bauthor{\bsnm{Kwa{\'s}nicki},~\bfnm{Mateusz}\binits{M.}}
(\byear{2017}).
\btitle{Ten equivalent definitions of the fractional Laplace operator}.
\bjournal{Fractional Calculus and Applied Analysis}
\bvolume{20}
\bpages{7--51}.
\end{barticle}
\endbibitem

\bibitem{MR3024389}
\begin{bbook}[author]
\bauthor{\bsnm{Lifshits},~\bfnm{Mikhail}\binits{M.}}
(\byear{2012}).
\btitle{Lectures on {G}aussian processes}.
\bpublisher{Springer}.
\end{bbook}
\endbibitem

\bibitem{liptser_statistics_2001}
\begin{bbook}[author]
\bauthor{\bsnm{Liptser},~\bfnm{Robert}\binits{R.}} \AND \bauthor{\bsnm{Shiryaev},~\bfnm{Albert}\binits{A.}}
(\byear{2001}).
\btitle{Statistics of {Random} {Processes} {I}}.
\bseries{General {Theory}}.
\bpublisher{Springer}.
\end{bbook}
\endbibitem

\bibitem{liu_stochastic_2015}
\begin{bbook}[author]
\bauthor{\bsnm{Liu},~\bfnm{W}\binits{W.}} \AND \bauthor{\bsnm{Röckner},~\bfnm{M}\binits{M.}}
(\byear{2015}).
\btitle{Stochastic {Partial} {Differential} {Equations}: {An} {Introduction}}.
\bpublisher{Springer}.
\end{bbook}
\endbibitem

\bibitem{liu_statistical_2021}
\begin{barticle}[author]
\bauthor{\bsnm{Liu},~\bfnm{Xiao}\binits{X.}}, \bauthor{\bsnm{Yeo},~\bfnm{Kyongmin}\binits{K.}} \AND \bauthor{\bsnm{Lu},~\bfnm{Siyuan}\binits{S.}}
(\byear{2022}).
\btitle{Statistical {Modeling} for {Spatio}-{Temporal} {Data} {From} {Stochastic} {Convection}-{Diffusion} {Processes}}.
\bjournal{Journal of the American Statistical Association}
\bvolume{117}
\bpages{1482--1499}.
\bdoi{10.1080/01621459.2020.1863223}
\end{barticle}
\endbibitem

\bibitem{llopis2018particle}
\begin{barticle}[author]
\bauthor{\bsnm{Llopis},~\bfnm{Francesc~Pons}\binits{F.~P.}}, \bauthor{\bsnm{Kantas},~\bfnm{Nikolas}\binits{N.}}, \bauthor{\bsnm{Beskos},~\bfnm{Alexandros}\binits{A.}} \AND \bauthor{\bsnm{Jasra},~\bfnm{Ajay}\binits{A.}}
(\byear{2018}).
\btitle{Particle filtering for stochastic {Navier-Stokes} signal observed with linear additive noise}.
\bjournal{SIAM Journal on Scientific Computing}
\bvolume{40}
\bpages{A1544--A1565}.
\end{barticle}
\endbibitem

\bibitem{lockley_image-based_2017}
\begin{bphdthesis}[author]
\bauthor{\bsnm{Lockley},~\bfnm{Robert}\binits{R.}}
(\byear{2017}).
\btitle{Image-based {Modelling} of {Cell} {Reorientation}},
\btype{{PhD} {Thesis}},
\bpublisher{University of Warwick}.
\end{bphdthesis}
\endbibitem

\bibitem{lototsky_parameter_2003}
\begin{barticle}[author]
\bauthor{\bsnm{Lototsky},~\bfnm{Sergey~V.}\binits{S.~V.}}
(\byear{2003}).
\btitle{Parameter {Estimation} for {Stochastic} {Parabolic} {Equations}: {Asymptotic} {Properties} of a {Two}-{Dimensional} {Projection}-{Based} {Estimator}}.
\bjournal{Statistical Inference for Stochastic Processes}
\bvolume{6}
\bpages{65--87}.
\bdoi{10.1023/A:1022699622088}
\end{barticle}
\endbibitem

\bibitem{luce2013}
\begin{barticle}[author]
\bauthor{\bsnm{Luce},~\bfnm{Charles~H.}\binits{C.~H.}}, \bauthor{\bsnm{Tonina},~\bfnm{Daniele}\binits{D.}}, \bauthor{\bsnm{Gariglio},~\bfnm{Frank}\binits{F.}} \AND \bauthor{\bsnm{Applebee},~\bfnm{Ralph}\binits{R.}}
(\byear{2013}).
\btitle{Solutions for the Diurnally Forced Advection-Diffusion Equation to Estimate Bulk Fluid Velocity and Diffusivity in Streambeds from Temperature Time Series}.
\bjournal{Water Resources Research}
\bvolume{49}
\bpages{488--506}.
\end{barticle}
\endbibitem

\bibitem{lunardi_analytic_1995}
\begin{bbook}[author]
\bauthor{\bsnm{Lunardi},~\bfnm{Alessandra}\binits{A.}}
(\byear{1995}).
\btitle{Analytic {Semigroups} and {Optimal} {Regularity} in {Parabolic} {Problems}}.
\bseries{Modern {Birkhäuser} {Classics}}.
\bpublisher{Birkhäuser}.
\end{bbook}
\endbibitem

\bibitem{pazy_semigroups_1983}
\begin{bbook}[author]
\bauthor{\bsnm{Pazy},~\bfnm{Amnon}\binits{A.}}
(\byear{1983}).
\btitle{Semigroups of {Linear} {Operators} and {Applications} to {Partial} {Differential} {Equations}}.
\bpublisher{Springer}.
\end{bbook}
\endbibitem

\bibitem{reddy1994pseudospectra}
\begin{barticle}[author]
\bauthor{\bsnm{Reddy},~\bfnm{Satish~C}\binits{S.~C.}} \AND \bauthor{\bsnm{Trefethen},~\bfnm{Lloyd~N}\binits{L.~N.}}
(\byear{1994}).
\btitle{Pseudospectra of the convection-diffusion operator}.
\bjournal{SIAM Journal on Applied Mathematics}
\bvolume{54}
\bpages{1634--1649}.
\end{barticle}
\endbibitem

\bibitem{reis_asymptotic_2011}
\begin{barticle}[author]
\bauthor{\bsnm{Reiß},~\bfnm{Markus}\binits{M.}}
(\byear{2011}).
\btitle{Asymptotic equivalence for inference on the volatility from noisy observations}.
\bjournal{Annals of Statistics}
\bvolume{39}
\bpages{772--802}.
\bdoi{10.1214/10-AOS855}
\end{barticle}
\endbibitem

\bibitem{sauer_analysis_2016}
\begin{barticle}[author]
\bauthor{\bsnm{Sauer},~\bfnm{Martin}\binits{M.}} \AND \bauthor{\bsnm{Stannat},~\bfnm{Wilhelm}\binits{W.}}
(\byear{2016}).
\btitle{Analysis and approximation of stochastic nerve axon equations}.
\bjournal{Mathematics of Computation}
\bvolume{85}
\bpages{2457--2481}.
\bdoi{10.1090/mcom/3068}
\end{barticle}
\endbibitem

\bibitem{sheu_estimates_1991}
\begin{barticle}[author]
\bauthor{\bsnm{Sheu},~\bfnm{Shuenn-Jyi}\binits{S.-J.}}
(\byear{1991}).
\btitle{Some estimates of the transition density of a nondegenerate diffusion markov process}.
\bjournal{The Annals of Probability}
\bvolume{19}
\bpages{538--561}.
\bdoi{10.2307/2244362}
\end{barticle}
\endbibitem

\bibitem{shimakura_partial_1992}
\begin{bbook}[author]
\bauthor{\bsnm{Shimakura},~\bfnm{Norio}\binits{N.}}
(\byear{1992}).
\btitle{Partial {Differential} {Operators} of {Elliptic} {Type}}.
\bpublisher{American Mathematical Soc.}
\end{bbook}
\endbibitem

\bibitem{sigrist_stochastic_2015}
\begin{barticle}[author]
\bauthor{\bsnm{Sigrist},~\bfnm{Fabio}\binits{F.}}, \bauthor{\bsnm{Künsch},~\bfnm{Hans~R}\binits{H.~R.}} \AND \bauthor{\bsnm{Stahel},~\bfnm{Werner~A}\binits{W.~A.}}
(\byear{2015}).
\btitle{Stochastic partial differential equation based modelling of large space–time data sets}.
\bjournal{Journal of the Royal Statistical Society: Series B (Statistical Methodology)}
\bvolume{77}
\bpages{3--33}.
\bdoi{10.1111/rssb.12061}
\end{barticle}
\endbibitem

\bibitem{tonaki2022parameter}
\begin{barticle}[author]
\bauthor{\bsnm{Tonaki},~\bfnm{Yozo}\binits{Y.}}, \bauthor{\bsnm{Kaino},~\bfnm{Yusuke}\binits{Y.}} \AND \bauthor{\bsnm{Uchida},~\bfnm{Masayuki}\binits{M.}}
(\byear{2023}).
\btitle{Parameter estimation for linear parabolic {SPDEs} in two space dimensions based on high frequency data}.
\bjournal{Scandinavian Journal of Statistics}.
\end{barticle}
\endbibitem

\bibitem{Triebel1983book}
\begin{bbook}[author]
\bauthor{\bsnm{Triebel},~\bfnm{H.}\binits{H.}}
(\byear{1983}).
\btitle{Theory of function spaces}.
\bpublisher{Birkh\"auser}.
\end{bbook}
\endbibitem

\bibitem{tsybakov_introduction_2008}
\begin{bbook}[author]
\bauthor{\bsnm{Tsybakov},~\bfnm{Alexandre~B.}\binits{A.~B.}}
(\byear{2009}).
\btitle{Introduction to {Nonparametric} {Estimation}}.
\bpublisher{Springer}.
\end{bbook}
\endbibitem

\bibitem{Tuckwell2013}
\begin{bincollection}[author]
\bauthor{\bsnm{Tuckwell},~\bfnm{Henry~C.}\binits{H.~C.}}
(\byear{2013}).
\btitle{Stochastic partial differential equations in neurobiology: {Linear} and nonlinear models for spiking neurons}.
In \bbooktitle{Stochastic biomathematical models: with applications to neuronal modeling}
(\beditor{\bfnm{Mostafa}\binits{M.}~\bsnm{Bachar}}, \beditor{\bfnm{Jerry}\binits{J.}~\bsnm{Batzel}} \AND \beditor{\bfnm{Susanne}\binits{S.}~\bsnm{Ditlevsen}}, eds.)
\bpages{149--173}.
\bpublisher{Springer}.
\end{bincollection}
\endbibitem

\bibitem{van_der_vaart_rates_2008}
\begin{barticle}[author]
\bauthor{\bparticle{van~der} \bsnm{Vaart},~\bfnm{Aad}\binits{A.}} \AND \bauthor{\bparticle{van} \bsnm{Zanten},~\bfnm{Harry}\binits{H.}}
(\byear{2008}).
\btitle{Rates of contraction of posterior distributions based on {Gaussian} process priors}.
\bjournal{The Annals of Statistics}
\bvolume{36}
\bpages{1435--1463}.
\bnote{Publisher: Institute of Mathematical Statistics}.
\bdoi{10.1214/009053607000000613}
\end{barticle}
\endbibitem

\bibitem{MR3967104}
\begin{bbook}[author]
\bauthor{\bsnm{Wainwright},~\bfnm{Martin~J.}\binits{M.~J.}}
(\byear{2019}).
\btitle{High-dimensional statistics}.
\bpublisher{Cambridge University Press}.
\end{bbook}
\endbibitem

\bibitem{walsh_stochastic_1981}
\begin{barticle}[author]
\bauthor{\bsnm{Walsh},~\bfnm{John~B}\binits{J.~B.}}
(\byear{1981}).
\btitle{A stochastic model of neural response}.
\bjournal{Advances in Applied Probability}
\bvolume{13}
\bpages{231--281}.
\bdoi{10.2307/1426683}
\end{barticle}
\endbibitem

\bibitem{whitt_martingale_FCLT_2007}
\begin{barticle}[author]
\bauthor{\bsnm{Whitt},~\bfnm{Ward}\binits{W.}}
(\byear{2007}).
\btitle{Proofs of the martingale {FCLT}}.
\bjournal{Probability Surveys}
\bvolume{4}
\bpages{268-302}.
\end{barticle}
\endbibitem

\bibitem{yagi_abstract_2009}
\begin{bbook}[author]
\bauthor{\bsnm{Yagi},~\bfnm{Atsushi}\binits{A.}}
(\byear{2010}).
\btitle{Abstract {Parabolic} {Evolution} {Equations} and their {Applications}}.
\bpublisher{Springer}.
\end{bbook}
\endbibitem

\end{thebibliography}

\end{document}